\documentclass[a4paper,10pt]{article}

\usepackage{amssymb}
\usepackage{graphicx}
\usepackage{bbm}
\usepackage{mathrsfs}
\usepackage{amsmath}
\usepackage{amsthm}
\usepackage[center]{caption}
\usepackage{enumerate}
\usepackage{verbatim}
\usepackage{xcolor}
\usepackage{authblk}
\usepackage[hmargin=3cm, vmargin=3.5cm]{geometry}
\usepackage{hyperref}
\usepackage{enumerate}
\usepackage[cp1251]{inputenc}
\usepackage{tikz}
\usetikzlibrary{arrows}

\theoremstyle{plain}
\newtheorem{thm}{Theorem}[section]
\newtheorem{cor}[thm]{Corollary}
\newtheorem{lem}[thm]{Lemma}
\newtheorem{prop}[thm]{Proposition}

\theoremstyle{definition}

\newcommand{\Q}{\mathbb{Q}}
\renewcommand{\P}{\mathbb{P}}
\newcommand{\E}{\mathbb{E}}
\newcommand{\R}{\mathbb{R}}
\newcommand{\N}{\mathbb{N}}

\newcommand{\F}{\mathcal{F}}
\newcommand{\G}{\mathcal{G}}

\newcommand{\ind}{\mathbbm{1}}

\newcommand{\eps}{\varepsilon}
\newcommand{\bp}{\begin{proof}}
\newcommand{\ep}{\end{proof}}
\def\bal#1\eal{\begin{align*}#1\end{align*}}
\DeclareMathOperator{\PL}{PL}
\renewcommand{\d}{\,\textrm{d}}

\newcommand{\Nc}{\mathcal{N}}
\newcommand{\sfrac}[2]{{\textstyle{\frac{#1}{#2}}}}
\newcommand{\eb}{\mathbbm{e}}

\DeclareMathOperator{\gen}{gen}

\newcounter{enumeratecounter}

\author{Alice Callegaro and Matthew I.~Roberts}

\title{A spatially-dependent fragmentation process}

\begin{document}
\maketitle

\begin{abstract}
We define a spatially-dependent fragmentation process, which involves rectangles breaking up into progressively smaller pieces at rates that depend on their shape. Long, thin rectangles are more likely to break quickly, and are also more likely to split along their longest side. We are interested in how the system evolves over time: how many fragments are there of different shapes and sizes, and how did they reach that state? Our theorem gives an almost sure growth rate along paths, which does not match the growth rate in expectation---there are paths where the expected number of fragments of that shape and size is exponentially large, but in reality no such fragments exist at large times almost surely.
\end{abstract}

\section{Introduction}\label{intro_sec}

\subsection{Fragmenting rectangles}

A fragmentation process describes the breaking up of a structure into pieces, and occurs naturally in many situations. Mathematically, fragmentation processes have been a subject of active research in probability for at least 20 years, incorporating several varieties, including homogeneous fragmentations \cite{bertoin:homogeneous_frags}, self-similar fragmentations \cite{bertoin:self_sim_frags}, and growth fragmentations \cite{bertoin:growth-frags}. The textbook of Bertoin \cite{bertoin:fragmentation_coagulation} gives an excellent introduction to this rich mathematical theory. It begins by listing some real-world examples of phenomena that might be considered fragmentation processes, including ``stellar fragments in astrophysics, fractures and earthquakes in geophysics, breaking of crystals in crystallography, degradation of large polymer chains in chemistry, DNA fragmentation in biology, fission of atoms in nuclear physics, fragmentation of a hard drive in computer science,'' and particularly valid from a mathematical point of view, ``evolution of blocks of mineral in a crusher.''

However, the traditional mathematical definition of a fragmentation process insists that (again quoting Bertoin \cite{bertoin:fragmentation_coagulation}) ``each fragment can be characterized by a real number that should be thought of as its size. This stops us from considering the spatial position of a fragment or further geometrical properties like its shape; physicists call such models mean field.'' In \cite{bertoin:multitypefrag}, Bertoin does analyse a multitype model where fragments can take finitely many types, but in applications there is often a continuum of possible shapes.

We consider a spatially-dependent fragmentation process as follows. Begin with a square of side length $1$. After a random time, the square breaks into two rectangular pieces, uniformly at random. Each of these pieces then repeats this behaviour independently, except that long, thin rectangles break more quickly, and are more likely to break along their longest side. This model is designed to mimic a physical crushing process, where long, thin pieces of rock are likely to break more easily than more evenly-proportioned pieces. See Figures \ref{simplefig} and \ref{fig:comparison}.

\begin{figure}[h!] \centering \includegraphics[width=2.3cm]{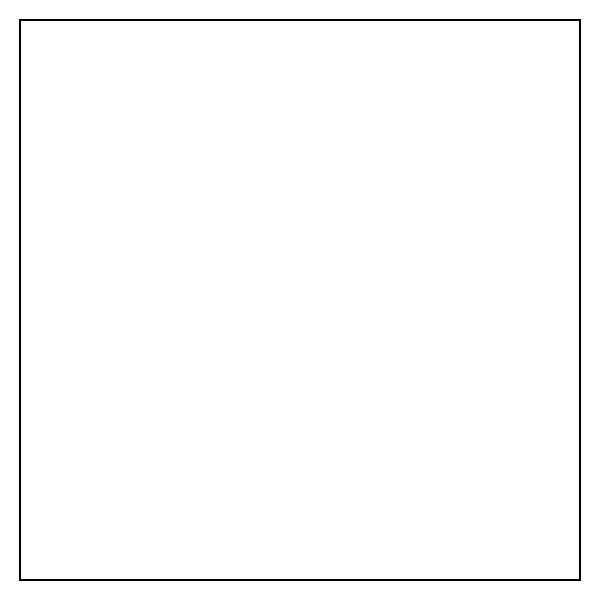} \includegraphics[width=2.3cm]{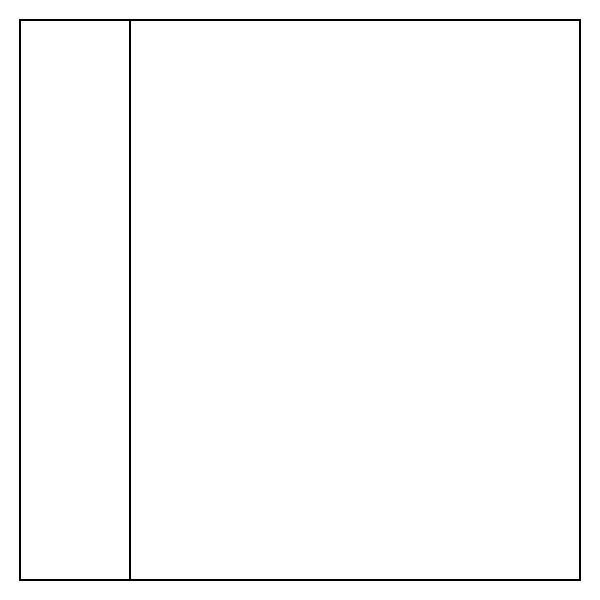} \includegraphics[width=2.3cm]{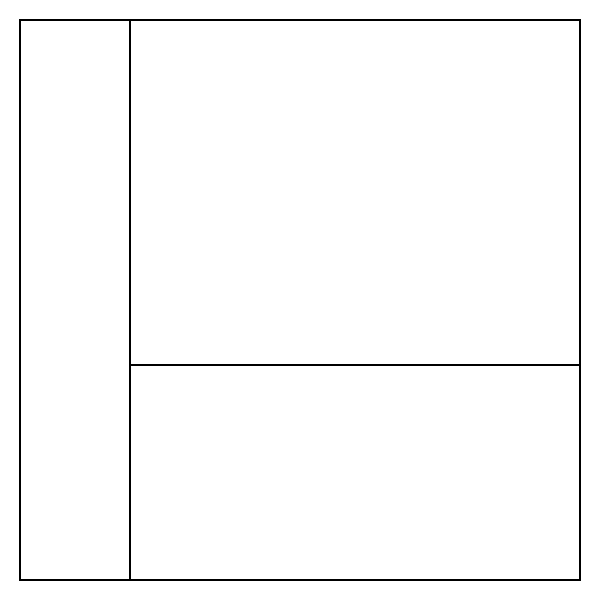} \includegraphics[width=2.3cm]{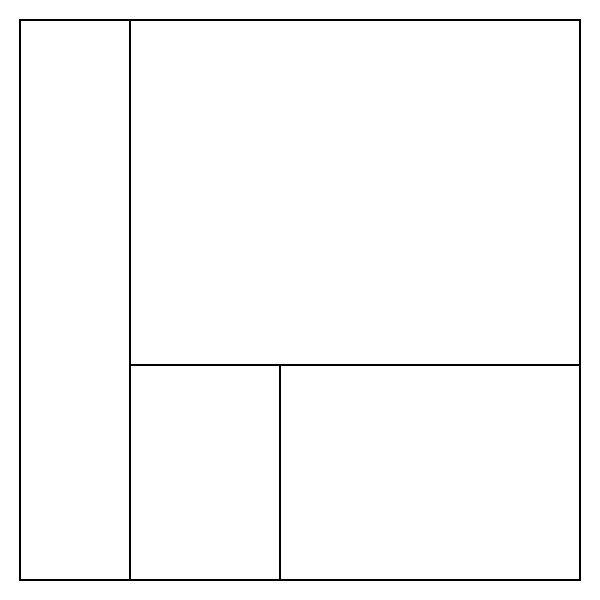} \includegraphics[width=2.3cm]{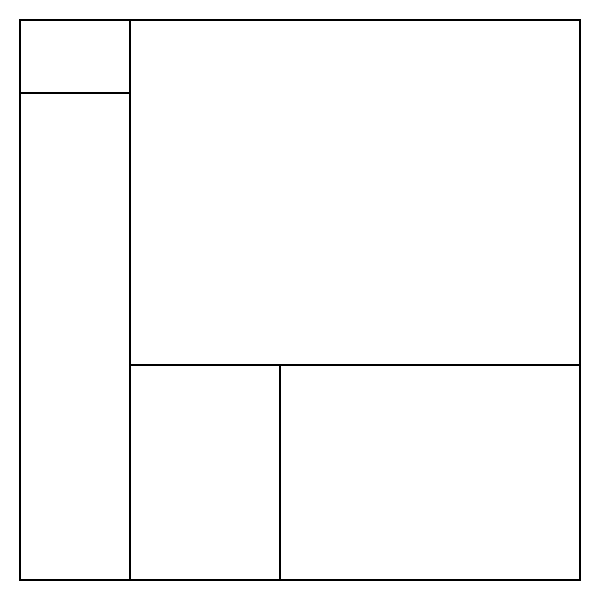} \includegraphics[width=2.3cm]{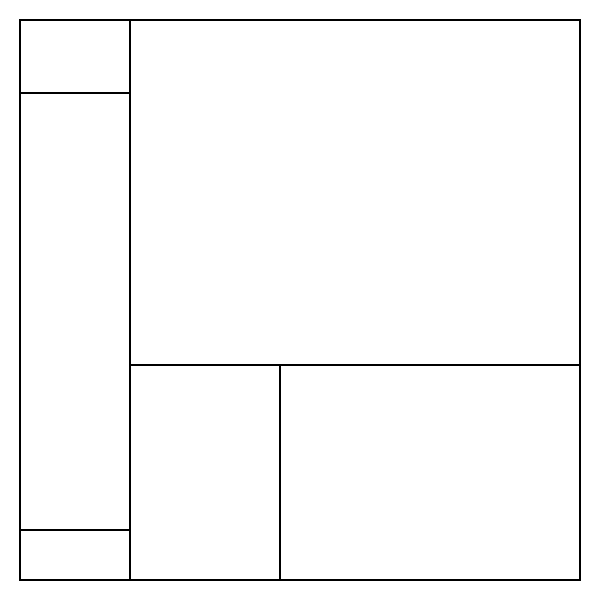} \caption{We begin with a square, which splits vertically into two rectangles. One of these then splits horizontally, and the process continues. Thinner rectangles are more likely to split first.} \label{simplefig} \end{figure}

We work in two dimensions to keep notation manageable, but our proofs should work in three or more dimensions with little additional work. For the same reason, we make a particular choice for the splitting rule---that is, the functions that decide how a fragment's shape affects its branching rate and the direction in which it breaks---but our methods should be adaptable to a variety of spatially-dependent fragmentation models.

Understanding spatially-dependent branching systems is an important problem in its own right, since almost any real-world application of branching tools---from nuclear reactors \cite{harris_et_al:stochastic_NTE_II, horton_et_al:stochastic_NTE_I} to the spread of disease \cite{durrett_foo_leder:spatial_moran, foo_leder_schweinsberg:mutation_timing_spatial_evolution}---involves spatial inhomogeneity. One purpose of this paper is to contribute new techniques to the rigorous mathematical investigation of spatially-dependent branching structures more generally.

A key observation in the study of (mathematical) fragmentation processes is that they satisfy the branching property, in that the future evolution of one fragment, given its current state, does not depend on the other fragments. This enables us to use branching tools in the analysis of fragmentation processes. For example, if we consider the negative logarithm of the sizes of the fragments of a homogeneous fragmentation, then we obtain a continuous-time branching random walk. Bertoin's multitype fragmentation, under the same transformation, becomes a multitype branching random walk. In the same way, our system of fragmenting rectangles can also be thought of as a multitype branching random walk, but one with uncountably many types.

\begin{figure}[htp!] 
\begin{center}
\includegraphics[width=0.46\textwidth]{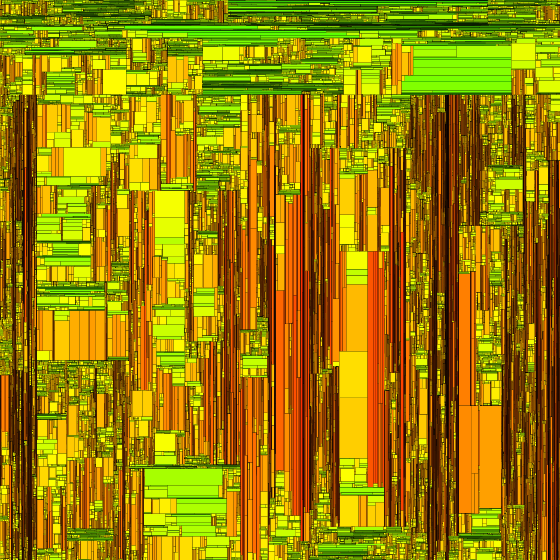} 
\includegraphics[width=0.46\textwidth]{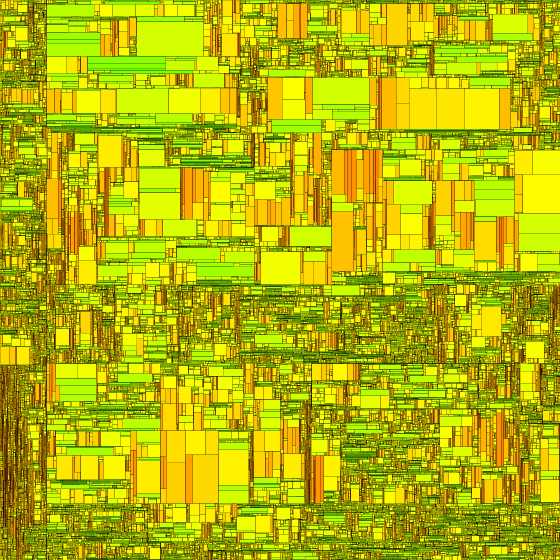}
\end{center}
\vspace{-0.3 cm}
\caption{On the left: a homogeneous model, where every rectangle splits at rate $1$ and splits horizontally or vertically with probability $1/2$ each. On the right: our model where long, thin rectangles split faster, and are more likely to split along their longest side. Tall rectangles are coloured red, fat rectangles are coloured green, and squares are coloured yellow.} \label{fig:comparison}
\end{figure}

Analysing branching systems with uncountably many types is notoriously difficult. Even multitype Galton-Watson processes with countably many types are beyond the scope of standard tools, hence the restriction to finitely many types in most papers on multitype branching systems, including \cite{bertoin:multitypefrag}. Our model includes not just a continuum of types, but a two-dimensional set of possibilities. To add to the difficulty, since we split our rectangles using uniform random variables and then take negative logarithms, and $-\log U$ is exponentially distributed, the jump distribution of our multitype branching random walk does not satisfy the strong Cram\'er condition, and even when rescaled the sample paths include macroscopic jumps. Our analysis is therefore highly technically challenging.

\subsection{The model}

As mentioned above, we work in continuous time, and begin with a square of side-length $1$. At any time, each rectangle of base $b$ and height $h$ independently splits at rate $r(b,h)$ into two smaller rectangles. The probability that it splits vertically is $p(b,h)$, and if so then it splits at a uniform point along its base; otherwise it splits horizontally at a uniform point along its height. The functions $r$ and $p$ are given by
\[r(b,h) = \Big(\frac{1-\log b}{1-\log h}\Big) \vee \Big(\frac{1-\log h}{1-\log b}\Big)\]
and
\[p(b,h) = \frac{1-\log h}{2(1-\log b)} \mathbbm{1}_{b \le h} +  \left( 1-\frac{1-\log b}{2(1-\log h)} \right) \mathbbm{1}_{b > h}.\]
It is easy to see that rectangles with either large base relative to their height, or large height relative to their base, split faster, and are more likely to split along their longer side. The appearance of $1-\log b$ and $1-\log h$, rather than $b$ and $h$, is because splitting events have a multiplicative effect: the distribution of the ratio of each rectangle's measurements to its parent's is invariant. Thus the logarithm of the measurements behaves additively, which ensures that the functions $r$ and $p$ remain non-trivial when we rescale space and time. On the other hand, our choices of $r$ and $p$ are not the only choices with this property, and our methods appear to be fairly robust: it should be possible to adapt them to other sensible splitting rules. We propose this model as a proof of concept that spatial fragmentations (with uncountably many types) can be analysed mathematically.

For a rectangle $v$, we denote its base by $B_v$ and its height by $H_v$. We let $X_v = -\log B_v$ and $Y_v = -\log H_v$. As suggested above, $X_v$ and $Y_v$ are more useful parameterisations of size than $B_v$ and $H_v$ from a mathematical point of view, simply because rectangles' sizes will decay exponentially with time. Under this transformation, our system has the following alternative description.

Begin with one particle at $(0,0)\in\R^2$. Each particle, when at position $(x,y)$ with $x,y \ge 0$, branches at rate
	\begin{equation}\label{eqn:rate2}
	R(x,y)= \frac{x+1}{y+1} \vee \frac{y+1}{x+1}. 
	\end{equation} 
At a branching event, the particle is replaced by two children: letting $\mathcal{U}$ be a uniform random variable on $(0,1)$, independent of everything else, then with probability 
	\begin{equation}\nonumber
	P(x,y) = \frac{y+1}{2(x+1)} \mathbbm{1}_{x \ge y} +  \left( 1-\frac{x+1}{2(y+1)} \right) \mathbbm{1}_{x < y}.
	\end{equation}
the two children have positions $(x - \log \mathcal{U}, y)$ and $(x-\log(1-\mathcal{U}), y)$, and with probability $1-P(x,y)$ they have positions $(x, y - \log \mathcal{U})$ and $(x, y-\log(1-\mathcal{U}))$.

We let $R_X(x,y)=R(x,y)P(x,y)$ and $R_Y(x,y)=R(x,y)(1-P(x,y))$. Then $R_X$ and $R_Y$ denote the rates at which a particle at position $(x,y)$ moves in the first spatial dimension, or the second, respectively.

From now on, we mostly use the second description, and refer to particles and their positions, rather than rectangles and their sizes. As seen above, the two descriptions are entirely equivalent.

\subsection{Main theorem}\label{thm_sec}

Let $E$ be the set of non-decreasing c\`adl\`ag functions $f:[0,1]\to\R$ with $f(0)=0$. Set $f'(s)= \infty$ if $f \in E$ is not differentiable at the point $s\in[0,1]$. By Lebesgue's decomposition theorem, for any function $f\in E$ we may write $f = \tilde f + \hat f$ where $\tilde f$ is absolutely continuous and $\hat f$ is singular.

Our main theorem aims to quantify how many particles have paths which, when rescaled appropriately, fall within a given subset of $E^2$. It is written in the style of a large deviations result, although it is not actually a large deviations result since it is concerned with the almost-sure behaviour of the system rather than events of small probability. 

Since we are interested in rescaled paths for large times, $R$ and $P$ are essentially governed by the ratios $x/y$ and $y/x$. We define the functions $R^*:[0,\infty)^2\to[0,\infty]$ and $P^*:[0,\infty)^2\to[0,1]$ by
\[R^*(x,y) := \begin{cases} \frac{x}{y} \vee \frac{y}{x} &\text{ if } x>0 \text{ or } y>0\\
							1 &\text{ if } x=y=0\end{cases}\]
and
\[P^*(x,y) := \begin{cases} \frac{y}{2 x}\ind_{x \ge y} +  \big(1-\frac{x}{2 y}\big) \ind_{x < y} &\text{ if } x>0 \text{ or } y>0 \\
							 1/2 &\text{ if } x=y=0.\end{cases}\]
Although our splitting rule is described by the functions $R$ and $P$ (or equivalently $r$ and $p$), which are continuous at $0$, at large times the constant terms in those functions become insignificant and the behaviour when the system is rescaled appropriately is captured instead by $R^*$ and $P^*$. We let
\[R^*_X(x,y) := \begin{cases}  R^*(x,y)P^*(x,y) &\text{ if } y>0 \\
							1/2 &\text{ if } y=0\end{cases}\]
							and
\[R^*_Y(x,y) := \begin{cases} R^*(x,y)(1-P^*(x,y)) &\text{ if } x>0 \\
							1/2 &\text{ if } x=0. \end{cases} \]
Suppose that $f=(f_X,f_Y) \in E^2$ and $0\le a\le b\le 1$. Define the functionals
\[I(f,a,b)= \int_a^b \big(2^{1/2} R^*_X(f(s))^{1/2} - f_X'(s)^{1/2} \big)^2 \d s  + \int_a^b \big(2^{1/2} R^*_Y(f(s))^{1/2} - f_Y'(s)^{1/2} \big)^2 \d s,\]
\[J(f,a,b) = I(f,a,b) + \hat f_X(b) - \hat f_X(a) + \hat f_Y(b) - \hat f_Y(a)\]
and
\[\tilde K(f,a,b) = \begin{cases} \int_a^b R^*(f(s)) ds - J(f,a,b) & \text{ if } J(f,a,b)<\infty;\\
								  -\infty & \text{ otherwise.}\end{cases}\]
Note that any $f\in E^2$ is necessarily continuous at $0$, and therefore if $\lim_{t\to 0}\tilde K(f,0,t)\neq -\infty$ then $\tilde K(f,0,t)$ is differentiable (in $t$) at $0$. If $\lim_{t\to 0}\tilde K(f,0,t) = -\infty$ then write $\frac{d}{dt} \tilde K(f,0,t)|_{t=0} = -\infty$. Define
\[K(f) = \begin{cases} \tilde K(f,0,1) & \text{ if }\, \frac{d}{dt} \tilde K(f,0,t)|_{t=0} > 0 \,\text{ and }\, \tilde K(f,0,s)> 0 \,\,\,\,\forall s\le 1;\\
							  -\infty			   & \text{ if $\exists s\le 1$ such that } \tilde K(f,0,s)<0;\\
							  0 & \text{ otherwise.}\end{cases}\]

The functional $\tilde K(f,0,1)$ will be our expected growth rate, in that the expected number of particles at time $T$ whose paths, when rescaled by a factor $T$, are ``near'' $f$ should look something like $e^{\tilde K(f,0,1)T}$. However, the actual number of particles behaving in this way will only look like $e^{\tilde K(f,0,1)T}$ if $\tilde K(f,0,\theta)> 0$ for all $\theta\in[0,1]$. If there exists $\theta\in[0,1]$ such that  $\tilde K(f,0,\theta)<0$ then (with high probability) there will be no particles whose $T$-rescaled paths look like $f$, essentially because this point on $f$ acts as a bottleneck; at this point, it is too difficult for particles to follow $f$, and the population near $f$ dies out.

In order to make this discussion precise we need to specify a topology on our space of functions $E^2$. Since the jumps in our process are exponentially distributed, it is possible for particles to make macroscopic jumps in the sense that their rescaled paths will not be continuous. In fact, it is possible for particles to make two (or more) macroscopic jumps in quick succession. This means that the usual Skorohod topology---the $J_1$ topology, generated by the Skorohod metric---is not suitable, as the rescaled set of paths that our particles take will not be compact in this topology. Instead we use the L\'evy metric \cite{levy1954theorie} on the set of increasing functions on $[0,1]$, which generates Skorohod's $M_2$ topology. We will recall the definition shortly, but first point out that the L\'evy metric has been extended to a metric on the set of c\`adl\`ag functions, known as the graph metric or Borovkov metric, which again generates the $M_2$ topology on this larger space. This metric was introduced in \cite{borovkov:convergence_distn}, and used for example in \cite{mogulskii:large_devs_CPP} to give a large deviations result for compound Poisson processes under only a weak moment Cram\'er condition, because of similar incompatibility with the $J_1$ topology.

The L\'evy metric on $E$ is defined by
\begin{equation}\label{metricdefn}
d(f,g) = \inf\big\{r>0 : f(x-r)-r < g(x) < f(x+r)+r \,\,\,\,\forall x\in[-r,1+r]\big\}
\end{equation}
where $f(x)$ is interpreted to equal $f(0)$ for $x<0$ and $f(1)$ for $x>1$, and similarly for $g$. The metric space $(E,d)$ is complete and separable. In an abuse of notation, we will also write $d$ to mean the product metric on $E^2$ defined by $d((f_X,f_Y),(g_X,g_Y))=\max\{d(f_X,g_X),d(f_Y,g_Y)\}$.

Take $T\ge0$ and let $\Nc_T$ be the set of particles that are alive at time $T$. For $u\in\Nc_T$ and $t\le T$, let $Z_u(t) = (X_u(t),Y_u(t))$ be the position of the unique ancestor of $u$ in $\Nc_t$. For $u\in\Nc_T$ and $s\in[0,1]$, write
\[Z_u^T(s) = Z_u(sT)/T;\]
we call $(Z_u^T(s),\, s\in[0,1])$ the $T$-rescaled path of $u$. For $F\subset E^2$, define
\[N_T(F) = \#\{u\in \Nc_{T} : Z_u^T \in F\},\]
the number of particles at time $T$ whose $T$-rescaled paths have remained within $F$.
Throughout the article we use the convention that $\inf \emptyset = + \infty$ and $\sup \emptyset = - \infty$.

\begin{thm}\label{mainthm}
If $F\subset E^2$ is closed and $\sup_{f\in F} K(f) \neq 0$, then
\[\limsup_{T\to\infty}\frac{1}{T}\log N_T(F) \le \sup_{f\in F} K(f) \,\,\,\, \text{almost surely,}\]
and if $F\subset E^2$ is open and $\sup_{f\in F} K(f) \neq 0$, then
\[\liminf_{T\to\infty}\frac{1}{T}\log N_T(F) \ge \sup_{f\in F} K(f)\,\,\,\,\text{almost surely.}\]
\end{thm}

The case when $\sup_{f\in F} K(f)=0$ is extremely delicate. If, for some function $f$, we have $K(f)=0$ then at some point along the path $f$, the population remaining near $f$ is ``critical'', in the sense of a critical branching process. When $\sup_{f\in F} K(f)=0$, this means that amongst all paths in $F$, the easiest path for particles to follow is a critical one. In general, critical branching processes are significantly more challenging to analyse than non-critical processes, and our situation is made more complex by the inhomogeneity of our branching system. Indeed, we do not even know if there are open sets $F\subset E^2$ that satisfy $\sup_{f\in F} K(f) = 0$. If not, then the condition that $\sup_{f\in F} K(f)\neq 0$ could essentially be removed, subject to a slight alteration to the definition of $K(f)$.

\subsection{Heuristics}

At a basic level, our theorem says that the number of particles whose $T$-rescaled paths remain close to a function $f$ is roughly $\exp(K(f)T)$. The growth rate $K(f)$ consists of two parts: the growth of the population along the path, which is simply $\int_0^1 R^*(f(s)) ds$, and the cost of a typical particle following the path, which is $J(f,0,1)$. However, if the cumulative cost is ever larger than the cumulative growth at any point along the path---that is, if $\tilde K(f,0,s)$ is ever negative---then particles are unable to follow $f$ and therefore $K(f)=-\infty$.

The main strategy for the proof is to break time up into small intervals. On each small interval, we know roughly the location and gradient of $f$ and the rate $R(f(s))$, so we can control both the growth and the cost of following $f$. We bound the largest and smallest values that $R(z)$ can take when $z$ is within a small ball around $f(s)$, and use a coupling to trap a typical particle in our process between two compound Poisson processes that have jump rates corresponding to these maximum and minimum values of $R(z)$. First and second moment bounds then allow us to translate the behaviour of this typical particle into estimates for the whole branching system.

As mentioned in the introduction, this simple explanation disguises a highly technically demanding proof. One of the difficulties that does not usually appear in work on branching structures is the behaviour at early times. A standard approach would be to let the system evolve freely for some time so that there are a large number of particles alive, and then treat each of these particles as a blank canvas, essentially starting its own copy of the original process, using the independence of these copies to improve the accuracy of initial bounds on a single population. We cannot do this, since our rate function $R$ is not only spatially dependent, but once scaled by $T$, it becomes discontinuous at $0$ (reflected in the discontinuity of $R^*$ at $0$). Instead we are forced to use a discrete-time moment bound to show that there are many particles near one particular path---a straight line corresponding to rectangles that are roughly square---at small times, and then show that this collection of particles can ``feed'' a population at future times that is easier to control.

Another non-standard element in our proof is the appearance of the L\'evy metric. As mentioned in Section \ref{thm_sec}, since our particles take jumps whose sizes are exponentially distributed, there are (many) particles whose $T$-rescaled paths are not continuous. Indeed, every particle branches at rate at least $1$, so at time $tT$ there are at least of order $e^{tT}$ particles, and the probability that one particle performs a jump larger than $aT$---which corresponds to size $a$ in the rescaled picture---is $e^{-aT}$. Thus we expect to see many such jumps when $t>a$. (And since particles can branch faster than rate $1$, we will in fact see such jumps significantly earlier.) In order to bound the total number of particles from above, we therefore need to control particles whose paths are discontinuous; hence the appearance of the L\'evy metric.

\subsection{Growth rate in expectation}

A relatively minor modification of our proof of Theorem \ref{mainthm} would yield the growth rate in expectation mentioned after the definition of $K(f)$, namely that if $F\subset E^2$ is closed then
\begin{equation}\label{expectgrowth1}
\limsup_{T\to\infty}\frac{1}{T}\log \E[N_T(F)] \le \sup_{f\in F} \tilde K(f,0,1)
\end{equation}
and if $F\subset E^2$ is open then
\begin{equation}\label{expectgrowth2}
\liminf_{T\to\infty}\frac{1}{T}\log \E[N_T(F)] \ge \sup_{f\in F} \tilde K(f,0,1).
\end{equation}
In particular one may note that there are many sets $F$ such that the expected number of particles whose rescaled paths fall within $F$ is exponentially large, since $\sup_{f\in F^\circ}\tilde K(f,0,1)>0$, and yet almost surely no particles have rescaled paths that fall within $F$, since $\sup_{f\in \bar F} K(f) = -\infty$.

In order to keep this article to a manageable length, we do not include full proofs of \eqref{expectgrowth1} and \eqref{expectgrowth2} here, although they are significantly simpler than the proofs of the upper and lower bounds in Theorem \ref{mainthm}. We will sketch the main points of the arguments in Sections \ref{mainUBproofsec} and \ref{lower_bd_sec}, shortly after the respective proofs of the upper and lower bounds in Theorem \ref{mainthm}.

\subsection{Related work}

A similar model to ours has been considered by Cesana and Hambly \cite{cesana_hambly:probabilistic_martensitic}, and Ball, Cesana and Hambly \cite{ball_et_al:martensitic_avalanches}, motivated by applications to a martensitic phase transition observed in a class of elastic crystals. They consider different splitting rule variants, and work in both two and three dimensions; but rectangles always split at rates that depend only on their area, with a constant probability $p$ (or $1-p$) of splitting horizontally (or vertically), ensuring that their models, suitably transformed, fit into the framework of generalised branching random walks. They give almost sure growth rates for the number of fragments of different shapes, and, motivated by predictions from the physics literature, they study the lengths of the horizontal ``interfaces'' between fragments, obtaining that in certain cases the total number of interfaces larger than $x$ behaves like a random variable multiplied by an explicit power of $x$.

The methods seen in this paper are related to those of Berestycki \emph{et al.}~\cite{berestycki_et_al:growth_rates} on a branching Brownian motion (BBM) with inhomogeneous breeding potential. In that paper, the authors considered a BBM in $\R$ where a particle at position $z$ branched at infinitesimal rate $|z|^p$, for $p\in[0,2)$. Their main result was roughly analogous to ours, giving almost sure growth rates along paths, and they also gave growth rates in expectation analogous to \eqref{expectgrowth1} and \eqref{expectgrowth2}. They analysed their growth rates in some detail, giving implicit equations for the optimal paths and the location of the bulk of the population (which became explicit in the cases $p=0$ and $p=1$). This was a difficult analytic task even for the relatively simple, monotone growth rate seen in \cite{berestycki_et_al:growth_rates}. Our growth rate $K$ is much more complex and it would take a substantial amount of further work to analyse the optimal paths; in order to keep this paper to a manageable length we do not attempt this here.

There are three main difficulties in our model relative to that in \cite{berestycki_et_al:growth_rates}. Firstly, in the BBM, all particles move as standard Brownian motions, independent of their location and their branching rate, whereas in our model particles jump and branch simultaneously. Indeed, it is worth noting that if the branching Brownian motion in \cite{berestycki_et_al:growth_rates} were replaced by an analogous branching random walk, then if we started with one particle at $0$, the initial particle would never branch or move; whereas if we started with a particle at any other site, then even with bounded jump sizes, the collection of particles would colonise space dramatically faster than the BBM (subject to the initial population not returning to $0$ quickly), since a particle branching at rate $|z|^p$ would also be moving at rate $|z|^p$. This highlights the challenge of controlling the dependencies between particles' positions and the growth of the population in our model.

On top of this initial difference, our branching rate $R(z)$ is much more difficult to control than the smooth, symmetric, monotone (on each half-space) function $|z|^p$. And thirdly, our particles are able to make large jumps, meaning that standard large deviations apparatus is more difficult to apply, and we must use a non-standard topology.

Roberts and Schweinsberg \cite{roberts_schweinsberg:gaussian_BBM_inhom} also consider branching Brownian motion in an inhomogeneous potential, this time with a biological application in mind, where the position of a particle represents its fitness and fitter individuals branch more quickly. They used the tools from \cite{berestycki_et_al:growth_rates} to give a heuristic explanation of some of their results, but used a more precise truncation argument for their proofs, based on techniques from \cite{berestycki_et_al:genealogy} and \cite{berestycki_et_al:critical_bbm_absorp_sp}.

For homogeneous spatial branching processes, obtaining a full picture of the spread of the population has been a subject of interest for more than 45 years. To give just a few highlights, the position of the extremal particle in BBM was studied by McKean \cite{mckean:application_bbm_kpp} and Bramson \cite{bramson:maximal_displacement_BBM, bramson:convergence_Kol_eqn_trav_waves}, with more detailed recent studies on the behaviour near the extremal particle by A{\"\i}d{\'e}kon \emph{et.~al.}~\cite{aidekon_et_al:BBM_tip} and Arguin, Bovier and Kistler \cite{arguin_et_al:extremal_BBM}. For non-lattice branching random walks, A{\"\i}d{\'e}kon \cite{aidekon:convergence_law_min_brw} proved convergence in law for the re-centered position of the extremal particle under fairly weak conditions. Bramson, Ding and Zeitouni \cite{bramson_ding_zeitouni:convergence_nonlattice} gave a shorter proof using a second moment method and indicated that it should be possible to adapt their proof to branching random walks that take values on a lattice.

\subsection{Layout of the article}

We begin, in Sections \ref{outlinesec} and \ref{lower_bd_sec}, with outlines of the proofs of the upper and lower bounds in Theorem \ref{mainthm} respectively. In these sections we state several results that are needed for the proof of the main theorem without proving them. The proofs of these intermediate results are then given in later sections.

In Section \ref{GMTsec}, we give a full construction of our system in terms of a marked binary tree. This discrete setting is useful for decoupling some of the dependency structure between the jump times and jump sizes, and allows us to show that particles remain within some compact set with high probability, which will be an important ingredient, especially for the upper bound in Theorem \ref{mainthm}.

In Section \ref{small_times_sec} we aim to control the system at small times, which is a difficult task partly due to the discontinuity of $R^*$ at $0$. We again use the discrete setup described in Section \ref{GMTsec}, and use moment estimates that take advantage of the fact that our particles prefer to split along their longest edge. This work is used for the proof of the lower bound in Theorem \ref{mainthm}.

One of the main tools in our proof is a coupling between compound Poisson processes, which we describe in Section \ref{coupling_sec} and then apply to give upper and lower bounds on the probability that a typical particle remains near a given function.

In Section \ref{finaldetailssec} we put many of the previous results together, move from lattice times to continuous time, and complete the final details of the proof of the upper bound in Theorem \ref{mainthm}.

In Appendix \ref{det_bds_rate} we give deterministic bounds that relate the maximum and minimum of $R$ on small balls to the value of $R^*$ at the centre of the ball, and therefore allow us to link the probabilistic estimates obtained in Section \ref{coupling_sec} to our growth rate $\tilde K$. 

An elementary but somewhat intricate bound on compound Poisson processes is required in Section \ref{coupling_sec}, and we prove this in Appendix \ref{CPP_append}.

Finally, in Appendix \ref{techsec} we carry out some technical work, ensuring that our state space and our growth rate behave sensibly.



\section{Proof outline for the upper bound in Theorem \ref{mainthm}}\label{outlinesec}

Since the proof of Theorem \ref{mainthm} is rather long, we break it into upper and lower bounds. In this section we state a series of results that together enable us to complete the upper bound. We will then prove those results in later sections.

\subsection{Three probabilistic ingredients}\label{probabilistic_ingredients}

The first step in our proof of the upper bound in Theorem \ref{mainthm} is to rule out certain paths that it is difficult for particles to follow, thereby reducing the paths of interest to a compact set. We define, for $M>1$,
\[G_M = \left\{f \in E : s/M \le f(s) \le Ms \,\,\forall s\in[0,1]\right\} \subset E.\]
If $f\in G_M^2$ then we say that $f$ is ``$M$-good''. We note that if $f$ is $M$-good then $R^*_X(f(s))\le M^2$ for all $s\in[0,1]$ and similarly for $R^*_Y$.

We would like to say that the rescaled paths of all particles fall within $G_M^2$ for sufficiently large $M$ , but there is a complication near $s=0$ in that particles will not jump immediately and therefore their paths will fall, however briefly, outside $G_M^2$. Expanding $G_M^2$ by any fixed distance $\eps>0$ would not allow us to control the jump rate sufficiently well, and we instead define, for $M>0$ and $T>1$,
\[G_{M,T} := \left\{f \in E : s/M - 2T^{-2/3} \le f(s) \le M(s+2T^{-2/3}) \,\,\forall s\in[0,1]\right\}.\]
If $f\in G_{M,T}^2$ then we say that $f$ is ``$(M,T)$-good''. We can then show that for large $M$ all particles are $(M,T)$-good with high probability as $T\to\infty$. We note here that the choice of $-2/3$ is not essential; we could choose any power of $T$ in $(-1,-1/2)$.

\begin{lem}\label{AlwaysGMT}
There exist $M_0>1$ and $\delta_0>0$ such that for any sufficiently large $T$,
\[\P\big(\exists v\in\Nc_T : Z_v^T \not\in G_{M_0,T}^2\big) \le e^{-\delta_0 T^{1/3}}.\]
\end{lem}

\noindent
We will prove this lemma in Section \ref{GMTsec}.

Next we give a version of the many-to-one formula, which translates expectations over all particles in our system into calculations involving just one particle. For $z_0\in[0,\infty)^2$, write $\Q_{z_0}$ for a probability measure under which $\xi_t$ is a Markov process living in $\R^2$, such that
	\begin{itemize}
	\item $\xi_0=z_0$;
	\item when the process is in state $z$, jumps occur at rate $2R(z)$;
	\item when a jump occurs from state $z$, it is of the form $(\mathbbm{e},0)$ with probability $P(z)$ and $(0,\mathbbm{e})$ with probability $1-P(z)$, where $\mathbbm{e}$ is an independent exponentially-distributed random variable with parameter $1$.
	\end{itemize}
In other words, the process under $\Q_{z_0}$ behaves like a single particle under $\P_{z_0}$ except that it jumps at twice the rate. We write $\Q_{z_0}$ both for the measure and for its corresponding expectation operator. We will often take $z_0=0$, and in this case we sometimes write $\Q$ rather than $\Q_0$.

The measure $\Q_{z_0}$ described above is precisely the measure $\Q_{z_0}^1$ that appears in \cite{harris_roberts:many_to_few}. The following result is \cite[Lemma 1]{harris_roberts:many_to_few} in the case of our model when $k=1$.

\begin{lem}[Many-to-one, Lemma 1 of \cite{harris_roberts:many_to_few} with $k=1$] \label{MtO}
	Suppose that $z\in\R^2$ and $t\ge0$. For any measurable function $f:\R^2\to\R$,
	\begin{equation}\nonumber
	\E_z \left[ \sum_{u \in \mathcal N_t} f(Z_u(t)) \right] = \Q_z \left[ f(\xi_t) e^{\int_0^t R(\xi_s) ds} \right].
	\end{equation}
\end{lem}

This, combined with Markov's inequality, allows us to give upper bounds on the number of particles whose paths fall within a particular set $F$ simply by bounding $R(f)$ over all $f\in F$ and then estimating the probability that $\xi$ falls within $F$. Estimating this probability will be our next task, but our estimates will not be exactly in terms of the quantities $R_X^*$ and $R_Y^*$ seen in Theorem \ref{mainthm}. Instead they will involve taking the worst and best possible values of $R_X$ and $R_Y$ over small balls about appropriately chosen functions, during a small time interval. We will need several definitions. The reader may like to think of $F=B(f,\eps)$ for some suitably nice function $f$ and small $\eps>0$.

For a non-empty interval $I\subset [0,1]$, $F\subset E^2$ and $T\ge1$, define
\begin{equation*}
R_X^-(I,F,T) = \inf\big\{R_X(Tg(s)) : s\in I,\, g\in F\big\}
\end{equation*}
and
\begin{equation*}
R_X^+(I,F,T) = \sup\big\{R_X(Tg(s)) : s\in I,\, g\in F\big\},
\end{equation*}
and similarly for $R_Y^-(I,F,T)$ and $R_Y^+(I,F,T)$. These correspond to the maximal and minimal possible jump rates over the interval $I$ for particles whose $T$-rescaled paths fall within $F$. For $s\in[0,1]$, we also let
\[x^-(s,F) = \inf\{g_X(s) : g\in F\}, \hspace{4mm} x^+(s,F) = \sup\{g_X(s) : g\in F\},\]
and similarly for $y^-(s,F)$ and $y^+(s,F)$.

Writing $|I|$ for the length of $I$ and $I^-$ and $I^+$ for the infimum and supremum of $I$ respectively, say that we are in the ``$X-$ case'' if $2R^-_X(I,F,T)|I| > x^+(I^+,F) - x^-(I^-,F)$; and in the ``$X+$ case'' if $x^-(I^+,F) - x^+(I^-,F) > 2R^+_X(I,F,T)|I|$. Note that these two cases are mutually exclusive, and roughly correspond to whether the drift of the process on the interval $I$ multiplied by the length of the interval is larger or smaller than the distance we would like it to travel. Note also that it is possible to be in neither case. Define
\[\mathcal E^+_X(I,F,T) = \begin{cases} \Big(\sqrt{2R^-_X(I,F,T)|I|}-\sqrt{x^+(I^+,F) - x^-(I^-,F)}\Big)^2 & \text{in the $X-$ case;}\\
									  \Big(\sqrt{2R^+_X(I,F,T)|I|}-\sqrt{x^-(I^+,F) - x^+(I^-,F)}\Big)^2 & \text{in the $X+$ case;}\\
									  0 & \text{otherwise.}\end{cases}\]
Similarly define $\mathcal E^+_Y(I,F,T)$. We note that for one function $f\in E^2$, the quantity $\mathcal E^+_X([a,b],\{f\},T) + \mathcal E^+_Y([a,b],\{f\},T)$ should be an approximation to---but a little bit bigger than---the functional $I(f,a,b)$ seen in Section \ref{intro_sec}. We similarly define a quantity that should be an approximation to $I(f,a,b)$ from below, namely
\[\mathcal E^-_X(I,F,T) = \begin{cases} \Big(\sqrt{2R^+_X(I,F,T)|I|}-\sqrt{x^-(I^+,F) - x^+(I^-,F)}\Big)^2 & \text{in the $X-$ case;}\\
									  \Big(\sqrt{2R^-_X(I,F,T)|I|}-\sqrt{x^+(I^+,F) - x^-(I^-,F)}\Big)^2 & \text{in the $X+$ case;}\\
									  0 & \text{otherwise.}\end{cases}\]

Write $\|z_1-z_2\| = \max\{|x_1-x_2|,|y_1-y_2|\}$ when $z_i = (x_i,y_i)\in \R^2$ for $i=1,2$. To help us to break the time interval $[0,1]$ into smaller chunks, for $n\in\N$ we define a new metric $\Delta_n$ on $E^2$ by
\[\Delta_{n}(f,g) := \max \left\{  \|f(i/n)- g(i/n)\| : i=0,\dots, n\right\}.\]
For $T>1$, $n\in\N$, $M>1$ and $f\in E^2$, we let
\[\Gamma_{M,T}(f,n) = B_{\Delta_n}(f,1/n^2)\cap B_d(f,1/n)\cap G_{M,T}^2\]
and for $j\in\{0,1,\ldots,n-1\}$, let $I_j=[j/n,(j+1)/n]$.

We also need to extend our rescaling notation to $\xi$ in the natural way. Write $\xi^T$ for the rescaled process $(\xi(sT)/T,\,s\in[0,1])$, and for $I\subset[0,1]$ write $\xi^T|_{I}$ for the restriction to $I$, $(\xi(sT)/T,\,s\in I)$. If $F\subset E^2$ and a function $f$ is defined on a subinterval $I$ of $[0,1]$---for example $\xi^T \lvert_{[0,\theta]}$ with $I=[0,\theta]$---then say that $f\in F|_I$ if there exists $g\in F$ such that $f(s)=g(s)$ for all $s\in I$.

\begin{prop}\label{Qprobcor}
Suppose that $f\in E^2$, $n\in\N$, $T>1$ and $M>1$. Then for any $\theta\in(0,1]$, $i\in\{0,1,\ldots,\lfloor\theta n\rfloor -1\}$, and $z$ such that $\|z-f(i/n)\|<1/n^2$,
\begin{multline*}
\Q\big(\xi^T|_{[i/n,\theta]}\in\Gamma_{M,T}(f,n)\big|_{[i/n,\theta]}\,\big|\,\xi^T_{i/n}=z\big)\\
\le \exp\bigg(-T\sum_{j=i}^{\lfloor \theta n\rfloor-1} \big(\mathcal E^+_X(I_j,\Gamma_{M,T}(f,n),T) + \mathcal E^+_Y(I_j,\Gamma_{M,T}(f,n),T)\big)\bigg).
\end{multline*}
\end{prop}

The proof of Proposition \ref{Qprobcor} will be the most interesting part of this paper, and involves coupling the process $\xi$ with two other processes, which---as long as $\xi$ remains within $\Gamma_{M,T}(f,n)$---will stay above and below $\xi$ respectively. We carry out this part of the argument in Section \ref{coupling_sec}.

\subsection{Deterministic bounds}\label{det_bds_upper_statement}

The three results Lemma \ref{AlwaysGMT}, Lemma \ref{MtO} and Proposition \ref{Qprobcor} form the main part of our argument, and contain all of the probability required for the upper bound in Theorem \ref{mainthm}. 

Our next task is to translate the quantities $\mathcal E^+_X$ and $\mathcal E^+_Y$ into the more palatable rate functions seen in our main theorem. The deterministic arguments required are not particularly interesting. It will sometimes be useful to note 
that if $\int_a^b R^*(f(s)) ds<\infty$, then $\tilde K(f,a,b)$ has the following alternative representation:
\begin{multline}\label{Kalt}
\tilde K(f,a,b) = -\int_a^b R^*(f(s)) ds + 2\sqrt2\int_a^b \sqrt{R^*_X(f(s))f'_X(s)} ds + 2\sqrt2\int_a^b \sqrt{R^*_Y(f(s))f'_Y(s)} ds\\
- f_X(b) + f_X(a) - f_Y(b) + f_Y(a).
\end{multline}
This can be seen by expanding out the quadratic terms in the definition of $I(f,a,b)$ and simplifying.

Let $\PL_n$ be the subset of functions in $E$ that are linear on each interval $[i/n,(i+1)/n]$ for all $i=0,\ldots,n-1$ and continuous on $[0,1]$.

\begin{prop}\label{detboundratefn}
Suppose that $\theta\in(0,1]$, $M>1$, $n\ge 2M$ and $f\in\PL_n^2\cap G_M^2$. Then for any $k\in\{\lceil\sqrt n\rceil,\ldots,\lfloor \theta n\rfloor-1\}$,
\[\sum_{j=k}^{\lfloor \theta n\rfloor -1} \mathcal E^+_X(I_j,\Gamma_{M,T}(f,n),T) \ge \int_{k/n}^{\lfloor \theta n\rfloor/n} \Big(\sqrt{2R_X^*(f(s))} - \sqrt{f'_X(s)}\Big)^2 ds - O\Big(\frac{M^4}{n^{1/4}}+\frac{M^3n}{T^{1/2}}\Big).\]
\end{prop}

We do not aim to give best possible bounds on the error term. Similarly for the sum on the left-hand side, small values of $j$ give rise to larger errors, so there should be some cut-off, but the choice of $\lceil\sqrt n\rceil$ is convenient rather than optimal. We will prove Proposition \ref{detboundratefn} in Appendix \ref{dbrf_sec}.

We will also need the following bound to control the $\exp(\int_0^t R(\xi_s) ds)$ term seen in Lemma \ref{MtO}.

\begin{lem}\label{finaldetRbd}
Suppose that $\theta\in(0,1]$, $M>1$, $n\ge 2M$, $T^{2/3}\ge 3Mn^{1/2}$, $f\in G_M^2$ and $g\in\Gamma_{M,T}(f,n)$. Then
\[\int_0^\theta R(Tg(s)) ds \le \int_0^{\lfloor\theta n\rfloor/n} R^*(f(s))ds + \eta(M,n,T)\]
and for any $k\in \{\lceil\sqrt n\rceil, \lceil\sqrt n\rceil +1, \ldots, \lfloor\theta n\rfloor\}$,
\[\int_{k/n}^{\lfloor\theta n\rfloor/n} R^*(f(s))ds - \eta(M,n,T)\le \int_{k/n}^\theta R(Tg(s)) ds \le \int_{k/n}^{\lfloor\theta n\rfloor/n} R^*(f(s))ds + \eta(M,n,T)\]
where
\[\eta(M,n,T) = O\Big(\frac{M^4}{n^{1/2}} + \frac{M^3n}{T^{1/3}}\Big).\]
\end{lem}

\noindent
This result will be proved in Appendix \ref{fdrsec}. Again we make little effort to make $\eta(M,n,T)$ the best possible bound.

\subsection{Completing the proof of the upper bound in Theorem \ref{mainthm}}\label{mainUBproofsec}

Recall that if $F\subset E^2$, and $g:[0,\theta]\to\R^2$, we say that $g\in F|_{[0,\theta]}$ if there exists a function $h\in F$ such that $h(u)=g(u)$ for all $u\in [0,\theta]$. We also generalise our rescaling notation slightly: for $t\in[0,T]$, $v\in\Nc_{t}$ and $s\in[0,t/T]$, write
\[Z_v^T(s) = Z_v(sT)/T;\]
again we call $(Z_v^T(s), s\in[0,t/T])$ the $T$-rescaled path of $v$ (previously this was only defined when $t=T$). We can then define
\[N_T(F,\theta) = \#\{v\in\Nc_{\theta T} : Z_v^T\in F|_{[0,\theta]}\},\]
the number of particles at time $\theta T$ whose $T$-rescaled paths have remained within $F$ up to time $\theta$.

\begin{prop}\label{fixedgprop}
Suppose that $M>1$, $\theta\in(0,1]$, $n\ge 2M$ and $T\ge 6M^{3/2}n^{3/4}$. Then for any $g\in G_M^2\cap \PL_n^2$ and $\kappa>0$,
\[\P\big(N_T(\Gamma_{M,T}(g,n),\theta)\ge \kappa\big) \le \frac{1}{\kappa}\exp\bigg(T\tilde K\Big(g,0,\frac{\lfloor\theta n\rfloor}{n}\Big) + O\Big(\frac{M^4T}{n^{1/4}} + M^3 n T^{2/3}\Big)\bigg).\]
\end{prop}

\noindent
We will prove Proposition \ref{fixedgprop}, which forms the heart of the argument to prove the upper bound in Theorem \ref{mainthm}, in Section \ref{fixedTsec}.

Our next result applies Proposition \ref{fixedgprop} to show that for $F\subset E^2$, at any large time $T$, the number of particles whose $T$-rescaled paths fall within $F$ is unlikely to be much larger than $\tilde K(f,0,1)$. Recall the definition of $M_0$ and $\delta_0$ from Lemma \ref{AlwaysGMT}.

\begin{prop}\label{fixedTprop}
Suppose that $F\subset E^2$ is closed and $M\ge 4M_0$. Then for any $\eps>0$,
\[\lim_{T\to\infty} \frac{1}{T^{1/3}}\log \P\bigg( N_T(F) \ge \exp\Big(T\sup_{f\in F\cap G_M^2} \tilde K(f,0,1) + T\eps\Big)\bigg) \le -\delta_0.\]
\end{prop}

\noindent
The proof of this result will use Lemma \ref{AlwaysGMT} together with some technical lemmas to ensure that we can cover $F$ with finitely many small balls around piecewise linear functions, and then apply Proposition \ref{fixedgprop}. The proof is also in Section \ref{fixedTsec}.

There are many paths $f$ that satisfy $K(f)=-\infty$ but $\tilde K(f,0,1)>0$. These are paths where there exists $\theta\in(0,1)$ such that $\tilde K(f,0,\theta)<0$, and therefore the population of particles whose rescaled paths are near $f$ becomes extinct around time $\theta T$. Since a population cannot recover once it becomes extinct, no particles follow such paths up to time $T$ even though the expected growth by the end of the path, $\tilde K(f,0,1)$, can be positive. For sets $F$ that only contain such paths, Proposition \ref{fixedTprop} does not provide a useful bound, and we therefore need a slightly different approach.

\begin{lem}\label{supKnull}
If $F\subset E^2$ is closed and $\sup_{f\in F} K(f) = -\infty$, then 
\[\lim_{T\to\infty} \frac{1}{T^{1/3}}\log \P\big( N_T(F) \ge 1\big) \le -\delta_0.\]
\end{lem}

\noindent
The proof of Lemma \ref{supKnull} is in Section \ref{sKn_sec}. We can then upgrade Proposition \ref{fixedTprop} and Lemma \ref{supKnull}, which are both statements about a particular large time $T$, to get the same result at \emph{all} large times simultaneously.

\begin{prop}\label{thetaUBprop}
Suppose that $F\subset E^2$ is closed and $M\ge 4M_0$. Then
\[\limsup_{T\to\infty}\frac{1}{T}\log N_T(F)\le \sup_{f\in F\cap G_M^2} \tilde K(f,0,1)\]
almost surely. If moreover $\sup_{f\in F} K(f) = -\infty$, then $\limsup_{T\to\infty} N_T(F) = 0$ almost surely.
\end{prop}

\noindent
The proof of this result will appear in Section \ref{tUBsec}. We can now complete the proof of the upper bound in our main theorem.

\begin{proof}[Proof of Theorem \ref{mainthm}: upper bound]
Since $K(f)\in\{-\infty\}\cup[0,\infty)$, if $\sup_{f\in F} K(f)<0$ then we must have $\sup_{f\in F} K(f)=-\infty$. In this case the second part of Proposition \ref{thetaUBprop} tells us that almost surely, $N_T(F)=0$ for all large $T$, and therefore $\lim_{T\to\infty}\log N_T(F) = -\infty$. On the other hand if $\sup_{f\in F} K(f) > 0$ then we have
\[\sup_{f\in F} K(f) = \sup_{f\in F} \tilde K(f,0,1),\]
and then applying the first part of Proposition \ref{thetaUBprop} tells us that
\[\limsup_{T\to\infty}\frac{1}{T}\log N_T(F)\le \sup_{f\in F\cap G_M^2} \tilde K(f,0,1) \le \sup_{f\in F} \tilde K(f,0,1) = \sup_{f\in F} K(f)\]
almost surely, and the proof is complete.
\end{proof}

\begin{proof}[Sketch proof of \eqref{expectgrowth1}]
The upper bound in expectation \eqref{expectgrowth1} follows more or less directly from estimates derived above. In particular, much of the proof of Proposition \ref{fixedgprop} involves bounding $\E[N_T(F)]$ from above when $F$ is a small ball around a suitably nice function. From there it is a relatively simple task, similarly to the proof of Proposition \ref{fixedTprop}, to apply Lemma \ref{AlwaysGMT} to reduce $F$ to a compact set, Lemma \ref{coverGMT} to cover this set with finitely many balls around suitably nice functions, and Corollary \ref{uppersemictscor} to check that the resulting bound does not significantly overshoot \eqref{expectgrowth1}.
\end{proof}

\section{Proof outline for the lower bound in Theorem \ref{mainthm}}\label{lower_bd_sec}

Let $\rho$ be the metric defined by
\[\rho(f,g) = \sup_{s\in[0,1]} \|f(s)-g(s)\| = \sup_{s\in[0,1]}\big\{|f_X(s)-g_X(s)|\vee |f_Y(s)-g_Y(s)|\big\}.\]
Rather than the set $\Gamma_{M,T}(f,n)$ seen in the proof of the upper bound, for the lower bound we will instead often use the set
\[\Lambda_{M,T}(f,n) = B_\rho(f,1/n^2)\cap G_{M,T}^2.\]
For $F\subset E^2$ and $T>0$, recall that
\[N_T(F) = \#\{u\in \Nc_{T} : Z_u^T \in F\},\]
and for $t\in[0,1]$ and $u\in\Nc_{tT}$, define
\begin{equation}\label{NtTunotation}
N_{t,T}^u(F) = \#\{v\in \Nc_{T} : u\le v,\, Z_v^T|_{[t,1]} \in F|_{[t,1]}\}.
\end{equation}
Also let $(\F_t, t\ge0)$ be the natural filtration for the process.

The main part of our proof relies on a standard second moment argument, and Propositions \ref{1mlbprop} and \ref{2mubprop} give the first and second moment bounds necessary to carry out that argument. However, this strategy on its own cannot give strong enough estimates to be able to prove an almost sure statement, as required for Theorem \ref{mainthm}. We therefore give bounds conditionally given $\F_{kT/n}$ for $\sqrt n < k \ll n$, with the aim of using the branching structure at time $kT/n$ to increase the accuracy of our estimates.

\begin{prop}\label{1mlbprop}
Suppose that $M>1$, $n\ge 6M$, $\sqrt n\le k\le n$, $T \ge 27M^{3/2}n^{9/2}$ and $f\in \PL_n^2\cap G_M^2$. Suppose also that $u\in\Nc_{kT/n}$ satisfies $\|Z_u^T(k/n) - f(k/n)\|\le \frac{1}{2n^2}$. Then
\[\E\big[N_{k/n,T}^u(\Lambda_{3M,T}(f,n))\,\big|\,\F_{kT/n}\big] \ge \exp\Big(T\tilde K(f,k/n,1) - O\Big(\frac{M^4T}{n^{1/4}} + M^3 nT^{2/3}\Big)\Big).\]
\end{prop}

We will prove Proposition \ref{1mlbprop} in Section \ref{1mlbsec}.

\begin{prop}\label{2mubprop}
Suppose that $M>1$, $n\ge 6M$, $\sqrt n\le k\le n$, $T\ge 27 M^{3/2} n^{9/2}$ and $f\in PL_n^2\cap G_M^2$. Suppose also that $u\in\Nc_{kT/n}$ satisfies $\|Z_u^T(k/n) - f(k/n)\| < \frac{1}{n^2}$. Then
\begin{multline*}
\E\big[N_{k/n,T}^u(\Lambda_{3M,T}(f,n))^2\,\big|\,\F_{kT/n}\big]\\
\le \int_{k T/n}^T \exp\big(- T\tilde K(f,k/n,t/T)\big) dt \cdot 12M^2n\exp\bigg(2T\tilde K(f,k/n,1) + O\Big(\frac{M^4 T}{n^{1/4}}+ M^3 n T^{2/3}\Big)\bigg)\\
+ \exp\bigg(T\tilde K(f,k/n,1) + O\Big(\frac{M^4 T}{n^{1/4}}+ M^3 n T^{2/3}\Big)\bigg).
\end{multline*}
\end{prop}

We will prove Proposition \ref{2mubprop} in Section \ref{2mubsec}. We now use a standard second moment method to turn Propositions \ref{1mlbprop} and \ref{2mubprop} into a lower bound on the probability that the number of particles whose rescaled paths remain near $f$ is roughly $\tilde K(f,k/n,1)T$, again conditionally on $\F_{kT/n}$.

\begin{cor}\label{PZcor}
Suppose that $M>1$, $n\ge 6M$, $k\ge\sqrt n$, $T \ge 27M^{3/2}n^{9/2}$ and $f\in PL_n^2\cap G_M^2$. Suppose also that $u\in\Nc_{kT/n}$ satisfies $\|Z_u^T(k/n) - f(k/n)\|\le \frac{1}{2n^2}$, and that $\tilde K(f,k/n,t)\ge 0$ for all $t\ge k/n$. Then
\[\P\big(N_{k/n,T}^u(\Lambda_{3M,T}(f,n)) \ge e^{T\tilde K(f,k/n,1) - O(M^4T/n^{1/4} + M^3 n T^{2/3})}\,\big|\,\F_{kT/n}\big) \ge e^{-O(\frac{M^4 T}{n^{1/4}}+ M^3 n T^{2/3})}.\]
\end{cor}

\begin{proof}
The Paley-Zygmund inequality says that, for any non-negative random variable $X$ and $\theta\in[0,1]$,
\[P\big(X\ge \theta E[X]\big) \ge (1-\theta)^2\frac{E[X]^2}{E[X^2]}.\]
Taking $P$ to be the conditional probability measure $\P(\,\cdot\, |\,\F_{kT/n})$ with $X = N_{k/n,T}^u(\Lambda_{3M,T}(f,n))$ and $\theta=1/2$, we have
\begin{multline}\label{PZconsequence}
\P\big(N_{k/n,T}^u(\Lambda_{3M,T}(f,n)) \ge (1/2)\E\big[N_{k/n,T}^u(\Lambda_{3M,T}(f,n))\,\big|\,\F_{kT/n}\big]\,\big|\,\F_{kT/n}\big)\\
\ge \frac{\E\big[N_{k/n,T}^u(\Lambda_{3M,T}(f,n))\,\big|\,\F_{kT/n}\big]^2}{4\E\big[N_{k/n,T}^u(\Lambda_{3M,T}(f,n))^2\,\big|\,\F_{kT/n}\big]}.
\end{multline}
Proposition \ref{1mlbprop} tells us that
\[\E\big[N_{k/n,T}^u(\Lambda_{3M,T}(f,n))\,\big|\,\F_{kT/n}\big] \ge \exp\Big(T\tilde K(f,k/n,1) - O\Big(\frac{M^4T}{n^{1/4}} + M^3 nT^{2/3}\Big)\Big)\]
and Proposition \ref{2mubprop} gives
\begin{multline*}
\E\big[N_{k/n,T}^u(\Lambda_{3M,T}(f,n))^2\,\big|\,\F_{kT/n}\big]\\
\le \int_{k T/n}^T \exp\big(- T\tilde K(f,k/n,t/T)\big) dt \cdot 12M^2n\exp\bigg(2T\tilde K(f,k/n,1) + O\Big(\frac{M^4 T}{n^{1/4}}+ M^3 n T^{2/3}\Big)\bigg)\\
+ \exp\bigg(T\tilde K(f,k/n,1) + O\Big(\frac{M^4 T}{n^{1/4}}+ M^3 n T^{2/3}\Big)\bigg).
\end{multline*}
and since $\tilde K(f,k/n,t)\ge 0$ for all $t\ge k/n$, this reduces to
\begin{align*}
\E\big[N_{k/n,T}^u(\Lambda_{3M,T}(f,n))^2\,\big|\,\F_{kT/n}\big] &\le 12M^2 n T \exp\bigg(2T\tilde K(f,k/n,1) + O\Big(\frac{M^4 T}{n^{1/4}}+ M^3 n T^{2/3}\Big)\bigg)\\
&= \exp\bigg(2T\tilde K(f,k/n,1) + O\Big(\frac{M^4 T}{n^{1/4}}+ M^3 n T^{2/3}\Big)\bigg).
\end{align*}
Substituting these estimates into \eqref{PZconsequence} gives the result.
\end{proof}

By Corollary \ref{PZcor}, each particle near $f(kT/n)$ at time $kT/n$ has a not-too-small probability of having roughly $\exp (\tilde K(f,k/n,1)T)$ descendants whose rescaled paths remain near $f$ up to time $1$. If we can ensure that there is a reasonably large number of particles near $f(kT/n)$ at time $kT/n$, then subject to some technicalities (for example Corollary \ref{PZcor} assumes that $f$ is piecewise linear, whereas there is no such condition in Theorem \ref{mainthm}) we will be able to prove the lower bound in Theorem \ref{mainthm}.

The discontinuity of $R^*$ at $0$ makes controlling the growth of the system at small times difficult. The first few particles in the system can have wildly different values of $R$ in different realisations of the process, and it is not \emph{a priori} clear that this cannot have a large effect on the long-term evolution of the system. Our method for showing that particles do in fact spread out in a predictable way is the following. First we show that there are many particles near the line $(s/2,s/2)$ at time $s$, for suitable values of $s$. The idea is that our jump distribution prefers to create ``almost square'' rectangles (since rectangles are more likely to break along their longest side) and therefore we should see many particles near $(s/2,s/2)$. However, since particles away from this line branch and jump more quickly, we use a discrete-time argument to keep control of the dependence between the jump locations and the jump times. A rough estimate using moments in discrete time can then be translated back into continuous time, giving the following result.

\begin{prop}\label{ctssquareprop}
Define
\[V'_{n,T} = \{u\in\mathcal N_{\lceil n^{7/8}\rceil T/n} : \|Z_u(s) - (s/2,s/2)\|\le \sfrac{T}{2n^2} \,\,\forall s\le \lceil n^{7/8}\rceil T/n\}.\]
There exists a finite constant $C$ such that for any $T\ge Cn^{48}$,
\[\P\big(|V'_{n,T}| < 2^{T/n^{1/8} - 2T/n^2}\big) \le 1/T^{3/2}.\]
\end{prop}

We will prove this result in Section \ref{ctssquaresec}. The choice of $\lceil n^{7/8}\rceil$ is somewhat arbitrary, but ensures that there are enough particles at time $\lceil n^{7/8}\rceil T/n$ to outweigh the error arising from Corollary \ref{PZcor}. The bound of $1/T^{3/2}$ is not the best possible, but is enough to use a Borel-Cantelli argument at the end of the proof of Theorem \ref{mainthm}. The requirement that $T\ge Cn^{48}$ is also certainly not optimal, but since we will take $T\to\infty$, it is sufficient for our needs.

Once we have shown that there are particles near $(s/2,s/2)$ at small times $s$, then we need to show that these particles ``feed'' other directions $(\lambda s',\mu s')$ for suitable $\lambda$ and $\mu$ and $s'>s$. Given $f\in G_M^2$, we will construct a function $h$ that begins by moving along the line $(s/2,s/2)$, so that we can guarantee large numbers of particles near $h$ at small times using Proposition \ref{ctssquareprop}, but which then gradually changes its gradient to be closer and closer to our given function $f$. At the same time we will ensure that $h$ is piecewise linear, so that we can then use Corollary \ref{PZcor} to ensure appropriate growth of particles along the whole path $h$. We then show that for $k=\lceil n^{7/8}\rceil < nt$ we have $\tilde K(h,k/n,t) \approx \tilde K(f,k/n,t)$. This is part of Proposition \ref{hfmnprop} below, which will be proved in Section \ref{hfmnsec}.

\begin{prop}\label{hfmnprop}
Suppose that $f\in G_M^2$ satisfies $\frac{d}{dt} \tilde K(f,0,t)|_{t=0} > 0$ and $\tilde K(f,0,t)>0$ for all $t\in(0,1]$. Then for any $\eps>0$ and $n\in\N$, there exists $h_{f,n}\in E^2$ such that 
\begin{equation}\label{hfmnprop0}
h_{f,n}(s)=(s/2,s/2) \,\,\,\,\text{ for all } s \le \lceil n^{7/8}\rceil/n
\end{equation}
and if $n$ is sufficiently large,
\begin{equation}\label{hfmnprop1}
h_{f,n}\in\PL_n^2\cap G_M^2\cap B(f,\eps),
\end{equation}
\begin{equation}\label{hfmnprop2}
\tilde K(h_{f,n},\lceil n^{7/8}\rceil/n,s)>0 \,\,\,\,\text{ for all } s\in(\lceil n^{7/8}\rceil/n,1]
\end{equation}
and
\begin{equation}\label{hfmnprop3}
\tilde K(h_{f,n},\lceil n^{7/8}\rceil/n,1)\ge \tilde K(f,0,1)-\eps.
\end{equation}
\end{prop}

We will prove this in Section \ref{hfmnsec}. We are now able to finish the proof of our main result.

\begin{proof}[Proof of Theorem \ref{mainthm}: lower bound]
Fix $\eps>0$. Recall $M_0$ from Lemma \ref{AlwaysGMT}. Since $K(f)\in\{-\infty\}\cup[0,\infty)$, if $\sup_{f\in F} K(f)\le 0$ then there is nothing to prove. We therefore assume that there exists $f \in F$ with $K(f)>0$. In this case, since $F$ is open and all functions $f$ with $K(f)>0$ are in $G_M^2$ for some $M$, we can choose $M\ge M_0$, $\eps'>0$ and $f\in G_M^2$ such that $B(f,2\eps')\subset F$ and
\[K(f) \ge \max\Big\{ \sup_{g\in F} K(g)-\eps, \, \frac{1}{2}\sup_{g\in F} K(g) \Big\} > 0.\]
Since $K(f)>0$, we have $\frac{d}{dt} \tilde K(f,0,t)|_{t=0} > 0$ and $\tilde K(f,0,t)>0$ for all $t\in(0,1]$. Therefore by Proposition \ref{hfmnprop}, for all sufficiently large $n\in\N$ the function $h_{f,n}$ satisfies \eqref{hfmnprop1}, \eqref{hfmnprop2} and \eqref{hfmnprop3} with $\min\{\eps/2,\eps'/2\}$ in place of $\eps$.

Take $n\in\N$ and write $k=\lceil n^{7/8}\rceil$. From Proposition \ref{ctssquareprop}, if we define
\[V'_{n,T} = \big\{u\in\mathcal N_{kT/n} : \|Z_u(s) - (s/2,s/2)\|\le \sfrac{T}{2n^2} \,\,\,\forall s\le kT/n\big\},\]
then for $T\ge Cn^{48}$ and $C$ large, we have $\P(|V'_{n,T}| \ge 2^{T/n^{1/8}-2T/n^2})\ge 1- 1/T^{3/2}$. 

Since $h_{f,n}$ satisfies \eqref{hfmnprop1}, \eqref{hfmnprop2} and \eqref{hfmnprop3} with $\min\{\eps/2,\eps'/2\}$ in place of $\eps$,
\begin{align*}
\P\big(N_T(B(f,\eps'))<e^{(\tilde K(f,0,1)-\eps)T}\big) &\le \P\big(N_T(B(h_{f,n},\eps'/2))<e^{(\tilde K(h_{f,n},k/n,1)-\eps/2)T}\big)\\
&\le \E\Big[\P\Big(N_T(B(h_{f,n},\eps'/2))<e^{(\tilde K(h_{f,n},k/n,1)-\eps/2)T}\,\Big|\,\F_{kT/n}\Big)\Big].
\end{align*}
Recalling the notation \eqref{NtTunotation}, note that if $u\in V'_{n,T}$ and $N^u_{k/n,T}(\Lambda_{3M,T}(h_{f,n},n)) \ge r$, and $n$ is sufficiently large, then $N_T(B(h_{f,n},\eps'/2))\ge r$, for any $r\ge 0$. Indeed if $u \in V'_{n,T}$ and $u \le v$ is such that $Z_v^T\lvert_{[k/n,1]} \in \Lambda_{3M,T}(h_{f,n},n)\lvert_{[k/n,1]}$ then using \eqref{hfmnprop0}
\[ \sup_{s \in [0,1]} \big\| Z_v^T(s)-h_{f,n}(s) \big\| \le \sup_{s \in [0,k/n]} \big\| Z_u^T(s)-(s/2,s/2) \big\| + \sup_{s \in [k/n,1]} \big\| Z_v^T(s)-h_{f,n}(s) \big\| \le \frac{1}{2n^2} + \frac{1}{n^2}, \]
so $Z_v^T \in B(h_{f,n}, \eps'/2)$ when $n$ is large. Thus
\begin{multline}\label{produVprime}
\P\big(N_T(B(f,\eps'))<e^{(\tilde K(f,0,1)-\eps)T}\big)\\
\le \E\Bigg[\prod_{u\in V'_{n,T}}\P\Big(N^u_{k/n,T}(\Lambda_{3M,T}(h_{f,n},n))<e^{(\tilde K(h_{f,n},k/n,1)-\eps/2)T}\,\Big|\,\F_{kT/n}\Big)\Bigg].
\end{multline}
For $n$ and $T$ sufficiently large, we check that we may apply Corollary \ref{PZcor}: indeed, by \eqref{hfmnprop2}, we have $\tilde K(h_{f,n},k/n,t)\ge 0$ for all $t\ge k/n$, and for $u\in V'_{n,T}$ we have
\[\|Z_u^T(k/n) - h_{f,n}(k/n)\| = \frac{1}{T}\big\|Z_u(kT/n) - (\sfrac{kT}{2n},\sfrac{kT}{2n})\big\| \le \frac{1}{2n^2}.\]
Thus, applying Corollary \ref{PZcor} to bound the conditional probability in \eqref{produVprime} from above, we obtain that
\[\P\big(N_T(B(f,\eps'))<e^{(\tilde K(f,0,1)-\eps)T}\big) \le \E\Bigg[\prod_{u\in V'_{n,T}}\Big(1-e^{-O(M^4 T/n^{1/4} + M^3 n T^{2/3})}\Big)\Bigg].\]
Recalling that $|V'_{n,T}|\ge 2^{T/n^{1/8}-2T/n^2}$ with probability at least $1-1/T^{3/2}$, we get
\[\P\big(N_T(B(f,\eps'))<e^{(\tilde K(f,0,1)-\eps)T}\big) \le \big(1-e^{-O(M^4 T/n^{1/4} + M^3 n T^{2/3})}\big)^{2^{T/n^{1/8}-2T/n^2}} + 1/T^{3/2}\]
and using that $1-x\le e^{-x}$ for all $x$,
\begin{equation}\label{mainlbprobT}
\P\big(N_T(B(f,\eps'))<e^{(\tilde K(f,0,1)-\eps)T}\big) \le \exp\big(-2^{T/n^{1/8}-2T/n^2}e^{-O(M^4 T/n^{1/4} + M^3 n T^{2/3})}\big) + 1/T^{3/2}.
\end{equation}

By Lemma \ref{rescalingcts} with $s=T$, for $T\ge 3M$, whenever $t-1\le T\le t$ we have
\[N_T(B(f,\eps')\cap G_{M,T}^2) \le N_t(B(f,\eps'+6M/t))\]
and therefore if $T\ge 6M/\eps'$, then we have
\[N_T(B(f,\eps')\cap G_{M,T}^2) \le \inf_{t\in[T,T+1]} N_t(B(f,2\eps')).\]
Thus
\begin{align*}
&\P\Big(\inf_{t\in[T,T+1]} N_t(B(f,2\eps'))e^{-(\tilde K(f,0,1)-\eps)t}<1\Big)\\
&\hspace{40mm}\le \P\big(N_T(B(f,\eps')\cap G_{M,T}^2)<e^{(\tilde K(f,0,1)-\eps)T}\big)\\
&\hspace{40mm}\le \P\big(N_T(B(f,\eps'))<e^{(\tilde K(f,0,1)-\eps)T}\big) + \P\big(N_T((G_{M,T}^2)^c) \ge 1\big).
\end{align*}
By Lemma \ref{AlwaysGMT}, since $M\ge M_0$, the last term is at most $e^{-\delta_0 T^{1/3}}$, and then applying \eqref{mainlbprobT}, we obtain
\begin{multline*}
\P\Big(\inf_{t\in[T,T+1]} N_t(B(f,2\eps'))e^{-(\tilde K(f,0,1)-\eps)t} < 1\Big)\\
\le \exp(-2^{T/n^{1/8}}e^{-O(M^4 T/n^{1/4} + M^3 n T^{2/3})}) + \frac{1}{T^{3/2}} + e^{-\delta_0 T^{1/3}}.
\end{multline*}
Taking $n$ large enough that the $2^{T/n^{1/8}}$ term dominates the exponent when $T$ is large, we see that this is summable in $T$, and therefore by the Borel-Cantelli lemma,
\[\P\Big(\liminf_{t\to\infty} N_t(B(f,2\eps')) e^{-(\tilde K(f,0,1)-\eps)t} < 1 \Big) = 0.\]
Since $B(f,2\eps')\subset F$ and $\tilde K(f,0,1) = K(f) \ge \sup_{g\in F}K(g) - \eps$, the statement of the theorem follows.
\end{proof}

\begin{proof}[Sketch proof of \eqref{expectgrowth2}]
Proving the lower bound in expectation \eqref{expectgrowth2} involves slightly more work than the upper bound \eqref{expectgrowth1}. Proposition \ref{hfmnprop} creates a function that approximates a given $f$ for much of its path, but begins by following the lead diagonal $(s/2,s/2)$ for a short period. Unfortunately it is designed to work for functions $f$ that satisfy $\frac{d}{dt} \tilde K(f,0,t)|_{t=0} > 0$ and $\tilde K(f,0,s)>0$ for all $s\in(0,1]$. To prove \eqref{expectgrowth2} we cannot make these assumptions on $f$, but can instead take a simpler approach than Proposition \ref{hfmnprop}. We define a function $\hat h_{f,n}$ that follows the lead diagonal $(s/2,s/2)$ until time $\lceil \sqrt n\rceil/n$, then satisfies
\[\hat h_{f,n}(j/n) = \Big(\frac{\lceil \sqrt n\rceil}{2n},\frac{\lceil \sqrt n\rceil}{2n}\Big) + f(j/n) - f(\lceil \sqrt n\rceil/n)\]
for every $j\in\{\lceil\sqrt n\rceil, \ldots, n\}$, and interpolates linearly between these values. Following a similar proof to that of Proposition \ref{lowersemicts}, one can show that
\[\liminf_{n\to\infty} \tilde K(\hat h_{f,n},0,1) \ge \tilde K(f,0,1),\]
and then combining Propositions \ref{1mlbprop} and \ref{ctssquareprop} yields \eqref{expectgrowth2}.
\end{proof}

In the proofs of the results above, it will be useful several times to note that since, for any $f$, $n$, $M$ and $T$,
\begin{equation}\label{LambdaGamma}
\Lambda_{M,T}(f,n) \subset \Gamma_{M,T}(f,n),
\end{equation}
we have
\begin{equation}\label{Rtrap}
R_X^-(I_j,\Gamma_{M,T}(f,n),T) \le R_X^-(I_j,\Lambda_{M,T}(f,n),T) \le R_X^+(I_j,\Lambda_{M,T}(f,n),T) \le R_X^+(I_j,\Gamma_{M,T}(f,n),T)
\end{equation}
and therefore by \eqref{detRbdeq}, if $M,T>1$, $n\ge 2M$, $f\in G_M^2$, $j\ge n^{1/2}$ and $s\in I_j$,
\begin{equation}\label{Rtrap2}
R_X^+(I_j,\Lambda_{M,T}(f,n),T) - \delta_{M,T}(j,n) \le R_X^*(f(s)) \le R_X^-(I_j,\Lambda_{M,T}(f,n),T) + \delta_{M,T}(j,n).
\end{equation}

\subsection{Lower bound on the first moment: proof of Proposition \ref{1mlbprop}}\label{1mlbsec}

Our aim in this section is to outline a proof of Proposition \ref{1mlbprop}. Fix $f$ as in the statement of the proposition. Let $\mathcal Z_0 = \{(0,0)\}$ and, for $j\in\{1,\ldots,n-1\}$, define
\[\mathcal Z_j = \{z\in[0,\infty)^2 : \|z-f(j/n)\|\le \sfrac{1}{2n^2}\}.\]
Lemma \ref{MtO} combined with Lemma \ref{finaldetRbd} will reduce the problem to bounding 
\[\Q\Big(\xi^T|_{[k/n,1]}\in \Lambda_{M,T}(f,n)|_{[k/n,1]}\,\Big|\, \xi^T(k/n) = w\Big)\]
for $w\in\mathcal Z_k$, so we concentrate on estimating this quantity.

Fix $n\in\N$ and $M,T>1$, and consider $f\in\PL_n^2\cap G_M^2$ and $j\in\{0,\ldots,n-1\}$. We will use the coupling mentioned in Section \ref{probabilistic_ingredients}, with details given in Section \ref{coupling_sec}. We will apply this coupling with $I=I_j$ and $F=\Lambda_{M,T}(f,n)$. Define
\begin{align*}
&q^X_{n,M,T}(z,j,f)\\
&= Q_z^{I_j,\Lambda_{M,T}(f,n),T}\Big(\big|X_-(s)- f_X(s)\big| \le \sfrac{1}{n^2} \,\,\forall s\in I_j,\, \big|X_-(\sfrac{j+1}{n})-f_X(\sfrac{j+1}{n})\big| \le \sfrac{1}{2n^2}, \, X_-|_{I_j}\in G_{M,T}|_{I_j}\hspace{-0.5mm}\Big)
\end{align*}
and
\[\hat q^X_{n,M,T}(z,j,f) = Q_z^{I_j,\Lambda_{M,T}(f,n),T}\Big( X_+(\sfrac{j+1}{n})-X_-(\sfrac{j+1}{n}) =0\Big)\]
and similarly for $q^Y_{n,M,T}(z,j,f)$ and $\hat q^Y_{n,M,T}(z,j,f)$.

\begin{lem}\label{BigQtolittleqprod}
Suppose that $n\ge 3$, $f\in\PL_n^2$ and $T>1$. Then for any $k\in\{0,\ldots,n-1\}$ and $w\in\mathcal Z_k$,
\begin{multline*}
\Q\Big(\xi^T|_{[k/n,1]}\in \Lambda_{M,T}(f,n)|_{[k/n,1]}\,\Big|\, \xi^T(k/n) = w\Big)\\
\ge \prod_{j=k}^{n-1} \inf_{z\in\mathcal Z_j} q^X_{n,M,T}(z,j,f)\, \hat q^X_{n,M,T}(z,j,f)\, q^Y_{n,M,T}(z,j,f)\, \hat q^Y_{n,M,T}(z,j,f).
\end{multline*}
\end{lem}

We carry out the proof of Lemma \ref{BigQtolittleqprod}, which consists of applying the properties of the coupling defined in Section \ref{CPP_sec}. We then need to bound the terms on the right-hand side. Bounding the $\hat q$ terms is fairly straightforward.

\begin{lem}\label{qhatlb}
Suppose that $M>1$, $n\ge 2M$, $T>1$ and $f\in \PL_n^2 \cap G_M^2$. Then for any $k\ge\lceil n^{1/2}\rceil$,
\[\prod_{j=k}^{n-1} \inf_{z\in\mathcal Z_j} \hat q^X_{n,M,T}(z,j,f) \hat q^Y_{n,M,T}(z,j,f) \ge \exp\Big(-O\Big(\frac{M^4T}{n^{1/2}} + M^3 n\Big)\Big).\]
\end{lem}

Again we will prove Lemma \ref{qhatlb} in Section \ref{CPP_sec}. Bounding the $q$ terms is much more delicate. In the following lemma, the precise form of $\Delta(j)$ is not important; we consider it a small term.

\begin{lem}\label{qlblem}
Suppose that $M>1$, $n\ge 2M$, $T>8n^{9/2}M^{3/2}$ and $f\in \PL_n^2 \cap G_M^2$. Then for any $j\in\{\lceil\sqrt n\rceil,\ldots,n-1\}$ and $z=(x,y)\in\mathcal Z_j$,
\[q^X_{n,3M,T}(z,j,f) \ge \exp\Big(-T\int_{j/n}^{(j+1)/n} \Big(\sqrt{2R^*_X(f(s))}-\sqrt{f'_X(s)}\Big)^2 ds - T\Delta(j)\Big)\]
where
\[\Delta(j) = \frac{2(M+1)}{n^{3/2}} + \frac{2\delta_{M,T}(j,n)}{n} + \frac{1}{\sqrt{n}}\sqrt{2\delta_{M,T}(j,n)\big(f_X(\sfrac{j+1}{n})-f_X(\sfrac{j}{n})\big)}\]
and $\delta_{M,T}(j,n)$ is defined in Lemma \ref{detRbd}.
\end{lem}

We again delay the proof of Lemma \ref{qlblem} to Section \ref{CPP_sec}. Putting the above ingredients together and bounding $\sum_{j=\lceil\sqrt n\rceil/n}^{n-1}\Delta(j)$ gives us our main bound, which we now state.

\begin{prop}\label{Qlbprop}
Suppose that $M>1$, $n\ge 2M$, $T>8n^{9/2}M^{3/2}$ and $f\in\PL_n^2\cap G_M^2$. Then for any $k\ge \lceil\sqrt n\rceil$ and $w\in\mathcal Z_{k}$,
\begin{multline*}
\Q\Big(\xi^T|_{[k/n,1]}\in \Lambda_{3M,T}(f,n)|_{[k/n,1]}\,\Big|\, \xi^T(k/n) = w\Big)\\
\ge \exp\bigg( -TI(f,k/n,1) - O\Big(\frac{M^4T}{n^{1/4}} + M^3 n T^{1/2}\Big)\bigg).
\end{multline*}
\end{prop}

\begin{proof}
Combining Lemmas \ref{BigQtolittleqprod}, \ref{qhatlb} and \ref{qlblem}, we have
\begin{multline*}
\Q\Big(\xi^T|_{[k/n,1]}\in \Lambda_{3M,T}(f,n)|_{[k/n,1]}\,\Big|\, \xi^T(k/n) = w\Big)\\
\ge \exp\bigg( -TI(f,k/n,1) - 2T\sum_{j=k}^{n-1}\Delta(j) - O\Big(\frac{M^4T}{n^{1/2}} + M^3 n\Big)\bigg).
\end{multline*}
Recall that
\[\Delta(j) = \frac{2(M+1)}{n^{3/2}} + \frac{2\delta_{M,T}(j,n)}{n} + \sqrt{\frac{2\delta_{M,T}(j,n)}{n}\big(f_X(\sfrac{j+1}{n})-f_X(\sfrac{j}{n})\big)}.\]
By \eqref{sumdeltaeq},
\[\sum_{j=\lceil\sqrt n\rceil}^{n-1}\frac{\delta_{M,T}(j,n)}{n} = O\Big(\frac{M^4}{n^{1/2}}+\frac{M^3 n}{T}\Big).\]
By Cauchy-Schwarz,
\[\sum_{j=\lceil\sqrt n\rceil}^{n-1} \sqrt{\frac{2\delta_{M,T}(j,n)}{n}\big(f_X(\sfrac{j+1}{n})-f_X(\sfrac{j}{n})\big)} \le \bigg(\sum_{j=\lceil\sqrt n\rceil}^{n-1} \frac{2\delta_{M,T}(j,n)}{n}\sum_{i=\lceil\sqrt n\rceil}^{n-1} \big(f_X(\sfrac{i+1}{n})-f_X(\sfrac{i}{n})\big)\bigg)^{1/2}.\]
Using \eqref{sumdeltaeq} again, together with the fact that $f\in G_M^2$, and that $\sqrt{a+b}\le\sqrt{a}+\sqrt{b}$ for $a,b\ge 0$, we have
\[\sum_{j=\lceil\sqrt n\rceil}^{n-1} \frac{1}{\sqrt n}\sqrt{2\delta_{M,T}(j,n)\big(f_X(\sfrac{j+1}{n})-f_X(\sfrac{j}{n})\big)} = O\Big(\frac{M^2}{n^{1/4}} + \frac{M^{3/2}n^{1/2}}{T^{1/2}}\Big)M^{1/2} = O\Big(\frac{M^{5/2}}{n^{1/4}} + \frac{M^2n^{1/2}}{T^{1/2}}\Big).\]
Therefore
\[\sum_{j=k}^{n-1} \Delta(j) \le \sum_{j=\lceil\sqrt n\rceil}^{n-1} \Delta(j) = O\Big(\frac{M}{n^{1/2}} + \frac{M^4}{n^{1/2}}+\frac{M^3 n}{T} + \frac{M^{5/2}}{n^{1/4}} + \frac{M^2n^{1/2}}{T^{1/2}}\Big).\]
Combining error terms gives the result.
\end{proof}

As promised, we can now easily prove Proposition \ref{1mlbprop}.

\begin{proof}[Proof of Proposition \ref{1mlbprop}]
For $u \in \mathcal{N}_{kT/n}$, let $\mathcal{N}_T^{(u)}$ be the set of descendants of $u$ in $\mathcal{N}_T$. Since $u \in \mathcal{N}_{kT/n}$, by the Markov property and Lemma \ref{MtO}, for any $k\in\{0,\ldots,n-1\}$,
\begin{multline*}
\E\Bigg[\sum_{v\in\Nc_T^{(u)}} \ind_{\{Z_v^T|_{[k/n,1]}\in\Lambda_{3M,T}(f,n)|_{[k/n,1]}\}}\,\Bigg|\,\F_{kT/n}\Bigg]\\
= \Q\Big[\ind_{\{\xi^T|_{[k/n,1]}\in\Lambda_{3M,T}(f,n)|_{[k/n,1]}\}}e^{\int_{kT/n}^T R(\xi_s)ds} \,\Big|\, \xi^T(k/n)=w \Big]\Big|_{w=Z_u^T(k/n)}.
\end{multline*}
Now, since $k\ge \lceil\sqrt n\rceil$ and $f\in G^2_M\subset G^2_{3M}$, by \eqref{LambdaGamma} and Lemma \ref{finaldetRbd}, if $\xi^T|_{[k/n,1]}\in\Lambda_{3M,T}(f,n)|_{[k/n,1]}$, then
\[\int_{kT/n}^T R(\xi_s)ds = T\int_{k/n}^1 R(T\xi^T(s)) ds \ge T\int_{k/n}^1 R^*(f(s))ds - T\eta(3M,n,T),\]
and therefore
\begin{multline*}
\E\Bigg[\sum_{v\in\Nc_T^{(u)}} \ind_{\{Z_v^T|_{[k/n,1]}\in\Lambda_{3M,T}(f,n)|_{[k/n,1]}\}}\,\Bigg|\,\F_{kT/n}\Bigg]\\
\ge e^{T\int_{k/n}^1 R^*(f(s))ds - T\eta(3M,n,T)}\Q\Big(\xi^T|_{[k/n,1]}\in\Lambda_{3M,T}(f,n)|_{[k/n,1]} \,\Big|\, \xi^T(k/n)=w \Big)\Big|_{w=Z_u^T(k/n)}.
\end{multline*}
We also know from Proposition \ref{Qlbprop} that if $w\in\mathcal Z_{k}$, then
\begin{multline*}
\Q\Big(\xi^T|_{[k/n,1]}\in \Lambda_{3M,T}(f,n)|_{[k/n,1]}\,\Big|\, \xi^T(k/n) = w\Big)\\
\ge \exp\bigg( -TI(f,k/n,1) - O\Big(\frac{M^4T}{n^{1/4}} + M^3 nT^{1/2}\Big)\bigg).
\end{multline*}
Combining these estimates and recalling that $\eta(3M,n,T)=O(M^4n^{-1/2} + M^3nT^{-1/3})$ gives the result.
\end{proof}

\subsection{Upper bound on the second moment: proof of Proposition \ref{2mubprop}}\label{2mubsec}

For our first moment bounds we used the many-to-one lemma, Lemma \ref{MtO}, which gives a method for calculating expectations of sums over all the particles in our population at a fixed time. For our second moment bound, we will need an analogue for calculating expectations of squares of sums over particles. This will involve another measure $\Q^2$, whose description is again adapted from \cite{harris_roberts:many_to_few}, this time in the case $k=2$.

Let $\Q^2$ be a probability measure under which $\xi^1_t$ and $\xi^2_t$ are Markov processes each living in $\R^2$ constructed in the following way:
\begin{itemize}
\item Take an exponential random variable $\mathbbm{e}$ of parameter $1$.
\item Let $(\chi_t,t\ge0)$ be a pure jump Markov process in $\R^2$ independent of $\mathbbm{e}$ such that $\chi_0=0$ and when $\chi_t$ is in state $z$, jumps occur at rate $2R(z)$. When there is a jump from state $z$, it is of the form $(\mathcal{E},0)$ with probability $P(z)$ and $(0,\mathcal{E})$ with probability $1-P(z)$, where $\mathcal{E}$ is an independent exponentially-distributed random variable with parameter 1.
\item Let $\tau = \inf\{t>0 : \int_0^t 2R(\chi_s) ds > \mathbbm{e}\}$.
\item Let $\xi^1_t = \xi^2_t = \chi_t$ for $t<\tau$.
\item Let $\xi^1_\tau$ equal $\chi_\tau$ plus a jump of the form $(-\log\mathcal{U},0)$ with probability $P(\chi_\tau)$ and $(0,-\log\mathcal{U})$ with probability $1-P(\chi_\tau)$, where $\mathcal{U}$ is an independent uniformly distributed random variable on $(0,1)$; let $\xi^2_\tau$ equal $\chi_\tau$ plus either $(-\log(1-\mathcal{U}),0)$ or $(0,-\log(1-\mathcal{U}))$ respectively.
\item Conditionally on $\tau$, $(\xi_t^1)_{t\le \tau}$ and $(\xi_t^2)_{t\le\tau}$, the processes $(\xi_{\tau+t}^1, t\ge 0)$ and $(\xi_{\tau+t}^2, t\ge 0)$ behave independently as if under $\Q_{\xi_\tau^1}$ and $\Q_{\xi_\tau^2}$ respectively.
\end{itemize}
We write $\Q^2$ both for the measure and for its corresponding expectation operator.

\begin{lem}[Many-to-two, Lemma 1 of \cite{harris_roberts:many_to_few} with $k=2$] \label{mt2}
	Suppose that $t\ge0$. For any measurable function $f:(\R^2)^2\to\R$,
	\[\E \Bigg[ \sum_{u_1, u_2 \in \mathcal{N}_t} f(Z_{u_1}(t), Z_{u_2}(t)) \Bigg] = \Q^2 \left[ f(\xi_t^1, \xi_t^2) e^{3 \int_0^{\tau \wedge t} R(\xi^1_s) ds +\int_{\tau \wedge t}^t R(\xi^1_s) ds + \int_{\tau \wedge t}^t R(\xi^2_s) ds} \right].\]
\end{lem}

In fact, by using the description of $\Q^2$ above, the key to the second moment bound will be to estimate terms of the form
\[\Q\big(\xi^T|_{[a,b]}\in\Lambda_{M,T}(f,n)\big|_{[a,b]}\,\big|\,\xi^T_a=z\big)\]
where $\Q=\Q_0$ is the measure seen in Section \ref{probabilistic_ingredients}. The same coupling used for Proposition \ref{Qprobcor} will yield the following result.

\begin{prop}\label{Qprobcor2}
Suppose that $f\in E^2$, $n\in\N$, $T>1$ and $M>1$. Then for any $0\le a<b\le 1$ and $z$ such that $\|z-f(a)\|<1/n^2$,
\begin{multline*}
\Q\big(\xi^T|_{[a,b]}\in\Lambda_{M,T}(f,n)\big|_{[a,b]}\,\big|\,\xi^T_a=z\big)\\
\le \exp\bigg(-T\sum_{j=\lfloor an\rfloor}^{\lceil bn\rceil-1} \big(\mathcal E^+_X(I_j\cap[a,b],\Lambda_{M,T}(f,n),T) + \mathcal E^+_Y(I_j\cap[a,b],\Lambda_{M,T}(f,n),T)\big)\bigg).
\end{multline*}
\end{prop}

We postpone the details of the proof to Section \ref{coupling_sec}. We then need to relate the right-hand side in Proposition \ref{Qprobcor2} to our rate function, in the form of the following lemma.

\begin{lem}\label{Esum2mlb}
Suppose that $M,T>1$, $n\ge 2M$ and $f\in PL_n^2 \cap G_M^2$. Then for any $a,b$ such that $\lceil\sqrt n\rceil/n\le a<b\le 1$,
\[\sum_{j=\lfloor an\rfloor}^{\lceil bn\rceil -1} \Big(\mathcal E_X^+(I_j\cap[a,b],\Lambda_{M,T}(f,n),T)+\mathcal E_Y^+(I_j\cap[a,b],\Lambda_{M,T}(f,n),T)\Big) \ge I(f,a,b) - O\Big(\frac{M^4}{n^{1/4}}+\frac{M^3 n}{T^{1/2}}\Big).\]
\end{lem}

The proof of Lemma \ref{Esum2mlb} is similar to the deterministic bounds required for the upper bound in Section \ref{det_bds_upper_statement}, but also uses the uniform structure of $\Lambda_{M,T}(f,n)$ and therefore requires slightly different estimates. We carry this out in Appendix \ref{det_bds_rate_2}, and for now continue to the proof of Proposition \ref{2mubprop}. The proof is fairly long, but uses only the ingredients above together with bounds already developed for the upper bound on the first moment.

\begin{proof}[Proof of Proposition \ref{2mubprop}]
Recall the construction of $\Q^2$ together with the Markov processes $\xi^1$ and $\xi^2$ above. For $s\ge 0$ and $T>0$, write $\xi^{1,T}_s = \xi^1_{sT}/T$ and $\xi^{2,T}_s = \xi^2_{sT}/T$. For $i=1,2$, define the event
\[\mathcal B_i = \{\xi^{i,T}|_{[k/n,1]} \in \Lambda_{3M,T}(f,n)|_{[k/n,1]}\}\]
and for the single spine $\xi$ defined under $\Q$, define
\[\mathcal B(a,b) = \{\xi^T|_{[a,b]} \in \Lambda_{3M,T}(f,n)|_{[a,b]}\}\]
By Lemma \ref{mt2},
\begin{multline*}
\E\Bigg[\Bigg(\sum_{v\in\Nc_T^{(u)}} \ind_{\{Z_v^T|_{[k/n,1]} \in \Lambda_{3M,T}(f,n)|_{[k/n,1]}\}}\Bigg)^2\,\Bigg|\,\F_{kT/n}\Bigg]\\
= \Q^2\Big[\ind_{\mathcal B_1\cap\mathcal B_2} e^{3\int_{kT/n}^{T\wedge \tau} R(\xi^1_s)ds + \int_{T\wedge\tau}^T R(\xi^1_s)ds + \int_{T\wedge\tau}^T R(\xi^2_s)ds} \,\Big|\, \tau>\sfrac{kT}{n},\, \xi^{1,T}_{k/n} = z\Big]\Big|_{z=Z_u(kT/n)/T}.
\end{multline*}
From the construction of $\Q^2$ before Lemma \ref{mt2}, it is clear that $\tau$ has a density, and that
\begin{align*}
&\Q^2\Big[\ind_{\mathcal B_1\cap\mathcal B_2} e^{3\int_{kT/n}^{T\wedge \tau} R(\xi^1_s)ds + \int_{T\wedge\tau}^T R(\xi^1_s)ds + \int_{T\wedge\tau}^T R(\xi^2_s)ds} \,\Big|\, \tau>\sfrac{kT}{n},\, \xi^{1,T}_{k/n} = z\Big]\\
&\le \int_{kT/n}^T \Q^2\Big[\ind_{\mathcal B_1\cap\mathcal B_2\cap\{\tau\in dt\}}  \,\Big|\, \tau>\sfrac{kT}{n},\, \xi^{1,T}_{k/n} = z\Big] \sup_{g\in\Lambda_{3M,T}(f,n)} e^{3\int_{kT/n}^t R(Tg(s/T))ds + 2\int_t^T R(Tg(s/T))ds}\\
& \hspace{20mm} + \Q^2\Big[\ind_{\mathcal B_1\cap\mathcal B_2\cap\{\tau>T\}}  \,\Big|\, \tau>\sfrac{kT}{n},\, \xi^{1,T}_{k/n} = z\Big] \sup_{g\in\Lambda_{3M,T}(f,n)} e^{3\int_{kT/n}^T R(Tg(s/T))ds}.
\end{align*}
It also follows from the construction of $\Q^2$ before Lemma \ref{mt2} that
\begin{align*}
&\Q^2\Big[\ind_{\mathcal B_1\cap\mathcal B_2\cap\{\tau\in dt\}}  \,\Big|\, \tau>\sfrac{kT}{n},\, \xi^{1,T}_{k/n} = z\Big]\\
&\le \Q\Big[\ind_{\mathcal B(k/n,t/T)} 2R(\xi_t)e^{-2\int_{kT/n}^t R(\xi_s)ds}dt\,\Big|\,\xi^T_{k/n} = z\Big] \sup_{\|w-f(t/T)\|<1/n^2} \Q\Big(\mathcal B(t/T,1)\,\Big|\,\xi^T_{t/T}=w\Big)^2\\
&\le \Q\Big(\mathcal B(k/n,t/T) \,\Big|\,\xi^T_{k/n} = z\Big)\sup_{\|w-f(t/T)\|<1/n^2} \Q\Big(\mathcal B(t/T,1)\,\Big|\,\xi^T_{t/T}=w\Big)^2\\
&\hspace{70mm}\cdot \sup_{h\in\Lambda_{3M,T}(f,n)} 2R(Th(t/T))e^{-2T\int_{k/n}^{t/T} R(Th(s))ds}
\end{align*}
and that
\begin{align*}
&\Q^2\Big[\ind_{\mathcal B_1\cap\mathcal B_2\cap\{\tau>T\}}  \,\Big|\, \tau>\sfrac{kT}{n},\, \xi^{1,T}_{k/n} = z\Big]\\
&\hspace{25mm}= \Q\Big[\ind_{\mathcal B(k/n,1)}e^{-2\int_{kT/n}^T R(\xi_s)ds}\,\Big|\,\xi^T_{k/n} = z\Big]\\
&\hspace{25mm}\le \Q\Big(\mathcal B(k/n,1)\,\Big|\,\xi^T_{k/n} = z\Big) \sup_{h\in\Lambda_{3M,T}(f,n)} e^{-2T\int_{k/n}^1 R(Th(s))ds}.
\end{align*}
Combining these bounds, we have shown that
\begin{align}
&\E\Bigg[\Bigg(\sum_{v\in\Nc_T^{(u)}} \ind_{\{Z_v^T|_{[k/n,1]} \in \Lambda_{3M,T}(f,n)|_{[k/n,1]}\}}\Bigg)^2\,\Bigg|\,\F_{kT/n}\Bigg]\nonumber\\
&\le \int_{kT/n}^T \Q\Big(\mathcal B(k/n,t/T)\,\Big|\,\xi^T_{k/n} = z\Big)\Big|_{z=Z_u(kT/n)/T} \cdot \sup_{\|w-f(t/T)\|<1/n^2} \Q\Big(\mathcal B(t/T,1)\,\Big|\,\xi^T_{t/T}=w\Big)^2\nonumber\\
&\hspace{5mm}\cdot \sup_{h\in\Lambda_{3M,T}(f,n)} 2R(Th(t/T))e^{-2T\int_{k/n}^{t/T} R(Th(s))ds} \sup_{g\in\Lambda_{3M,T}(f,n)} e^{3T\int_{k/n}^{t/T} R(Tg(s))ds + 2T\int_{t/T}^1 R(Tg(s))ds}dt\nonumber\\
& \hspace{5mm} + \Q\Big(\mathcal B(k/n,1)\,\Big|\,\xi^T_{k/n} = z\Big)\Big|_{z=Z_u(kT/n)/T}\nonumber\\
&\hspace{20mm}\cdot \sup_{h\in\Lambda_{3M,T}(f,n)} e^{-2T\int_{k/n}^1 R(Th(s))ds} \sup_{g\in\Lambda_{3M,T}(f,n)} e^{3T\int_{k/n}^1 R(Tg(s))ds}.\label{2momubmaster}
\end{align}

Recall that $k\ge \lceil\sqrt n\rceil$. By \eqref{LambdaGamma} and Lemma \ref{finaldetRbd}, for any $t\in[kT/n,T]$,
\[\sup_{g\in\Lambda_{3M,T}(f,n)}\int_{k/n}^{t/T} R(Tg(s))ds \le \int_{k/n}^{\lfloor nt/T\rfloor/n} R^*(f(s))ds + \eta(3M,n,T)\]
and
\[\inf_{h\in\Lambda_{3M,T}(f,n)}\int_{k/n}^{t/T} R(Th(s))ds \ge \int_{k/n}^{\lfloor nt/T\rfloor/n} R^*(f(s))ds - \eta(3M,n,T).\]
Thus
\begin{multline*}
\sup_{h\in \Lambda_{3M,T}(f,n)} e^{-2T\int_{k/n}^{t/T} R(Th(s))ds} \cdot \sup_{g\in\Lambda_{3M,T}(f,n)} e^{3T\int_{k/n}^{t/T} R(Tg(s))ds + 2T\int_{t/T}^1 R(Tg(s))ds}\\
\le \exp\bigg( -T\int_{k/n}^{\lfloor nt/T\rfloor/n} R^*(f(s))ds + 2T\int_{ k/n}^1 R^*(f(s))ds + 7T\eta(3M,n,T)\bigg).
\end{multline*}
Similarly,
\begin{multline}\label{2momexpterms2}
\sup_{h\in \Lambda_{3M,T}(f,n)} e^{-2T\int_{k/n}^1 R(Th(s))ds} \cdot \sup_{g\in\Lambda_{3M,T}(f,n)} e^{3T\int_{k/n}^1 R(Tg(s))ds}\\
\le \exp\bigg(T\int_{k/n}^1 R^*(f(s))ds + 5T\eta(3M,n,T)\bigg).
\end{multline}
By the definition of $G_{M,T}$, plus the assumption that $T^{2/3}\ge 9Mn^{1/2}$, for any $t\in [kT/n,T]$ we also have
\[\sup_{h\in\Lambda_{3M,T}(f,n)} 2R(Th(t/T)) \le 2\frac{TM(t/T+2T^{-2/3})+1}{T(t/(MT)-2T^{-2/3})} \le \frac{6MT}{Tk/(2Mn)} \le 12M^2n.\]

The above estimates bound the non-probabilistic terms in \eqref{2momubmaster}. For the other terms we apply Proposition \ref{Qprobcor2} and Lemma \ref{Esum2mlb} to obtain the bound
\begin{align*}
&\Q\Big(\mathcal B(a,b) \,\Big|\,\xi^T_{a}=z\Big)\\
&\hspace{15mm}\le \exp\bigg(-T\sum_{j=\lfloor an\rfloor}^{\lceil bn\rceil-1} \big(\mathcal E^+_X(I_j\cap[a,b],\Lambda_{3M,T}(f,n),T) + \mathcal E^+_Y(I_j\cap[a,b],\Lambda_{3M,T}(f,n),T)\big)\bigg)\\
&\hspace{15mm}\le \exp\Big(-TI(f,a,b) + O\Big(\frac{M^4 T}{n^{1/4}}+ M^3 n T^{1/2}\Big)\Big).
\end{align*}
Putting all these ingredients together, we obtain that
\begin{align}
&\E\Bigg[\Bigg(\sum_{v\in\Nc_T^{(u)}} \ind_{\{Z_v^T|_{[k/n,1]} \in \Lambda_{3M,T}(f,n)|_{[k/n,1]}\}}\Bigg)^2\,\Bigg|\,\F_{kT/n}\Bigg]\nonumber\\
&\le \int_{kT/n}^T \exp\bigg(-TI(f,k/n,t/T)-2TI(f,t/T,1)+ O\Big(\frac{M^4 T}{n^{1/4}}+ M^3 n T^{1/2}\Big)\bigg)\nonumber\\
&\hspace{15mm}\cdot 12M^2n \exp\bigg( -T\int_{k/n}^{\lfloor nt/T\rfloor/n} R^*(f(s))ds + 2T\int_{k/n}^1 R^*(f(s))ds + 7T\eta(3M,n,T)\bigg) dt\nonumber\\
& \hspace{5mm} + \exp\bigg(-TI(f,k/n,1)+ O\Big(\frac{M^4 T}{n^{1/4}}+ M^3 n T^{1/2}\Big) + T\int_{k/n}^1 R^*(f(s))ds + 5T\eta(3M,n,T)\bigg).\label{2momubmaster2}
\end{align}
Using that $f\in\PL_n^2$ and therefore is absolutely continuous, we see that
\begin{multline*}
-I(f,k/n,t/T) - 2I(f,t/T,1) - \int_{k/n}^{\lfloor nt/T\rfloor/n} R^*(f(s))ds + 2\int_{k/n}^1 R^*(f(s))ds\\
\le 2\tilde K(f,k/n,1) - \tilde K(f,k/n,t/T) + O(M^2/n). 
\end{multline*}
The result follows from substituting this into \eqref{2momubmaster2} and recalling from Lemma \ref{finaldetRbd} that
\[\eta(3M,n,T) = O\Big(\frac{M^4}{n^{1/2}} + \frac{M^3n}{T^{1/3}}\Big).\qedhere\]
\end{proof}

\section{Detailed construction and ruling out difficult paths: proof of Lemma \ref{AlwaysGMT}}\label{GMTsec}

In this section, we prove Lemma \ref{AlwaysGMT}, which said that for large $M$ all particles are $(M,T)$-good with high probability as $T\to\infty$. We will begin by defining a discrete tree with labels to represent the positions and split times of particles, which besides being a necessary step in our proof, also provides a formal construction of the process introduced in Section \ref{intro_sec}.

Take an infinite binary tree $\mathbb T$ and let $\mathbb T_n$ be the vertices in the $n$th generation of $\mathbb T$, so that $|\mathbb T_n|=2^n$. Attach to each vertex $v\in \mathbb T$ two independent random variables $\mathcal U^{\text{split}}_v$ and $\mathcal U^{\text{dir}}_v$, both uniformly distributed on $(0,1)$. Also attach another independent random variable $\mathbbm{e}_v$ which is exponentially distributed with parameter $1$.

We recursively define random variables $B_v$, $H_v$ and $T_v$ for each vertex $v\in\mathbb T$, which represent the base, height and birth time of the rectangle corresponding to $v$. Write $\rho$ for the unique vertex in $\mathbb T_0$, which we call the root. Under the probability measure $\P_{a,b}$, set $B_\rho = a$, $H_\rho = b$ and $T_\rho=0$. We write $\P$ as shorthand for $\P_{1,1}$.

Now take an integer $n\ge0$ and suppose that we have defined $B_u$, $H_u$ and $T_u$ for all vertices $u$ in generations $0,\ldots,n$. For a vertex $v\in \mathbb{T}_n$, define
\[D_v = \begin{cases} 1 & \text{ if } \mathcal U^{\text{dir}}_v\le P(-\log B_v, -\log H_v)\\
					  0 & \text{ if } \mathcal U^{\text{dir}}_v > P(-\log B_v, -\log H_v).\end{cases}\]
Write $v1$ and $v2$ for the two children of $v$ in generation $n+1$. If $D_v=1$, then set
\[B_{v1} = \mathcal U^{\text{split}}_v B_v, \,\,\,\, B_{v2} = (1-\mathcal U^{\text{split}}_v) B_v, \,\,\,\,\text{ and } \,\,\,\, H_{v1}=H_{v2}=H_v;\]
if on the other hand $D_v=0$, then set
\[H_{v1} = \mathcal U^{\text{split}}_v H_v, \,\,\,\, H_{v2} = (1-\mathcal U^{\text{split}}_v) H_v, \,\,\,\,\text{ and } \,\,\,\, B_{v1}=B_{v2}=B_v.\]
Then, for each $v\in\mathbb T$, define
\[X_v = -\log B_v \,\,\,\,\text{ and }\,\,\,\, Y_v = -\log H_v.\]
Finally, set
\[T_{v1} = T_{v2} = T_v + \frac{\mathbbm{e}_{v}}{R(X_v,Y_v)}.\]

We now translate this discrete-time process (with continuous labels) into the continuous-time model described in the introduction. For each $t\ge0$, define
\[\Nc_t = \Big\{v\in\mathbb T : T_v\le t<T_v+\frac{\mathbbm{e}_{v}}{R(X_v,Y_v)}\Big\},\]
the set of particles that are alive at time $t$. Then for $v\in\Nc_t$ and $s\le t$, if $u$ is the unique ancestor of $v$ in $\mathbb T$ that satsfies $T_u\le s<T_u+\mathbbm{e}_{u}/R(X_u,Y_u)$, then set $B_v(s) = B_u$, $H_v(s)=H_u$, $X_v(s) = X_u$ and $Y_v(s) = Y_u$. We call $Z_v(s) = (X_v(s),Y_v(s))$ the \emph{position} of particle $v$ at time $s$. For $T>0$, we can also consider particles' paths rescaled by $T$, by which we mean, for $s\le t$ and $v\in\Nc_{tT}$,
\[X_v^T(s) = \frac{X_v(sT)}{T}, \,\,\,\, Y_v^T(s) = \frac{Y_v(sT)}{T}, \,\,\,\, Z_v^T(s) = (X_v^T(s),Y_v^T(s)).\]
If we have $v\in\Nc_T$ then we may refer to $X_v^T$ to mean the function $X_v^T:[0,1]\to\R$, and similarly for $Y_v^T$ and $Z_v^T$.

\begin{lem}\label{discretedifficult}
For any $\kappa>0$, there exists $M>1$ and $N\in\N$ such that
\begin{multline*}
\P\Big(\exists v\in\mathbb T_n : X_v \not\in[n/M, Mn] \,\text{ or }\, Y_v \not\in[n/M, Mn] \,\text{ or }\, T_v < n/M \,\text{ or }\, T_v+\frac{\mathbbm{e}_{v}}{R(X_v,Y_v)} > Mn\Big)\\
\le e^{-\kappa n}
\end{multline*}
for all $n\ge N$.
\end{lem}

\begin{proof}
Note that for any $u\in\mathbb T_n$,
$X_u$ is the sum of $n$ random variables, each of which is (stochastically) bounded above by an independent exponential random variable with parameter $1$ (this is the distribution of $-\log U$ when $U$ is $U(0,1)$). Thus, if $E\sim \text{Exp}(1)$,
\[\P(X_u > Mn) \le \E[e^{X_u/2}]e^{-Mn/2} \le \E[e^{E/2}]^n e^{-Mn/2} = 2^n e^{-Mn/2}\]
and, since there are $2^n$ vertices in $\mathbb T_n$, a union bound gives
\[\P(\exists v\in\mathbb T_n : X_v > Mn) \le 4^n e^{-Mn/2}.\]
By choosing $M$ large enough, we can make this smaller than $e^{-\kappa n}$. By symmetry we also have
\[\P(\exists v\in\mathbb T_n : Y_v > Mn) \le e^{-\kappa n}.\]

For a lower bound on $X_v$ and $Y_v$, we first give a lower bound on $X_v+Y_v$. Indeed, note that for $u\in\mathbb T_n$, $X_u+Y_u$ is a sum of $n$ independent random variables, each of which is exponentially distributed with parameter $1$. Thus, for any $\lambda>0$ and any $u\in\mathbb{T}_n$,
\[\P(X_u+Y_u < n/M) \le \E[e^{-\lambda (X_u+Y_u)}]e^{\lambda n/M} = \E[e^{-\lambda E}]^n e^{\lambda n/M} = \frac{1}{(1+\lambda)^n} e^{\lambda n/M},\]
so that we can choose $M_0$ large enough that
\begin{equation}\label{XplusYlb}
\P(X_u+Y_u < n/M_0) \le 2^{-2n-2}e^{-2\kappa (n+1)}.
\end{equation}

Take $u\in\mathbb{T}_n$, let $u'$ be the unique ancestor of $u$ in $\mathbb T_{\lfloor n/2\rfloor}$ and take $M>M_0$. Note that, applying \eqref{XplusYlb}, if $n \ge 6$
\begin{align}
&\P(\exists v\in\mathbb T_n : X_v\wedge Y_v < n/M - 1)\nonumber\\
&\hspace{5mm}\le \E[\#\{v\in \mathbb T_n : X_v\wedge Y_v < n/M-1\}]\nonumber\\
&\hspace{5mm}= 2^n\P(X_u\wedge Y_u < n/M-1)\nonumber\\
&\hspace{5mm}\le 2^n\P(X_u\wedge Y_u < n/M-1 \text{ and } X_{u'}+Y_{u'}\ge \lfloor n/2 \rfloor /M_0) + 2^n\P(X_{u'}+Y_{u'}< \lfloor n/2 \rfloor /M_0)\nonumber\\
&\hspace{5mm}\le 2^n\P(X_u\wedge Y_u < n/M-1 \text{ and } X_{u'}+Y_{u'}\ge n/(3M_0)) + 2^n\cdot 2^{-2\lfloor n/2\rfloor-2}e^{-2\kappa(\lfloor n/2 \rfloor +1)}\nonumber\\
&\hspace{5mm}\le 2^n\P(X_u\wedge Y_u < n/M-1 \text{ and } X_{u'}+Y_{u'}\ge n/(3M_0)) + e^{-\kappa n}/2.\label{XminYtoX+Y}
\end{align}
Now, if $X_u\wedge Y_u < n/M-1$ and $X_{u'}+Y_{u'}\ge n/(3M_0)$, then for all vertices $v$ on the path from $u'$ to $u$, we have
\[\frac{X_v\vee Y_v + 1}{X_v\wedge Y_v+1} \ge \frac{X_v+Y_v - X_v \wedge Y_v + 1}{X_v\wedge Y_v+1} \ge \frac{X_v+Y_v}{X_v\wedge Y_v+1} -1 \ge \frac{M}{3M_0}-1.\]
Recalling the definition of $P$, this means that
\[ P(X_v, Y_v) \ge 1- \frac{X_v \wedge Y_v+1}{2(X_v \vee Y_v +1)} \ge 1- \frac{1}{2M/(3M_0)-2} \] 
and the same holds for $1-P(X_v,Y_v)$.
This means that $X_u\wedge Y_u - X_{u'}\wedge Y_{u'}$ consists of $\lceil n/2\rceil$ random variables, each of which is (stochastically) bounded below by an independent random variable $E'$ which is zero with probability $1/(2M/(3M_0)-2)$ and equals an independent copy of $E$ with probability $1-1/(2M/(3M_0)-2)$. Thus, for any $\lambda>0$,
\begin{align*}
&\P(X_u\wedge Y_u < n/M-1 \text{ and } X_{u'}+Y_{u'}\ge n/M_0)\\
&\hspace{35mm}\le \E\Big[e^{-\lambda X_u\wedge Y_u} \ind_{\big\{\frac{X_v\vee Y_v + 1}{X_v\wedge Y_v+1} \ge \frac{M}{3M_0}-1 \big\}}\Big]e^{\lambda n/M}\\
&\hspace{35mm}\le \E\Big[e^{-\lambda (X_u\wedge Y_u - X_{u'}\wedge Y_{u'} )} \ind_{\big\{\frac{X_v\vee Y_v + 1}{X_v\wedge Y_v+1} \ge \frac{M}{3M_0}-1\big\}}\Big]e^{\lambda n/M}\\
&\hspace{35mm}\le \E[e^{-\lambda E'}]^{\lceil n/2 \rceil}e^{\lambda n/M}\\
&\hspace{35mm}\le \Big(\frac{1}{2M/(3M_0)-2} + \E[e^{-\lambda E}]\Big(1-\frac{1}{2M/(3M_0)-2}\Big)\Big)^{\lceil n/2 \rceil}e^{\lambda n/M}\\
&\hspace{35mm}\le  \Big(\frac{1}{2M/(3M_0)-2} + \frac{1}{\lambda+1}\Big)^{\lceil n/2 \rceil}e^{\lambda n/M}.
\end{align*}
By choosing $\lambda$ large and then $M$ large, we can ensure that this is smaller than $2^{-n} e^{-\kappa n}/2$, which when combined with \eqref{XminYtoX+Y}, shows that for $n$ sufficiently large,
\begin{equation}\label{preUps}
\P\big(\exists v\in\mathbb T_n : X_v \not\in[n/M, Mn] \,\text{ or }\, Y_v \not\in[n/M, Mn]\big)\le e^{-\kappa n}.
\end{equation}

We now turn to $T_v$. As for $X_v$ and $Y_v$, the upper bound is easy: since $R(x,y)\ge 1$ for all $x$ and $y$, for any fixed $u\in\mathbb{T}_n$ we have
\begin{align*}
\P\Big(T_u + \frac{\mathbbm{e}_u}{R(X_u,Y_u)} > Mn\Big) &= \P\bigg(\sum_{w\le u} \frac{\mathbbm e_w}{R(X_w,Y_w)} > Mn\bigg)\\
&\le \P\bigg(\sum_{w\le u} \mathbbm e_w > Mn\bigg) \le \E[e^{E/2}]^{n+1} e^{-Mn/2} = 2^{n+1}e^{-Mn/2},
\end{align*}
so a union bound gives
\[\P\Big(\exists v\in \mathbb{T}_n : T_v + \frac{\mathbbm{e}_v}{R(X_v,Y_v)} > Mn \Big) \le 2\cdot 4^n e^{-Mn/2}\]
which can be made smaller than $e^{-\kappa n}$ by choosing $M$ large.

For a lower bound on $T_v$, define the event
\[\Upsilon_{n,M} = \{X_v \in[k/M, Mk] \,\text{ and }\, Y_v \in[k/M, Mk] \,\,\,\,\forall v\in\mathbb{T}_k,\,\,\forall k\ge n\}.\]
By \eqref{preUps}, for any $\kappa>0$, we may choose $N$ and $M_0$ sufficiently large that
\begin{equation}\label{Upsbd}
\P(\Upsilon_{n,M_0}^c) \le \sum_{j=n}^\infty \P\big(\exists v\in\mathbb T_n : X_v \not\in[n/M_0, M_0n] \,\text{ or }\, Y_v \not\in[n/M_0, M_0n]\big) \le 2^{-2n-3}e^{-2\kappa (n+1)}
\end{equation}
for all $n\ge N$. Fix $u\in\mathbb{T}_n$ and let $\rho = u_0,u_1,u_2,\ldots,u_n = u$ be the unique path from the root $\rho$ to $u$ in the tree. Then for $n\ge 2N$,
\begin{align}
\P(T_u < n/M) &= \P\bigg(\sum_{j=0}^{n-1} \frac{\mathbbm e_{u_j}}{R(X_{u_j},Y_{u_j})} < \frac{n}{M}\bigg)\nonumber\\
&\le \P(\Upsilon_{\lfloor n/2\rfloor,M_0}^c) + \P\bigg(\Upsilon_{\lfloor n/2\rfloor,M_0}\cap\bigg\{\sum_{j=\lfloor n/2\rfloor}^n \frac{\mathbbm e_{u_j}}{R(X_{u_j},Y_{u_j})} < \frac{n}{M}\bigg\}\bigg).\label{splitUps}
\end{align}
Since $n\ge 2N$, we have
\begin{equation}\label{Upsbd2}
\P(\Upsilon_{\lfloor n/2\rfloor,M_0}^c) \le 2^{-2\lfloor n/2\rfloor -3} e^{-2\kappa(\lfloor n/2\rfloor+1)}/2 \le 2^{-n-1}e^{-\kappa n}.
\end{equation}
On the event $\Upsilon_{\lfloor n/2\rfloor,M_0}$, we have
\[R(X_{u_j},Y_{u_j})\le \frac{M_0 j + 1}{j/M_0 +1} \le M_0^2\]
for all $j\ge \lfloor n/2\rfloor$; therefore
\[\P\bigg(\Upsilon_{\lfloor n/2\rfloor,M_0}\cap\bigg\{\sum_{j=\lfloor n/2\rfloor}^n \frac{\mathbbm e_{u_j}}{R(X_{u_j},Y_{u_j})} < \frac{n}{M}\bigg\}\bigg) \le \P\bigg(\sum_{j=\lfloor n/2\rfloor}^n \frac{\mathbbm e_{u_j}}{M_0^2} < \frac{n}{M}\bigg).\]
But for any $\lambda>0$,
\begin{align*}
\P\bigg(\sum_{j=\lfloor n/2\rfloor}^n \frac{\mathbbm e_{u_j}}{M_0^2} < \frac{n}{M}\bigg) &= \P\big(e^{-\lambda \sum_{j=\lfloor n/2\rfloor}^n \mathbbm e_{u_j}} > e^{-\lambda M_0^2 n/M}\big)\\
&\le \E[e^{-\lambda \sum_{j=\lfloor n/2\rfloor}^n \mathbbm e_{u_j}}]e^{\lambda M_0^2 n/M}\\
&\le \E[e^{-\lambda E}]^{n/2} e^{\lambda M_0^2 n/M} = \frac{1}{(1+\lambda)^{n/2}} e^{\lambda M_0^2 n/M}.
\end{align*}
Substituting this and \eqref{Upsbd2} into \eqref{splitUps}, we have
\[\P(T_u < n/M) \le 2^{-n-1}e^{-\kappa n} + \frac{1}{(1+\lambda)^{n/2}} e^{\lambda M_0^2 n/M}.\]
Finally, taking a union bound over all $2^n$ vertices in $\mathbb{T}_n$, we obtain
\[\P(\exists v\in\mathbb T_n : T_v < n/M) \le e^{-\kappa n}/2 + \Big(\frac{2e^{\lambda M_0^2/M}}{\sqrt{1+\lambda}}\Big)^n\]
which can be made smaller than $e^{-\kappa n}$ by choosing $\lambda$ large and then $M$ large.
\end{proof}

Fix $\alpha\in(0,1)$ and define the event
\begin{multline*}
\mathcal G_M(T) = \Big\{X_v\in\Big[\frac{n}{M}-T^\alpha, Mn+T^\alpha\Big],\, Y_v\in\Big[\frac{n}{M}-T^\alpha, Mn+T^\alpha\Big] ,\\
T_v \ge \frac{n}{M}-T^\alpha \,\text{ and }\, T_v+\mathbbm{e}_v \le Mn +T^\alpha\,\,\,\,\forall v\in\mathbb T_n \,\,\,\,\forall n\ge 0\Big\}.
\end{multline*}

\begin{cor}\label{calGcor}
There exist $M>1$ and $\delta>0$ such that for any $T\ge 0$,
\[\P(\mathcal G_M(T)^c) \le \exp(-\delta T^\alpha).\]
\end{cor}

\begin{proof}
By Lemma \ref{discretedifficult} we may choose $M\in(1,\infty)$ such that for all $n$ large enough,
\[\P(\exists v\in\mathbb T_n : X_v \not\in[n/M, Mn] \,\text{ or }\, Y_v \not\in[n/M, Mn] \,\text{ or }\, T_v < n/M \,\text{ or }\, T_v+\mathbbm{e}_v > Mn) \le e^{-n}.\]
Let
\begin{multline*}
\mathcal G_{M,n}(T) = \Big\{X_v\in\Big[\frac{n}{M}-T^\alpha, Mn+T^\alpha\Big],\, Y_v\in\Big[\frac{n}{M}-T^\alpha, Mn+T^\alpha\Big] ,\\
T_v \ge \frac{n}{M}-T^\alpha \,\text{ and }\, T_v+\mathbbm{e}_v \le Mn +T^\alpha\,\,\,\,\forall v\in\mathbb T_n\Big\},
\end{multline*}
so that
\[\mathcal G_M(T) = \bigcap_{n=0}^\infty \mathcal G_{M,n}(T).\]

For $n \le T^\alpha/M$, since $n/M-T^\alpha \le 0$, we have
\[\mathcal G_{M,n}(T) = \big\{X_v\le Mn+T^\alpha,\, Y_v\le Mn+T^\alpha \,\text{ and }\, T_v+\mathbbm{e}_v \le Mn +T^\alpha\,\,\,\,\forall v\in\mathbb T_n\big\}\]
and therefore
\[\mathcal G_{M,n}(T) \supset \big\{X_v\le T^\alpha,\, Y_v\le T^\alpha \,\text{ and }\, T_v+\mathbbm{e}_v \le T^\alpha\,\,\,\,\forall v\in\mathbb T_n\big\}.\]
By monotonicity
\[\bigcap_{n=0}^{\lfloor T^\alpha/M\rfloor} \mathcal G_{M,n}(T) \supset \big\{X_v\le T^\alpha,\, Y_v\le T^\alpha \,\text{ and }\, T_v+\mathbbm{e}_v \le T^\alpha\,\,\,\,\forall v\in\mathbb T_{\lfloor T^\alpha/M\rfloor}\big\}.\]
and thus, by our choice of $M$,
\[\P\bigg(\bigcup_{n=0}^{\lfloor T^\alpha/M\rfloor} \mathcal G_{M,n}(T)^c\bigg) \le \P\big(\exists v\in\mathbb T_{\lfloor T^\alpha/M\rfloor} : X_v> T^\alpha,\,\text{ or }\, Y_v> T^\alpha \,\text{ or }\, T_v+\mathbbm{e}_v > T^\alpha\big) \le e^{-\lfloor T^\alpha/M\rfloor}.\]
On the other hand, for $n>T^\alpha/M$,
\[\mathcal G_{M,n}(T) \supset \Big\{X_v\in\Big[\frac{n}{M}, Mn\Big],\, Y_v\in\Big[\frac{n}{M}, Mn\Big] ,\, T_v \ge \frac{n}{M} \,\text{ and }\, T_v+\mathbbm{e}_v \le Mn\,\,\,\,\forall v\in\mathbb T_n\Big\}\]
so by our choice of $M$,
\[\P(\mathcal G_{M,n}(T)^c) \le e^{-n}.\]
Combining the bounds for $n\le T^\alpha/M$ and $n> T^\alpha/M$, we have
\[\P\bigg(\bigcup_{n=0}^\infty \mathcal G_{M,n}(T)^c \bigg) \le e^{-\lfloor T^\alpha/M\rfloor} + \sum_{n>T^\alpha/M} e^{-n}\]
and choosing $\delta < 1/M$ completes the proof.
\end{proof}

We can now prove our main result for this section, Lemma \ref{AlwaysGMT}.



\begin{proof}[Proof of Lemma \ref{AlwaysGMT}]
For $t\ge0$, suppose that $u\in\mathcal N_{t}$ and let $n(u)$ be the unique $n$ such that $u\in\mathbb T_n$. By the definition of $\Nc_{t}$, we have $T_u\le t<T_u+\mathbbm{e}_u$. On $\mathcal G_M(T)$, we have $T_v+\mathbbm{e}_v \le t$ for all $v\in\mathbb T_n$ with $Mn+T^\alpha\le t$; so we must have $n(u)>(t-T^\alpha)/M$. Similarly, on $\G_M(T)$, we have $T_v>t$ for all $v\in \mathbb T_n$ with $n/M-T^\alpha>t$; so we must have $n(u)\le M(t+T^\alpha)$. Thus, on $\G_M(T)$, we have
\[\frac{t-T^\alpha}{M} < n(u)\le M(t+T^\alpha)\]
and therefore also
\[\frac{t-T^\alpha}{M^2}-T^\alpha < X_u \le M^2 (t+T^\alpha) + T^\alpha \,\,\,\,\text{ and }\,\,\,\, \frac{t-T^\alpha}{M^2}-T^\alpha < Y_u \le M^2 (t+T^\alpha) + T^\alpha.\]
Since this holds for any particle $u\in\Nc_t$ for any $t\ge0$, taking $\alpha=1/3$ and rescaling by $T$ we deduce that on $\mathcal G_M(T)$, the paths of all particles fall within $G_{M^2,T}^2$, and the result follows from Corollary \ref{calGcor}.
\end{proof}

\section{Growth of the population at small times}\label{small_times_sec}

In this section we prove two results that are essentially concerned with showing that the number of particles near any reasonable straight line $(\lambda s,\mu s)$, $s\ge 0$, grows exponentially fast. The first of these results is Proposition \ref{ctssquareprop}, which considers the case $\lambda=\mu=1/2$; the idea in this case is that our rectangles prefer to be ``roughly square'', and relatively simple moment bounds will show that there are indeed many particles near this line. This will be the content of Section \ref{ctssquaresec}. We then move on to proving Proposition \ref{hfmnprop}, which concerns a function that begins by moving along the line $(s/2,s/2)$ but then gradually shifts its gradient towards a general slope $(\lambda s,\mu s)$. This is done in Section \ref{hfmnsec}.

\subsection{The lead diagonal: proof of Proposition \ref{ctssquareprop}}\label{ctssquaresec}

Recall the discrete-time setup from Section \ref{GMTsec}. In order to initially remove the dependence between time and space, let $\tilde T_\rho = 0$, and recursively for each $v\in\mathbb{T}$ let $\tilde T_{v1} = \tilde T_{v2} = \tilde T_v + \eb_v$.

For $v\in\mathbb T_k$ and $j\le k$, write $X_v(j)$ to mean $X_u$ where $u$ is the unique ancestor of $v$ in $\mathbb T_j$. Similarly write $Y_v(j)$, $T_v(j)$, $\tilde T_v(j)$ and $Z_v(j)$. Also define
\[\Delta_v(j) = X_v(j)-Y_v(j) \;\;\text{ and }\;\; S_v(j) = X_v(j) + Y_v(j) - j,\]
and let $(\G_j, j\ge 0)$ be the natural filtration of the discrete-time process. We begin with sixth moment estimates on $\Delta_v(j)$ and $S_v(j)$. The reason for using the sixth moment is that this eventually gives us a decay of order $1/T$, which will be strong enough for our purposes. We could use higher moments if we wanted to get a better rate of decay.

\begin{lem}\label{4momlem}
There exists a finite constant $C$ such that for any vertex $v\in\mathbb T_k$ and $0\le j\le k$, we have
\[\E\big[(X_v(j) - Y_v(j))^6\big] \le Cj^3\]
and
\[\E[(X_v(j) + Y_v(j) - j)^6] \le Cj^3.\]
\end{lem}

\begin{proof}
Let $v_j$ be the vertex in $\mathbb T_j$ consisting of all $1$s, i.e.~$v_j = v_{j-1}1$ for all $j$. By symmetry it suffices to consider $v=v_k$. Letting $E_j = -\log\mathcal U_{v_j}^{\text{split}}$, we see by construction that
\[X_{v_j} - X_{v_{j-1}} = D_{v_{j-1}}E_{j-1}\;\;\text{ and }\;\;Y_{v_j} - Y_{v_{j-1}} = (1-D_{v_{j-1}})E_{j-1}.\]
We also note that $\{E_j : j\ge 0\}$ is a collection of independent exponentially distributed random variables of parameter $1$, such that $E_j$ is independent of $D_{v_j}$ for each $j$.

Let $\Delta_j = X_{v_j} - Y_{v_j}$. We begin by bounding the second moment of $\Delta_j$, then the fourth moment, before we tackle the sixth moment. By the above,
\begin{align}
\E\big[\Delta_j^2\big|\G_{j-1}\big] &= \E\big[\big(\Delta_{j-1} + (2D_{v_{j-1}}-1)E_{j-1}\big)^2\big|\G_{j-1}\big]\nonumber\\
&= \Delta_{j-1}^2 + 2\Delta_{j-1}\E[(2D_{v_{j-1}}-1)E_{j-1}|\G_{j-1}]+ \E[(2D_{v_{j-1}}-1)^2E_{j-1}^2|\G_{j-1}]\nonumber\\
&= \Delta_{j-1}^2 + 2\Delta_{j-1}\E[2D_{v_{j-1}}-1|\G_{j-1}]+ 2,\label{Deltarecurse}
\end{align}
where the last line follows from the independence of $E_{j-1}$ from $D_{v_{j-1}}$ and $\G_{j-1}$ and the fact that $(2D_{v_{j-1}}-1)^2 = 1$. Now we note that, from the definition of $D_{v_j}$, if $\Delta_j\ge 0$ then $D_{v_j}$ equals $1$ with probability at most $1/2$, whereas if $\Delta_j \le 0$ then $D_{v_j}$ equals $1$ with probability at least $1/2$. Thus
\begin{equation}\label{Deltanegcorrel}
\Delta_{j-1}\E[2D_{v_{j-1}}-1|\G_{j-1}] \le 0,
\end{equation}
so that \eqref{Deltarecurse} becomes
\[\E\big[\Delta_j^2\big|\G_{j-1}\big] \le \Delta_{j-1}^2 + 2.\]
Taking expectations and summing over $i\le j$, we obtain
\begin{equation}\label{Del2mom}
\E\big[\Delta_j^2\big] \le 2j.
\end{equation}
We now move on to the fourth moment, following a very similar argument:
\begin{align}
\E\big[\Delta_j^4\big|\G_{j-1}\big] &= \E\big[\big(\Delta_{j-1} + (2D_{v_{j-1}}-1)E_{j-1}\big)^4\big|\G_{j-1}\big]\nonumber\\
&= \Delta_{j-1}^4 + 4\Delta_{j-1}^3\E[(2D_{v_{j-1}}-1)E_{j-1}|\G_{j-1}]+ 6\Delta_{j-1}^2\E[(2D_{v_{j-1}}-1)^2E_{j-1}^2|\G_{j-1}]\nonumber\\
&\hspace{25mm}+ 4\Delta_{j-1}\E[(2D_{v_{j-1}}-1)^3E_{j-1}^3|\G_{j-1}] + \E[(2D_{v_{j-1}}-1)^4E_{j-1}^4|\G_{j-1}] \nonumber\\
&= \Delta_{j-1}^4 + 4\Delta_{j-1}^3\E[2D_{v_{j-1}}-1|\G_{j-1}]+ 6\Delta_{j-1}^2\E[E_{j-1}^2]\nonumber\\
&\hspace{25mm} + 4\Delta_{j-1}\E[2D_{v_{j-1}}-1|\G_{j-1}]\E[E_{j-1}^3] + \E[E_{j-1}^4],\label{Deltarecurse2}
\end{align}
where again for the last line we used the independence of $E_{j-1}$ from $D_{v_{j-1}}$ and $\G_{j-1}$ and the fact that $(2D_{v_{j-1}}-1)^2 = 1$. By \eqref{Deltanegcorrel} and the facts that $\E[E_{j-1}^2]=2$ and $\E[E_{j-1}^4]=24$, we obtain
\[\E\big[\Delta_j^4\big|\G_{j-1}\big] \le \Delta_{j-1}^4 + 12\Delta_{j-1}^2 + 24.\]
Taking expectations and using \eqref{Del2mom}, we have
\[\E\big[\Delta_j^4\big] \le \E[\Delta_{j-1}^4] + 24(j-1) + 24 = \E[\Delta_{j-1}^4] + 24j.\]
Summing over $i\le j$, this gives
\begin{equation}\label{Del4mom}
\E\big[\Delta_j^4\big] \le \sum_{i=1}^j 24j = 12j(j+1).
\end{equation}

For the sixth moment, the same strategy, expanding out $\Delta_j^6 = (\Delta_{j-1} + (2D_{v_{j-1}}-1)E_{j-1})^6$ and using the independence of $E_{j-1}$ from $D_{v_{j-1}}$ and $\G_{j-1}$, and then applying \eqref{Deltanegcorrel}, works again. Omitting the calculations, the upshot is that
\[\E\big[\Delta_j^6\big|\G_{j-1}\big] \le \Delta_{j-1}^6 + 30\Delta_{j-1}^4 + 360\Delta_{j-1}^2 + 720.\]
Taking expectations and using \eqref{Del2mom} and \eqref{Del4mom}, we have
\[\E\big[\Delta_j^6\big] \le \E[\Delta_{j-1}^6] + 360j(j+1) + 720j + 720 = \E[\Delta_{j-1}^6] + 360(j+1)(j+2).\]
Summing over $i\le j$, we have
\[\E\big[\Delta_j^6\big] \le 360\sum_{i=1}^j (i+1)(i+2) = 120j(j^2+6j+11)\]
which proves the first part of the lemma.

The second statement of the lemma is much simpler to prove, since $X_{v_j} + Y_{v_j} = \sum_{i=0}^{j-1}E_i$. Either by direct calculation or by using the moment generating function, one may write down an expression for every moment of $X_v(j) + Y_v(j) - j$; in particular one may check that
\[\E[(X_v(j) + Y_v(j) - j)^6] \le Cj^3\]
for some constant $C$, completing the proof.
\end{proof}

\begin{lem}\label{UnT}
Let $K(n,T) = \lceil n^{7/8}\rceil T/n + \lceil 2T/n^2\rceil$ and let
\[U_{n,T} = \big\{v\in\mathbb{T}_{K(n,T)} : \| Z_v(k) - (k/2,k/2)\|\le \sfrac{T}{32n^4} \text{ and } |\tilde T_v(k)-k|\le \sfrac{T}{4n^2} \,\,\forall k\le K(n,T)\big\}.\]
Then there exists a finite constant $C$ such that for any $T\ge Cn^{48}$,
\[\P\big(|U_{n,T}|\ge \sfrac{1}{2T^2} 2^{K(n,T)}\big) \ge 1-T^{-3/2}.\]
\end{lem}

\begin{proof}
Note that, for any $k\ge 0$ and $u\in\mathbb T_k$, by the triangle inequality we have
\[|X_u-k/2| = \frac{1}{2}|X_u+Y_u-k + X_u-Y_u| \le \frac{1}{2}|X_u+Y_u-k| + \frac{1}{2}|X_u-Y_u|,\]
and similarly for $|Y_u-k/2|$. Thus
\[\P\big(\|Z_u - (k/2,k/2)\| > \sfrac{T}{32n^4}\big) \le \P(|X_u +Y_u - k| > \sfrac{T}{32n^4}) + \P(|X_u -Y_u| > \sfrac{T}{32n^4}).\]
Applying Markov's inequality and the sixth moment estimates from Lemma \ref{4momlem}, we obtain
\begin{align*}
\P\big(\|Z_u - (k/2,k/2)\| > \sfrac{T}{32n^4}\big) &\le \E\big[|X_u +Y_u - k|^6\big]\big(\sfrac{32 n^{4}}{T}\big)^6 + \E\big[|X_u -Y_u|^6\big]\big(\sfrac{32 n^{4}}{T}\big)^6\\
&\le 2Ck^3\big(\sfrac{32 n^{4}}{T}\big)^6
\end{align*}
where $C$ is a finite constant. Thus, for $K\ge 0$, $v\in\mathbb{T}_{K}$ and $k\le K$,
\[\P\big(\|Z_v(k) - (k/2,k/2)\| > \sfrac{T}{32n^4}\big) \le 2Ck^3\big(\sfrac{32n^4}{T}\big)^6.\]

Now note that $\tilde T_v(k)$ is a sum of $k$ independent exponential random variables of parameter $1$, and therefore has the same distribution as $X_v(k)+Y_v(k)$. Thus, again by Lemma \ref{4momlem},
\[\E\big[|\tilde T_v(k) - k|^6\big] \le Ck^3\]
and therefore
\[\P\big(|\tilde T_v(k)-k| > \sfrac{T}{4n^2}\big) \le \E\big[|\tilde T_v(k) - k|^6\big]\big(\sfrac{4 n^{2}}{T}\big)^6 \le Ck^3\big(\sfrac{4 n^{2}}{T}\big)^6.\]
We deduce that, for some finite constant $C'$,
\begin{equation}\label{ZTKsum}
\P\big(\exists k\le K : \|Z_v(k) - (k/2,k/2)\| > \sfrac{T}{32n^4} \text{ or } |\tilde T_v(k)-k| > \sfrac{T}{4n^2}\big) \le \sum_{k=1}^{K} \frac{C' k^3 n^{24}}{T^6}.
\end{equation}
Summing over $k$, this is at most $C'K^4 n^{24}/T^6$, and since $K(n,T) = O(n^{-1/8} T) \le O(T)$, we have
\[\P\big(\exists k\le K(n,T) : \|Z_v(k) - (k/2,k/2)\| > \sfrac{T}{32n^4} \text{ or } |\tilde T_v(k)-k| > \sfrac{T}{4n^2}\big) \le \frac{C'' n^{24}}{T^2}\]
for some finite constant $C''$.

Converting the above to a statement about $U_{n,T}$, we have shown that
\[\P(v\in U_{n,T}) \ge 1 - \frac{C'' n^{24}}{T^2},\]
and since there are $2^{K(n,T)}$ vertices in $\mathbb{T}_{K(n,T)}$,
\[\E\big[|U_{n,T}|\big] \ge 2^{K(n,T)}\Big(1 - \frac{C'' n^{24}}{T^2}\Big).\]
Obviously we also have
\[\E\big[|U_{n,T}|^2\big] \le 2^{2K(n,T)},\]
and therefore by the Paley-Zygmund inequality,
\[\P\bigg(|U_{n,T}|\ge \frac{1}{T^2} 2^{K(n,T)}\Big(1 - \frac{C'' n^{24}}{T^2}\Big) \bigg) \ge \bigg(1-\frac{1}{T^2}\bigg)^2\frac{\E\big[|U_{n,T}|\big]^2}{\E\big[|U_{n,T}|^2\big]} \ge \bigg(1-\frac{2}{T^2}\bigg)\Big(1 - \frac{C'' n^{24}}{T^2}\Big)^2.\]
The result follows.
\end{proof}

\begin{lem}\label{VnT}
Define
\[V_{n,T} = \big\{v\in\mathbb{T}_{K(n,T)} : \|Z_v(k) - (k/2,k/2)\|\le \sfrac{T}{32n^4} \text{ and } |T_v(k)-k|\le \sfrac{7T}{8n^2} \,\,\,\forall k\le K(n,T)\big\},\]
where $K(n,T) = \lceil n^{7/8}\rceil T/n + \lceil 2T/n^2\rceil$ as in Lemma \ref{UnT}.
Then there exists a finite constant $C$ such that for any $T\ge Cn^{48}$,
\[\P\big(|V_{n,T}|\ge \sfrac{1}{2T^2} 2^{K(n,T)}\big) \ge 1-T^{-3/2}.\]
\end{lem}

\begin{proof}
We claim that every $v\in U_{n,T}$ is also in $V_{n,T}$. By Lemma \ref{UnT} this is sufficient to complete the proof.

Take $v\in U_{n,T}$. In particular, for each $k\le K(n,T)$, we have $\|Z_v(k)-(k/2,k/2)\|\le \sfrac{T}{32n^4}$. This ensures that $v$ satisfies the first condition required to be in $V_{n,T}$, but it also implies that
\[R(X_v(k),Y_v(k)) \le \frac{\frac{k}{2} + \frac{T}{32 n^4} + 1}{\frac{k}{2} - \frac{T}{32 n^4} + 1} = \frac{1+\frac{T}{32n^4(k/2+1)}}{1-\frac{T}{32n^4(k/2+1)}},\]
and so for $k\ge \frac{T}{4n^2}-2$,
\begin{equation}\label{VnTRbd}
R(X_v(k),Y_v(k)) \le \frac{1+\frac{1}{4n^2}}{1-\frac{1}{4n^2}} \le \frac{1}{(1-\frac{1}{4n^2})^2} \le \frac{1}{1-\frac{1}{2n^2}},
\end{equation}
where we used the fact that $1+x\le 1/(1-x)$ for $x\in[0,1)$.

Now, $\tilde T_v(k)$ consists of a sum of $k$ independent exponential random variables of parameter $1$, which we call $\mathbbm{e}_v(0),\ldots,\mathbbm{e}_v(k-1)$. For $k\ge \lfloor T/4n^2\rfloor$ we then have, by definition,
\[T_v(k) = \sum_{i=0}^{k-1} \frac{\mathbbm{e}_v(i)}{R(X_v(i),Y_v(i))} \ge \sum_{i=\lfloor T/4n^2\rfloor}^{k-1} \frac{\mathbbm{e}_v(i)}{R(X_v(i),Y_v(i))}.\]
Applying \eqref{VnTRbd}, this is at least
\[\Big(1-\frac{1}{2n^2}\Big) \sum_{i=\lfloor T/4n^2\rfloor}^{k-1} \mathbbm{e}_v(i) = \Big(1-\frac{1}{2n^2}\Big)\big(\tilde T_v(k) - \tilde T_v(\lfloor T/4n^2\rfloor)\big).\]
Since $v\in U_{n,T}$, whenever $k\le K(n,T)$ we have $|\tilde T_v(k) - k|\le T/4n^2$, and we obtain
\[T_v(k) \ge \Big(1-\frac{1}{2n^2}\Big)\big(k-T/4n^2 - (\lfloor T/4n^2\rfloor + T/4n^2)\big) \ge \Big(1-\frac{1}{2n^2}\Big)\Big(k-\frac{3T}{4n^2}\Big) \ge k-\frac{7T}{8n^2}.\]
We also obviously have $T_v(k) \ge 0 \ge k-\frac{7T}{8n^2}$ when $k<\lfloor T/4n^2\rfloor$; and since $R(x,y)\ge 1$ for all $x$ and $y$, we have $T_v(k)\le \tilde T_v(k)$ for all $k$. Thus we have shown that if $v\in U_{n,T}$ then $|T_v(k)-k|\le 7T/8n^2$ for all $k\le K(n,T)$, and we deduce that also $v\in V_{n,T}$, as required.
\end{proof}

We now want to move from discrete to continuous time. We need some more notation. For $v\in\mathcal N_t$ and $s\le t$, let $v(s)$ be the unique ancestor of $v$ that is in $\mathcal N_s$. Also let $\gen(v)$ be the unique integer such that $v\in\mathbb{T}_{\gen(v)}$.

\begin{lem}\label{genlem}
Recall the definition of $V_{n,T}$ from Lemma \ref{VnT}. If $v\in V_{n,T}$ then $v\in\mathcal N_t$ for some $t\ge \lceil n^{7/8}\rceil T/n + \lceil 2T/n^2\rceil - T/n^2$, and
\[\big|\gen(v(s)) - s\big| \le \frac{7T}{8n^2} + 1\]
for all $s\le \lceil n^{7/8}\rceil T/n + \lceil 2T/n^2\rceil - T/n^2 - 1$.
\end{lem}

\begin{proof}
If $v\in V_{n,T}$ then $|T_v(k)-k|\le 7T/8n^2$ for all $k\le K(n,T)$. In particular
\[T_v\ge K(n,T)-T/n^2 = \lceil n^{7/8}\rceil T/n + \lceil 2T/n^2\rceil - T/n^2,\]
and therefore $v\in\mathcal N_t$ for some $t\ge \lceil n^{7/8}\rceil T/n + \lceil 2T/n^2\rceil - T/n^2$.

Now, for any $s\le \lceil n^{7/8}\rceil T/n + \lceil 2T/n^2\rceil - T/n^2 - 1$, since $v\in V_{n,T}$,
\[T_{v(s)} = T_v(\gen(v(s))) \ge \gen(v(s)) - 7T/8n^2,\]
so since $T_{v(s)}\le s$ (because $v(s)\in\mathcal N_s$) we have
\begin{equation}\label{genvs}
\gen(v(s))\le s+7T/8n^2.
\end{equation}
Since $s\le \lceil n^{7/8}\rceil T/n + \lceil 2T/n^2\rceil - T/n^2 - 1$, the above implies in particular that $\gen(v(s))+1\le K(n,T) = \gen(v)$ and therefore we also have (again since $v\in V_{n,T}$)
\[T_{v(s)}+\mathbbm{e}_{v(s)} = T_v(\gen(v(s))+1) \le \gen(v(s))+1+7T/8n^2.\]
Combining this with the fact that $T_{v(s)}+\mathbbm{e}_{v(s)} > s$ (because $v(s)\in\mathcal N_s$), we obtain
\[s<\gen(v(s))+1+7T/8n^2,\]
and combining this with \eqref{genvs} gives the result.
\end{proof}

We can now prove the main result of this section.

\begin{proof}[Proof of Proposition \ref{ctssquareprop}]
By Lemma \ref{VnT}, with probability at least $1-1/T^{3/2}$, we have $|V_{n,T}|\ge \sfrac{1}{2T^2} 2^{K(n,T)}$. Suppose that $v\in V_{n,T}$ and let $t=\lceil n^{7/8}\rceil T/n$. Then by Lemma \ref{genlem}, $\gen(v(t)) \ge t-7T/8n^2-1$, and of course $\gen(v) = K(n,T) = \lceil n^{7/8}\rceil T/n + \lceil 2T/n^2\rceil$. Thus the number of descendants that $v(t)$ has in $V_{n,T}$ is at most
\[2^{\lceil n^{7/8}\rceil T/n + \lceil 2T/n^2\rceil - (\lceil n^{7/8}\rceil T/n - 7T/8n^2 - 1)} \le 2^{3T/n^2}.\]
We deduce that if $|V_{n,T}|\ge \frac{1}{2T^2} 2^{K(n,T)}$ then the number of \emph{distinct} ancestors of particles in $V_{n,T}$ that are in $\mathcal N_t$ must be at least
\[\frac{2^{K(n,T)}}{2T^2 \cdot 2^{3T/n^2}} \ge 2^{T/n^{1/8} - T/n^2 - 2\log_2 T - 1}.\]
For $T\ge Cn^{48}$ and $C$ large the right-hand side is certainly larger than $2^{T/n^{1/8} - 2T/n^2}$.

Now, if $u\in\mathcal N_t$ is an ancestor of a particle $v\in V_{n,T}$, and $s\le t$, then
\begin{align*}
\|Z_u(s)-(s/2,s/2)\| &= \|Z_v(s)-(s/2,s/2)\|\\
&\le \big\|Z_v(s) - \big(\sfrac{\gen(v(s))}{2},\sfrac{\gen(v(s))}{2}\big)\big\| + \|\big(\sfrac{\gen(v(s))}{2},\sfrac{\gen(v(s))}{2}\big) - (s/2,s/2)\big\|\\
&\le \frac{T}{32n^4} + \frac{1}{2}|\gen(v(s))-s|\\
&\le \frac{T}{32n^4} + \frac{1}{2}\Big(\frac{7T}{8n^2} + 1\Big)
\end{align*}
where for the first inequality we used the triangle inequality, for the second we used that $v\in V_{n,T}$, and for the third we again used that $v\in V_{n,T}$ together with Lemma \ref{genlem}. For $T\ge Cn^{48}$ and $C$ large this is smaller than $T/2n^2$, which completes the proof.
\end{proof}

\subsection{From the lead diagonal to other gradients: proof of Proposition \ref{hfmnprop}}\label{hfmnsec}

We will build up to the proof of Proposition \ref{hfmnprop} gradually, first constructing a suitable candidate function $h_{f,n}$, and then proving several lemmas that establish the required properties of $h_{f,n}$.

For $\mu\ge \lambda>0$ let
\[\kappa(\lambda,\mu) = \frac{\mu}{\lambda} - \Big(\sqrt2\Big(\frac{\mu}{\lambda}-\frac{1}{2}\Big)^{1/2} - \lambda^{1/2}\Big)^2 - (1-\mu^{1/2})^2.\]
We have defined $\kappa$ in such a way that, for $\mu\ge\lambda>0$, if $g(s)=(\lambda s,\mu s)$ for $s\in[0,1]$ then
\[\tilde K(g,0,t) = \kappa(\lambda,\mu)t.\]

We would like our function $h_{f,n}$ to begin with gradient $(1/2,1/2)$, but then to transition in small steps to having gradient $(f'_X(0),f'_Y(0))$. In order to ensure that $\tilde K(h_{f,n},0,t)$ remains positive for all small $t$, we need to check that $\kappa(\lambda,\mu)$ is strictly positive for all the gradients $(\lambda,\mu)$ that $h_{f,n}$ passes through at small times. If $\kappa$ was concave (or even concave on the region where it is positive) then this would be trivial since we could ask $h_{f,n}$ to transition linearly. Unfortunately there is a small region on which $\kappa$ is positive and not concave, so we have to use a more complicated argument. This is done in the following lemma.

\begin{lem}\label{kappaconcave}
For every $0 < \lambda \le \mu$ such that $\kappa(\lambda, \mu)>0$, there exists a path $\gamma(t)=(\gamma_X(t), \gamma_Y(t))$, $t \in [0,1]$ and $\kappa_0>0$ such that 
\begin{enumerate}[(i)]
\item \label{concave1} $(\gamma_X(0), \gamma_Y(0))=(1/2, 1/2)$ and $(\gamma_X(1), \gamma_Y(1))=(\lambda, \mu)$;
\item \label{concave2} $\kappa(\gamma(t)) \ge \kappa_0>0$ for all $t \in [0,1]$; 
\item \label{concave3} $\gamma$ is piecewise linear and $|\gamma_X'(t)|\le 20$ and $|\gamma_Y'(t)|\le 20$ for all $t \in [0,1]$ such that $\gamma$ is differentiable at $t$;
\item \label{concave4} $\gamma_X(t) \in [ 3/2 - \sqrt2, 10]$ and $\gamma_Y(t) \in [ 3/2 - \sqrt2, 10]$ for all $t\in[0,1]$.
\end{enumerate}
\end{lem}

\hspace{0.02 cm}
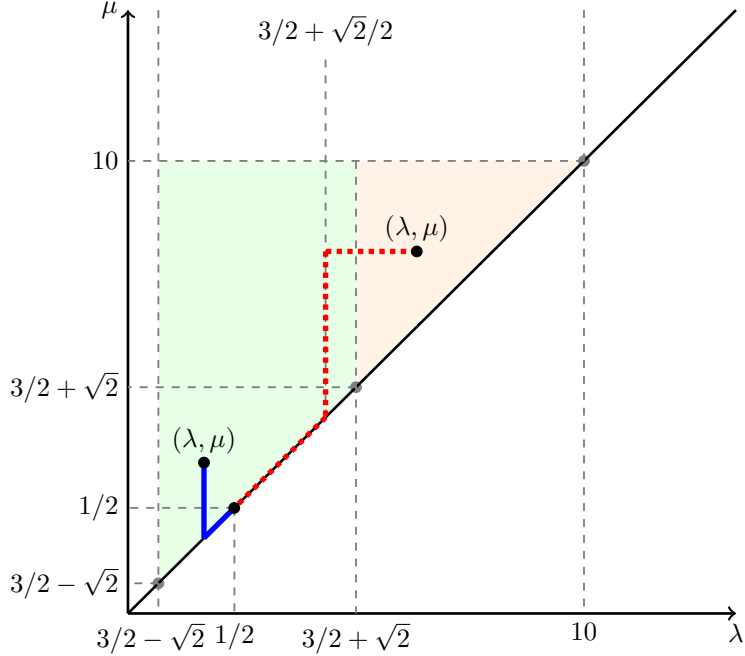
\begin{figure}[htp!]
\centering
\begin{tikzpicture}[ scale=2]
\draw (0, 0) coordinate(O)node[below left]{};

\draw (-3.8, -4) coordinate(A)node[below]{\hspace{-2mm}$3/2-\sqrt{2}$};
\draw (-3.3, -4) coordinate(B)node[below]{$1/2$};
\draw (-2.5, -4) coordinate(C)node[below]{$3/2+\sqrt{2}$};
\draw (-2.7, -0.3) coordinate(C')node[above]{$3/2+\sqrt{2}/2$};
\draw (-1, -4) coordinate(D)node[below]{$10$};

\draw (-4, -3.8) coordinate(E)node[left]{$3/2-\sqrt{2}$};
\draw (-4, -3.3) coordinate(F)node[left]{$1/2$};
\draw (-4, -2.5) coordinate(G)node[left]{$3/2+\sqrt{2}$};
\draw (-4, -1) coordinate(H)node[left]{$10$};

\coordinate (I) at (-3.8,-3.8);
\coordinate (II) at (intersection of -1,-9--D and 0,-1--H);
\coordinate (III) at (intersection of 0,-1--H and -3.8,0--A);
\coordinate (IV) at (intersection of 0,-2.5--G and -3.8,0--A);
\coordinate (V) at (-2.5,-2.5);
\coordinate (VI) at (-2.5,-1);
\fill[green!10](I)--(V)--(VI)--(III)--(IV)--cycle;
\fill[orange!10](V)--(VI)--(II)--cycle;
\filldraw [gray] (I) circle (1pt);
\filldraw [gray] (II) circle (1pt);
\filldraw [gray] (-2.5,-2.5) circle (1pt);

\draw[<-, line width=1pt] (-4,0) node[left]{$\mu$}--(-4,-4);
\draw[<-, line width=1pt] (0,-4) node[below]{$\lambda$} --(-4,-4);
\path[draw,line width=1pt] (0,0) -- (-4,-4);

\path[draw,line width=0.7pt, dashed, gray] (-3.8,0) -- (A);
\path[draw,line width=0.7pt, dashed, gray] (-3.3,-3.3) -- (B);
\path[draw,line width=0.7pt, dashed, gray] (-2.5,-1) -- (C);
\path[draw,line width=0.7pt, dashed, gray] (-2.7,-2.7) -- (C');
\path[draw,line width=0.7pt, dashed, gray] (-1,0) -- (D);

\path[draw,line width=0.7pt, dashed, gray] (-3.8,-3.8) -- (E);
\path[draw,line width=0.7pt, dashed, gray] (-3.3,-3.3) -- (F);
\path[draw,line width=0.7pt, dashed, gray] (-2.5,-2.5) -- (G);
\path[draw,line width=0.7pt, dashed, gray] (-1,-1) -- (H);

\coordinate (VI) at (intersection of -3.3,-3.3--B and -3.3,-3.3--F);
\draw (-2.1,-1.6) coordinate(L)node[above]{$(\lambda, \mu)$};
\path[draw,line width=2pt, dotted, red] (L) -- (-2.7,-1.6);
\path[draw,line width=2pt, dotted, red] (-2.7,-1.6)--(-2.7,-2.7);
\path[draw,line width=2pt, dotted, red] (-2.7,-2.7)--(VI);

\draw (-3.5,-3) coordinate(LL)node[above]{$(\lambda, \mu)$};
\filldraw [cyan] (-3.5,-3) circle (1pt);
\path[draw,line width=2pt, blue] (LL) -- (-3.5,-3.5);
\path[draw,line width=2pt, blue] (-3.5,-3.5)--(VI);

\filldraw [black] (L) circle (1pt);
\filldraw [black] (LL) circle (1pt);
\filldraw [black] (VI) circle (1pt);

\end{tikzpicture}
\caption{The pale green region is $\Upsilon_1$ and the pale orange region is $\Upsilon_2$. The thick blue (solid) and red (dotted) paths show our definition of $\gamma$ when $(\lambda, \mu)$ is in $\Upsilon_1$ and $\Upsilon_2$ respectively.} 
\label{fig:regionplot}
\end{figure}
\vspace{-0.2 cm}

\begin{proof}
We define $\Upsilon = \Upsilon_1\cup \Upsilon_2$ where
\[ \Upsilon_1 = \{(\lambda,\mu) : \lambda\in(3/2-\sqrt{2}, 3/2+\sqrt{2}),\, \mu\in[\lambda, 10)\} \]
and
\[ \Upsilon_2 = \{(\lambda,\mu) : \mu\in(3/2+\sqrt{2}, 10),\, \lambda\in[3/2+\sqrt{2},\mu]\}. \]
Figure \ref{fig:regionplot} shows $\Upsilon_1$ and $\Upsilon_2$ in pale green and pale orange respectively.
We show that the statement of the Lemma holds for all the points ($\lambda, \mu) \in \Upsilon$ and that $\kappa(\lambda, \mu)<0$ if $(\lambda, \mu) \notin \Upsilon$.
It is easy to see that
\begin{equation}\label{posdiag}
\kappa(\lambda,\lambda )= -2\lambda+4 \sqrt{\lambda}-1>0 \,\,\,\,\text{ for }\,\,\,\, \lambda \in (3/2-\sqrt{2}, 3/2+\sqrt{2}).
\end{equation}
and this is concave as a function of $\lambda$. Since for $0<\lambda\le\mu$ we have
\[\frac{\partial^2 \kappa(\lambda, \mu)}{\partial \lambda^2} = -(1/2)(2\mu-\lambda)^{-3/2} - (2 \mu)\lambda^{-3/2} <0\]
and
\[\frac{\partial^2 \kappa(\lambda, \mu)}{\partial \mu^2} = -2(2\mu-\lambda)^{-3/2} - (1/2)\mu^{-3/2} <0,\]
the functions $\kappa(\cdot,\mu)$ on $(0,\mu]$ and $\kappa(\lambda,\cdot)$ on $[\lambda,\infty)$ are concave for each fixed $\lambda$ and $\mu$ respectively. This means that if we move parallel to either axis, we have the concave property; and so if for example $\kappa(\lambda_1,\mu)>0$ and $\kappa(\lambda_2,\mu)>0$ with $\lambda_1,\lambda_2\le \mu$, then $\kappa(\lambda,\mu)>0$ for all $\lambda\in[\lambda_1,\lambda_2]$.

We now take advantage of this concavity parallel to the axes. For every $(\lambda, \mu) \in \Upsilon_1$ such that $\kappa(\lambda, \mu)>0$ we choose $\gamma$ to be the union of the linear paths connecting $(1/2,1/2)$ to $(\lambda, \lambda)$ and then to $(\lambda, \mu)$.
Then clearly $\gamma$ satisfies \eqref{concave1} and \eqref{concave4}. Since $0<\lambda\le\mu\le 10$ throughout $\Upsilon$, the total length of the linear paths described is at most $20$, and therefore we may choose a time parameterization of $\gamma$ such that $|\gamma_X'(t)|\le 20$ and $|\gamma_Y'(t)|\le 20$, so that $\gamma$ satisfies \eqref{concave3}. We claim that $\gamma$ also satisfies \eqref{concave2}. Indeed, by \eqref{posdiag} $\kappa(\gamma(t))$ is positive on the first linear segment; in particular $\kappa(\lambda,\lambda)>0$, and since by assumption $\kappa(\lambda,\mu)>0$, by concavity parallel to the axes $\kappa(\gamma(t))$ is positive on the second linear segment too and $\kappa (\lambda,\mu) \ge \kappa_0$ where $\kappa_0 :=\min\{ \kappa(1/2,1/2), \kappa(\lambda, \lambda), \kappa(\lambda, \mu)\} >0$.

Now consider $(\lambda, \mu) \in \Upsilon_2$ such that $\kappa(\lambda, \mu)>0$. Since $\lambda \ge 3/2+\sqrt{2}/2$ and $\mu<10$, we have
\begin{equation}\label{derivpos}
 \frac{\partial \kappa(\lambda, \mu)}{\partial \lambda} = \frac{\mu}{\lambda^2} - \frac{1}{\sqrt{2\mu-\lambda}}-1 \le \frac{4 \mu}{(3+\sqrt{2})^2} - \frac{1}{\sqrt{2\mu-\lambda}}-1 <0,
\end{equation}
so $\kappa(\lambda', \mu)> \kappa(\lambda, \mu)>0$ for every $\lambda' \in [3/2+\sqrt{2}/2, \lambda]$. In particular, $\kappa (3/2+\sqrt{2}/2, \mu)>0$ and $(3/2+\sqrt{2}/2, \mu) \in \Upsilon_1$, so we can define $\gamma$ as the union of the linear paths connecting $(1/2,1/2)$ to $(3/2+\sqrt{2}/2, 3/2+\sqrt{2}/2)$, then to $(3/2+\sqrt{2}/2, \mu)$, and then to $(\lambda, \mu)$. Then as above, $\gamma$ clearly satisfies \eqref{concave1} and \eqref{concave4} and can be parameterized such that it satisfies \eqref{concave3}. Also $\kappa(\gamma(t))$ is positive on the first and second linear segments by the analysis of the $\lambda \in (3/2-\sqrt{2}, 3/2+\sqrt{2})$ case above, it is positive on the third linear segment by \eqref{derivpos} and $\kappa(\lambda, \mu) \ge \kappa_0$. Thus $\gamma$ satisfies \eqref{concave2} too.

To complete our proof, it remains to show that $\kappa(\lambda, \mu)<0$ for $(\lambda,\mu) \notin \Upsilon$. If $0< \lambda \le \mu \le 3/2-\sqrt{2}$, this follows from the fact that for every $\lambda \le 3/2-\sqrt{2}$,
\begin{equation}\label{kapneg1}
\kappa(\lambda, 3/2-\sqrt{2}) = -\frac{3/2-\sqrt{2}}{\lambda}-\lambda+2 \sqrt{3-2\sqrt{2}-\lambda}-\frac{3}{2} +\sqrt{2}+2 \sqrt{3/2-\sqrt{2}}<0,
\end{equation}
and for every $t \in [0,3/2-\sqrt{2}-\mu]$,
\begin{align*} 
\frac{d}{dt} \kappa(\lambda+t,\mu+t) &= \left( \frac{\partial \kappa(\lambda, \mu)}{\partial \lambda} + \frac{\partial \kappa(\lambda, \mu)}{\partial \mu} \right) \bigg\lvert_{(\lambda, \mu) =(\lambda+t, \mu+t)} \\
&=\frac{\mu-\lambda}{(\lambda+t)^2} + \frac{1}{\sqrt{2\mu+t-\lambda}} +\frac{1}{\sqrt{\mu+t}} -2 \\
&\ge \frac{1}{\sqrt{3/2-\sqrt{2}}} -2 >0.
\end{align*}
Secondly, if $0<\lambda\le 3/2-\sqrt 2 < \mu$, then we use \eqref{kapneg1} plus the fact that
\begin{equation*} 
\frac{\partial \kappa(\lambda, \mu)}{\partial \mu} = -\frac{1}{\lambda} + \frac{2}{\sqrt{2 \mu-\lambda}} + \frac{1}{\sqrt{\mu}} -1 \le -\frac{1}{3/2-\sqrt 2} + \frac{3}{\sqrt{3/2-\sqrt 2}}-1 < 0.
\end{equation*}
Finally, when $0< \lambda \le \mu$ and $\mu \ge 10$, the key fact is to observe that for every $\mu \ge 10$
\begin{equation}\label{eqn:mu_derivative} 
\frac{\partial \kappa(\lambda, \mu)}{\partial \mu} = -\frac{1}{\lambda} + \frac{2}{\sqrt{2 \mu-\lambda}} + \frac{1}{\sqrt{\mu}} -1 \le -\frac{1}{\mu} + \frac{3}{\sqrt{\mu}}-1 < 0.
\end{equation}
Since $\kappa(\lambda, 10)<0$ for $\lambda \le 10$, \eqref{eqn:mu_derivative} gives that $\kappa(\lambda, \mu)<0$ for any $\mu \ge 10$ and $\lambda \le 10$; and since $\kappa(\lambda,\lambda)<0$ for $\lambda\ge 10$, \eqref{eqn:mu_derivative} gives that $\kappa(\lambda,\mu)<0$ whenever $\mu\ge 10$ and $\lambda\ge 10$. 
This completes the proof.
\end{proof}

Take $f\in G_M^2$ such that $\frac{d}{dt} \tilde K(f,0,t)|_{t=0} > 0$ and $\tilde K(f,0,t)>0$ for all $t\in(0,1]$. Also fix $n\in\N$ and $m\in\N$ such that $n\ge m$. We now construct a function $h=h_{f,n,m}$ which depends on $n$ and $m$; we will later show that for $m$ sufficiently large (and $n$ even larger) the resulting function satisfies the properties of Proposition \ref{hfmnprop}.

Let $\tau = m^m \lceil n^{7/8}\rceil/n$. We will eventually choose $n$ much larger than $m$, so that $\tau$ is small. Also let $\lambda = f'_X(0)$ and $\mu = f'_Y(0)$. Take $\gamma$ as in Lemma \ref{kappaconcave} and for $j\in\{0,1,\ldots,m\}$ define
\begin{equation}\label{lambdajmuj}
\lambda_j = \lambda_j^{(m)} = \gamma_X(j/m), \hspace{8mm} \mu_j = \mu_j^{(m)} = \gamma_Y(j/m)\hspace{4mm}\text{ and }\hspace{4mm}\tau_j = \tau m^{j-m}.
\end{equation}

Begin by defining $h(s) = (s/2,s/2) = (\lambda_0 s,\mu_0 s)$ for $s\le \tau_0$. Then recursively, for each $j=1,\ldots,m$, suppose that $h(s)$ is defined for $s\le \tau_{j-1}$ and set
\[h(s) = h(\tau_{j-1}) + \big(\lambda_j(s-\tau_{j-1}),\,\mu_j(s-\tau_{j-1})\big) \,\,\,\, \text{ for } s\in(\tau_{j-1},\tau_j].\]
Also define
\[h(s) = h(\tau) + \big(f(2\tau)-h(\tau))\Big(\frac{s-\tau}{\tau}\Big) \,\,\,\,\text{ for } s\in(\tau,2\tau].\]
Finally, for each $j\in\{2\tau n,2\tau n+1,\ldots,n\}$ let $h(j/n)=f(j/n)$ and interpolate linearly between these values.

Note that, since $\tilde K$ has only downward jumps and $\tilde K(f,0,t)>0$ for all $t\in(0,1]$, we have $\inf_{s\in[\nu,1]} \tilde K(f,0,s) > 0$ for every $\nu>0$. Thus we may choose $\nu = \nu_{f,m}\in(0,1]$ such that
\begin{enumerate}[(a)]
\item $\|f(s)-(\lambda s,\mu s)\|\le s/m$ for all $s\le \nu$, \label{fnearlinearitem}
\item $\tilde K(f,s,t)>0$ for all $s\le\nu$ and $t\ge s$, and \label{Kpositem}
\item $\tilde K(f,\nu,1) \ge \tilde K(f,0,1)-1/m$. \label{Kapprox1item}
\end{enumerate}

\begin{lem}\label{hnearlinearlem}
Suppose that $\mu\ge\lambda>0$ and $\kappa(\lambda,\mu)>0$. For any $m\ge 2$, $j\in\{1,\ldots,m\}$ and any $s\in[\tau_{j-1},\tau_j]$,
\[\|h_{f,n,m}(s)-(\lambda_j s,\mu_j s)\| \le 40\tau_{j-1}/m.\]
Moreover, if $2\tau\le\nu$, then for any $s\in[\tau,\nu]$,
\[\|h_{f,n,m}(s)-(\lambda s,\mu s)\| \le 40 s/m.\]
\end{lem}

\begin{proof}
We begin by noting that for $s\in[\tau_{j-1},\tau_j]$,
\begin{equation}\label{hconst}
h(s)-(\lambda_j s,\mu_j s) = h(\tau_{j-1})-(\lambda_{j}\tau_{j-1},\mu_{j}\tau_{j-1}),
\end{equation}
so 
for the first part of the lemma it suffices to show that for any $j\in\{1,\ldots,m\}$,
\begin{equation}\label{lammuj}
\|h(\tau_{j-1})-(\lambda_{j}\tau_{j-1},\mu_{j}\tau_{j-1})\|\le 40\tau_{j-1}/m.
\end{equation}
We prove \eqref{lammuj} by induction. Recall that for each $j$, $\lambda_j = \gamma_X(j/m)$, and by Lemma \ref{kappaconcave} \eqref{concave3}, $|\lambda_{j-1} - \lambda_{j}|\le 20/m$ and $\lvert \mu_{j}- \mu_{j-1} \lvert \le 20/m$. Thus we first have
\[\|h(\tau_0)-(\lambda_1 \tau_0,\mu_1 \tau_0)\| = \max \{ \lvert \lambda_0-\lambda_1 \lvert \tau_0, \lvert \mu_0-\mu_1 \lvert\tau_0 \} \le 20 \tau_0/m.\]
Suppose that $j\in\{1,\ldots,m-1\}$ and \eqref{lammuj} holds for $j$. By the triangle inequality,
\[|h_X(\tau_j)-\lambda_{j+1}\tau_j| \le |h_X(\tau_j) - \lambda_j\tau_j| + |\lambda_j\tau_j - \lambda_{j+1}\tau_j|\]
and then by \eqref{hconst}, this equals
\[|h_X(\tau_{j-1}) - \lambda_j\tau_{j-1}| + |\lambda_j-\lambda_{j+1}|\tau_j.\]
Applying \eqref{lammuj} and using the fact that $|\lambda_{j-1} - \lambda_{j}|\le 20/m$, we obtain that
\[|h_X(\tau_j)-\lambda_{j+1}\tau_j| \le \frac{40\tau_{j-1}}{m} + \frac{20\tau_j}{m} \le \frac{40\tau_j}{m}.\]
By symmetry we also have $|h_Y(\tau_{j})-\mu_{j+1} \tau_{j}|\le 40\tau_j/m$. Hence, by induction, \eqref{lammuj} holds for all $j\in\{1,\ldots,m\}$, proving the first part of the lemma.

For the second part, suppose that $2\tau\le \nu$. Note first that for $s\in[\tau,2\tau]$, $h$ is linear and therefore
\[\|h(s)-(\lambda s,\mu s)\| \le \max\big\{\|h(\tau)-(\lambda \tau,\mu \tau)\|,\,\|h(2\tau)-(2\lambda\tau,2\mu\tau)\|\big\}.\]
By the first part of the lemma,
\[\|h(\tau)-(\lambda \tau,\mu \tau)\|\le 40\tau_{m-1}/m \le 40\tau/m\]
and since $h(2\tau) = f(2\tau)$, by property (\ref{fnearlinearitem}) of $f$,
\[\|h(2\tau)-(2\lambda \tau,2\mu \tau)\| \le 2\tau/m.\]
This proves the second part of the lemma for $s\in[\tau,2\tau]$; for $s\in [2\tau,\nu]$, we note that $h$ linearly interpolates between values of $f$, and therefore for $j$ such that $s\in[j/n,(j+1)/n]$,
\[\|h(s)-(\lambda s,\mu s)\| \le \max\big\{\|f(\sfrac{j}{n})-(\lambda \sfrac{j}{n},\mu \sfrac{j}{n})\|,\,\|f(\sfrac{j+1}{n})-(\lambda \sfrac{j+1}{n},\mu \sfrac{j+1}{n})\|\big\} \le \frac{j+1}{nm} \le \frac{s+1/n}{m} \le \frac{2s}{m}\]
where we used (\ref{fnearlinearitem}) and the fact that $2\tau\ge 1/n$.
\end{proof}

\begin{cor}\label{hnearlinearcor}
Suppose that $\mu\ge\lambda>0$, $\kappa(\lambda,\mu)>0$ and $m\ge 1600$, $m\in\N$. Then for any $j\in\{1,\ldots,m\}$ and any $s\in[\tau_{j-1},\tau_j]$,
\[\frac{\mu_j}{\lambda_j}\Big(1-\frac{1600}{m}\Big) \le R^*(h_{f,n,m}(s)) \le \frac{\mu_j}{\lambda_j}\Big(1+\frac{3200}{m}\Big),\]
and if $2\tau\le\nu$ then for any $s\in[\tau,\nu]$
\[\frac{\mu}{\lambda}\Big(1-\frac{1600}{m}\Big) \le R^*_X(h_{f,n,m}(s))+1/2 \le R^*(h_{f,n,m}(s)) \le \frac{\mu}{\lambda}\Big(1+\frac{3200}{m}\Big).\]
\end{cor}

\begin{proof}
We begin with the lower bound on $R^*(h(s))$ for $s\in [\tau_{j-1},\tau_j]$. Since $\mu\ge\lambda$, we have $h_Y(s)\ge h_X(s)$, and therefore by the first part of Lemma \ref{hnearlinearlem},
\[R^*(h(s)) = \frac{h_Y(s)}{h_X(s)} \ge \frac{\mu_j s - 40\tau_{j-1}/m}{\lambda_j s + 40\tau_{j-1}/m} = \frac{\mu_j}{\lambda_j}\Big(\frac{1 - 40\tau_{j-1}/(m\mu_j s)}{1 + 40\tau_{j-1}/(m\lambda_j s)}\Big).\]
Using the fact that $s\ge\tau_{j-1}$, and then that $1/(1+x)\ge 1-x$ for $x\ge0$, this is at least
\[\frac{\mu_j}{\lambda_j}\Big(1 - \frac{40}{m\mu_j}\Big)\Big(1 - \frac{40}{m\lambda_j}\Big).\]
By Lemma \ref{kappaconcave} \eqref{concave4}, we have $\lambda_j\ge 3/2-\sqrt{2} \ge 1/20$ and similarly for $\mu_j$, so the above is at least $\frac{\mu_j}{\lambda_j}(1-1600/m)$, 
and the first lower bound on $R^*(h(s))$ follows. The first upper bound is similar, using that $1/(1-x) \le 1+2x$ for $x\in[0,1/2]$; since $m\ge 1600$ and $\lambda_j\ge 1/20$ we have $40/(m\lambda_j) \le 1/2$, and we obtain
\[R^*(h(s)) \le \frac{\mu_j}{\lambda_j}\Big(1 + \frac{40}{m\mu_j}\Big)\Big(1 + \frac{80}{m\lambda_j}\Big);\]
then since $\lambda_j\ge 1/20$, $\mu_j\ge 1/20$ and $m\ge 1600$, the product of the last two terms reduces to the desired form.

The proof of the second part of the corollary, when $s\in[\tau,\nu]$, is almost identical. Indeed, if $h_Y(s)\ge h_X(s)$ then $R^*_X(h(s))+1/2 = R^*(h(s))$ and we use the same argument but apply the second part of Lemma \ref{hnearlinearlem} rather than the first part. The same applies to the lower bound even when $h_Y(s)<h_X(s)$, since in any case $R^*_X(h(s)) +1/2 \ge h_Y(s)/h_X(s)$. However, we have to make a slight modification to the upper bound when $h_Y(s)<h_X(s)$; in this case, we instead have $R^*_X(h(s))+1/2 \le R^*(h(s))$ where
\[R^*(h(s)) = \frac{h_X(s)}{h_Y(s)}, \]
and then the argument above gives
\[R^*(h(s)) = \frac{h_X(s)}{h_Y(s)} \le \frac{\lambda s + 40s/m}{\mu s - 40s/m} \le \frac{\lambda}{\mu}\Big(1+\frac{40}{\lambda m}\Big)\Big(1+\frac{80}{\mu m}\Big) \le \frac{\lambda}{\mu}\Big(1+\frac{3200}{m}\Big).\]
However, since $\lambda\le\mu$, we have $\lambda/\mu \le 1 \le \mu/\lambda$ and so the same conclusion holds. 
\end{proof}

\begin{cor}\label{Ktildekappacor}
Suppose that $\mu\ge\lambda>0$, $\kappa(\lambda,\mu)>0$ and $m\ge 1600$, $m\in\N$. There exists a finite constant $C$ such that for any $j\in\{1,\ldots,m\}$ and any $s,t\in[\tau_{j-1},\tau_j]$ with $s\le t$, we have
\[\tilde K(h_{f,n,m},s,t) \ge \kappa(\lambda_j,\mu_j)(t-s) - \frac{C}{m}(t-s),\]
and if $2\tau\le\nu$ then for any $s,t\in[\tau,\nu]$ with $s\le t$, we have
\[\tilde K(h_{f,n,m},s,t) \ge \kappa(\lambda,\mu)(t-s) - \frac{C}{m}(t-s).\]
\end{cor}

\begin{proof}
We begin with the first statement. Using \eqref{Kalt}, since $h(t)-h(s) = (\lambda_j(t-s),\mu_j(t-s))$, we have
\begin{multline*}
\tilde K(h,s,t) = -\int_s^t R^*(h(u))du + 2\sqrt2 \int_s^t \sqrt{R^*_X(h(u))\lambda_j}\, du\\
+ 2\sqrt2 \int_s^t \sqrt{R^*_Y(h(u))\mu_j}\, du - \lambda_j(t-s) - \mu_j(t-s).
\end{multline*}
Since $\mu\ge\lambda$, we have $h_Y(u)\ge h_X(u)$ for all $u\le\tau$ and therefore $R^*_X(h(u))=R^*(h(u))-1/2$ and $R^*_Y(h(u))=1/2$ for all $u\le\tau$. Thus
\begin{multline}\label{tildeKnearlinear1}
\tilde K(h,s,t) = -\int_s^t R^*(h(u))du + 2\sqrt2 \int_s^t \sqrt{(R^*(h(u))-1/2)\lambda_j}\, du\\
+ 2\sqrt{\mu_j} (t-s) - \lambda_j(t-s) - \mu_j(t-s).
\end{multline}
By Corollary \ref{hnearlinearcor}, for any $u\in[s,t]$ we have
\[ \frac{\mu_j}{\lambda_j}\Big(1-\frac{1600}{m}\Big) \le R^*(h(u)) \le \frac{\mu_j}{\lambda_j}\Big(1+\frac{3200}{m}\Big)\]
and, using also that $(1-x)^{1/2}\ge 1-x$ for $x\in[0,1)$,
\begin{align*}
\sqrt{(R^*(h(u))-1/2)} &\ge \bigg(\frac{\mu_j}{\lambda_j}\Big(1-\frac{1600}{m}\Big) - 1/2\bigg)^{1/2}\\
&= \Big(\frac{\mu_j}{\lambda_j}-\frac{1}{2}\Big)^{1/2}\Big(1-\frac{1600\mu_j}{m\lambda_j(\mu_j/\lambda_j-1/2)}\Big)^{1/2}\\
&\ge \Big(\frac{\mu_j}{\lambda_j}-\frac{1}{2}\Big)^{1/2}\Big(1-\frac{3200}{m}\Big).
\end{align*}
Substituting these estimates into \eqref{tildeKnearlinear1}, we have
\begin{multline*}
\tilde K(h,s,t) \ge -\frac{\mu_j}{\lambda_j}\Big(1+\frac{3200}{m}\Big)(t-s) + 2\sqrt2 \Big(\frac{\mu_j}{\lambda_j}-\frac{1}{2}\Big)^{1/2}\lambda_j^{1/2}\Big(1-\frac{3200}{m}\Big)(t-s)\\
+ 2 \sqrt{\mu_j} (t-s) - \lambda_j(t-s) - \mu_j(t-s).
\end{multline*}
Recognising that
\[\kappa(\lambda_j,\mu_j) = -\frac{\mu_j}{\lambda_j} + 2\sqrt2\Big(\frac{\mu_j}{\lambda_j}-\frac{1}{2}\Big)^{1/2}\lambda_j^{1/2} + 2\mu_j^{1/2}-\lambda_j-\mu_j,\]
we see that
\[\tilde K(h,s,t) \ge \kappa(\lambda_j,\mu_j) (t-s)- \frac{3200\mu_j}{\lambda_j m}(t-s) - 2\sqrt2 \Big(\frac{\mu_j}{\lambda_j}-\frac{1}{2}\Big)^{1/2}\lambda_j^{1/2}\frac{3200}{m}(t-s)\]
and the first part of the result follows using Lemma \ref{kappaconcave} \eqref{concave4}.

The proof of the second part is almost identical, though since for $u\in[\tau,\nu]$ we do not have exactly $h'_X(u) = \lambda$ and $h'_Y(u)=\mu$, we must additionally use the bounds
\[h'_X(u) = \frac{f(2\tau)-h(\tau)}{\tau} \le \frac{2\lambda \tau + 2\tau/m^2 - \lambda\tau + 40\tau/m}{\tau} \le \lambda + \frac{42}{m}\]
and similarly
\[h'_X(u) \ge \lambda - \frac{42}{m}\]
and
\[\mu - \frac{42}{m} \le h'_Y(u) \le \mu + \frac{42}{m}.\]
With the addition of these estimates, the proof proceeds as before.
\end{proof}

Corollary \ref{Ktildekappacor} essentially guarantees that $\tilde K(h_{f,n,m},s,t)$ is positive for $0\le s<t\le \nu$, provided that $\kappa(\lambda,\mu)>0$. We now need to show that $\tilde K(h_{f,n,m},\nu,t)$ is not too negative for $t\ge\nu$. The following result will be used to check that $\tilde K(f,\nu,t)$ is closely approximated by $\tilde K(h_{f,n,m},\nu,t)$.

\begin{prop}\label{lowersemicts}
Suppose that $0\le s\le t\le 1$ and that $f\in G_M^2$. Let $f_n$ be the function in $\PL_n$ constructed by setting $f_n(j/n)=f(j/n)$ for each $j=0,\ldots,n$ and interpolating linearly. Then
\[\liminf_{n\to\infty}\tilde K(f_n,s,t) \ge \tilde K(f,s,t).\]
\end{prop}

We prove this in Appendix \ref{lsc_sec}. Later, in Proposition \ref{uppersemicts}, we will also show that the opposite inequality holds in certain circumstances. We now have the pieces in place to prove Proposition \ref{hfmnprop}.

\begin{proof}[Proof of Proposition \ref{hfmnprop}]
As usual let $\lambda = f'_X(0)$, $\mu = f'_Y(0)$ and $\tau = m^m\lceil n^{7/8}\rceil/n$, with $\lambda_j$ and $\mu_j$ as in \eqref{lambdajmuj} and $\tau_j = \tau m^{j-m}$, for $j\in\{0,1,\ldots,m\}$. We will check that $h_{f,n,m}$ satisfies the desired properties when $m$ and $n$ are sufficiently large. Without loss of generality we assume that $\mu\ge\lambda$.

Since $\tau_0 n$ is an integer we have $h_{f,n,m}\in\PL_n^2$, and since $f\in G_M^2$ it is easy to see that $h_{f,n,m}\in G_M^2$ too. Since $\|h_{f,n,m}(s)\|\le Ms$ and $\|f(s)\|\le Ms$ for $s\le 2\tau = 2m^m\lceil n^{7/8}\rceil/n$, and $h_{f,n,m}(j/n)=f(j/n)$ for $j\ge 2\tau n$, by choosing $n$ large enough that $2\tau M\le \eps$ we have $h_{f,n,m}\in B(f,\eps)$. This proves that $h_{f,n,m}$ satisfies \eqref{hfmnprop1} when $n$ is large.

For \eqref{hfmnprop2}, note first that $\tau_0 = \lceil n^{7/8}\rceil/n$. Take $n$ large enough that $2\tau<\nu$. Then we claim that since $\lim_{t\to 0}\tilde K(f,0,t)/t > 0$, we have $\kappa(\lambda,\mu)>0$. To see why the claim holds, for small $s$ we have $\|f'(s)-(\lambda,\mu)\|\le 1/m$ and $\|f(s)-(\lambda s,\mu s)\|\le s/m$. The same argument as in Corollary \ref{hnearlinearcor} then shows that for some finite constant $C$,
\[\frac{\mu}{\lambda}\Big(1-\frac{C}{m}\Big) \le R^*_X(f(s)) + \frac{1}{2} \le R^*(f(s)) \le \frac{\mu}{\lambda}\Big(1+\frac{C}{m}\Big)\]
and plugging these estimates into \eqref{Kalt} with $a=0$ and $b=t$ and using standard approximations shows that $\tilde K(f,0,t) \le \kappa(\lambda,\mu)t + C't/m$ for some finite constant $C'$. This implies the claim.
 
 Since $\kappa(\lambda, \mu)>0$, by Lemma \ref{kappaconcave} we may choose $\kappa_0>0$ such that $\kappa(\gamma(t))\ge \kappa_0$ for all $t\in[0,1]$, and then $\kappa(\lambda_j,\mu_j)\ge \kappa_0$ for all $j\in\{0,\ldots,m\}$. Corollary \ref{Ktildekappacor} then tells us that for $s\in[\tau_{j-1},\tau_j]$ we have
\begin{align*}
\tilde K(h_{f,n,m},\tau_0,s) &\ge \sum_{i=1}^{j-1}(\kappa(\lambda_i,\mu_i)-C/m)(\tau_i-\tau_{i-1}) + (\kappa(\lambda_j,\mu_j)-C/m)(s-\tau_{j-1})\\
&\ge (\kappa_0-C/m)(s-\tau_0),
\end{align*}
and for $s\in[\tau,\nu]$ we have
\begin{align*}
\tilde K(h_{f,n,m},\tau_0,s) &\ge \sum_{i=1}^{m}(\kappa(\lambda_i,\mu_i)-C/m)(\tau_i-\tau_{i-1}) + (\kappa(\lambda,\mu)-C/m)(s-\tau)\\
&\ge (\kappa_0-C/m)(s-\tau_0).
\end{align*}
Thus, by choosing $m$ large enough, we may ensure that $\tilde K(h_{f,n,m},\tau_0,s)\ge \kappa_0(s-\tau_0)/2$ for all $s\in[\tau_0,\nu]$.

For $s>\nu$, by the above argument we have
\begin{equation}\label{finalKeq}
\tilde K(h_{f,n,m},\tau_0,s) \ge \tilde K(h_{f,n,m},\tau_0,\nu) + \tilde K(h_{f,n,m},\nu,s) \ge \kappa_0(\nu-\tau_0)/2 + \tilde K(h_{f,n,m},\nu,s),
\end{equation}
and since $\kappa_0(\nu-\tau_0)/2$ increases to $\kappa_0\nu/2$ as $n\to\infty$ and $\tilde K(f,\nu,s)>0$ by \eqref{Kpositem}, to show \eqref{hfmnprop2} it suffices to show that for large $n$,
\[\tilde K(h_{f,n,m},\nu,s) \ge \tilde K(f,\nu,s) - \kappa_0\nu/4.\]
But since $h$ is the piecewise linear interpolation of $f$ on the interval $[\nu,1]$, this follows from Proposition \ref{lowersemicts}.

Finally, for \eqref{hfmnprop3}, applying \eqref{finalKeq} with $s=1$, we certainly have
\[\tilde K(h_{f,n,m},\tau_0,1) \ge \tilde K(h_{f,n,m},\nu,1);\]
by Proposition \ref{lowersemicts} the right-hand side converges to $\tilde K(f,\nu,1)$ as $n\to\infty$; and by \eqref{Kapprox1item} we know that $\tilde K(f,\nu,1) \ge \tilde K(f,0,1)-1/m$. This completes the proof.
\end{proof}

\section{Coupling $\xi^T$ with simpler processes}\label{coupling_sec}

One problem we face is that $\xi_X$ and $\xi_Y$ are not independent, because their jump rates at time $t$ are functions of the pair $(\xi_X(t),\xi_Y(t))$. However, if we already know that $\xi^T$ has remained near a fixed function $f$, then the jump rates are ``almost deterministic'' and therefore $\xi_X$ and $\xi_Y$ are ``almost independent''. In order to take advantage of this idea, we will construct new processes $Z_+$ and $Z_-$ which have the maximal and minimal jump rates (respectively) that $\xi^T$ may have if it remains near $f$. We will couple these processes with another process, $Z$, which will have the same distribution as $\xi^T$ but will be trapped between $Z_+$ and $Z_-$, as long as $Z$ remains near $f$.\\
Recall the definitions of $R_X^-(I,F,T)$, $R_X^+(I,F,T)$, $R_Y^-(I,F,T)$, $R_Y^+(I,F,T)$, $|I|$, $I^+$, $I^-$, $x^-(s,F)$, $x^+(s,F)$, $y^-(s,F)$, $y^+(s,F)$, $\Gamma_{M,T}(f,n)$ and $I_j$ from Section \ref{probabilistic_ingredients}. In what follows, the reader can think of the case $I=I_j$ and $F=\Gamma_{M,T}(f,n)$ for some function $f$.

Let
\[V(I,F) = [x^-(I^-,F),x^+(I^-,F)]\times [y^-(I^-,F),y^+(I^-,F)].\]
Take $z=(x,y)\in V(I,F)$. Under a probability measure $Q_z = Q^{I,F,T}_z$, let $(X_+(I^-+s), s\in |I|)$ be a compound Poisson process started from $x$ with rate $2R_X^+(I,F,T)T$ and jumps that are exponentially distributed with parameter $T$, and let $(Y_+(I^-+s), s\in |I|)$ be an independent compound Poisson process started from $y$ with rate $2R_Y^+(I,F,T)T$ and jumps that are exponentially distributed with parameter $T$. Let $Z_+ = (X_+,Y_+)$.

We now construct---again under $Q_z^{I,F,T}$---two more (pure jump) processes $Z(I^-+s)$ and $Z_-(I^-+s)$ for $s\in |I|$ recursively as follows. Start by setting $Z(I^-)=z$ and $Z_-(I^-) = z$. The jumps of both $Z$ and $Z_-$ are subsets of the jumps of $Z_+$. Suppose that $Z_+$ has a jump at time $s$, and that $Z(s-) = z'$. Let $U$ be an independent Uniform$[0,1]$ random variable. Since $X_+$ and $Y_+$ are independent, exactly one of $X_+$ or $Y_+$ jumps at time $s$. Suppose for a moment that $X_+$ has a jump of size $x'>0$. Then accept the jump for $Z$ if $U\le R_X(z')/R_X^+(I,F,T)$ and reject it otherwise; in other words, set $Z(s) = z'+(x',0)$ with probability $R_X(z')/R_X^+(I,F,T)$ and $Z(s) = z'$ otherwise. Accept the jump for $Z_-$ if $U\le R_X^-(I,F,T)/R_X^+(I,F,T)$ and reject it otherwise. Similarly, if $Y_+$ has a jump of size $y>0$, then accept the jump for $Z$ if $U\le R_Y(z')/R_Y^+(I,F,T)$, and accept it for $Z_-$ if $U\le R_Y^-(I,F,T)/R_Y^+(I,F,T)$.

Recall that for $F\subset E^2$, $g\in E$ and an interval $I\subset[0,1]$, we say that $g|_I\in F|_I$ if there exists a function $h\in F$ such that $h(u)=g(u)$ for all $u\in I$. Let
\[\mathcal A_\xi(I,F,T) = \big\{\xi^T|_I \in F|_I\big\}\]
and
\[\mathcal A(I,F,T) = \big\{Z|_I \in F|_I\big\}.\]

Note that for any $z\in V(I,F)$, on the event $\mathcal A(I,F,T)$, under $Q_z^{I,F,T}$ we always have
\[R_X(Z(s))\in [R_X^-(I,F,T),R_X^+(I,F,T)] \hspace{4mm}\text{ and }\hspace{4mm} R_Y(Z(s))\in [R_Y^-(I,F,T),R_Y^+(I,F,T)],\]
for all $s\in I$. Thus, by our construction:
\begin{enumerate}[(i)]
\item under $Q^{I,F,T}_z$, on the event $\mathcal A(I,F,T)$, we have $X_-(s) \le X(s)\le X_+(s)$ and $Y_-(s)\le Y(s)\le Y_+(s)$ for all $s\in I$;
\item the process $(Z(s)\ind_{\mathcal A(I\cap[0,s],F,T)})_{s\in I}$ under $Q_z^{I,F,T}$ is equal in distribution to the process\\
$(\xi^T(s)\ind_{\mathcal A_\xi(I\cap[0,s],F,T)})_{s\in I}$ conditionally on $\xi^T(I^-)=z$ under $\Q$;
\item under $Q^{I,F,T}_z$, the processes $(X_-,X_+)$ and $(Y_-,Y_+)$ are independent.
\setcounter{enumeratecounter}{\value{enumi}}
\end{enumerate}
Furthermore, by the thinning property of Poisson processes,
\begin{enumerate}[(i)]
\setcounter{enumi}{\value{enumeratecounter}}
\item under $Q^{I,F,T}_z$, the processes $X_-$ and $X_+-X_-$ are independent, as are $Y_-$ and $Y_+-Y_-$.
\end{enumerate}

\subsection{Applying the coupling to the upper bound: proof of Propositions \ref{Qprobcor} and \ref{Qprobcor2}}

Recall the terminology ``$X+$ case'' and ``$X-$ case'' from Section \ref{probabilistic_ingredients}, and the definitions of $\mathcal E^+_X(I,F,T)$ and $\mathcal E^+_Y(I,F,T)$.
The main part of the proof of Proposition \ref{Qprobcor} is the following lemma.

\begin{lem}\label{QtoEub}
Suppose that $F\subset E^2$ and $T>1$. Then for any $I\subset[0,1]$ and $z\in V(I,F)$,
\[\Q\big(\mathcal A_\xi(I,F,T) \,\big|\,\xi^T(I^-)=z\big) \le \exp\big(-T\mathcal E^+_X(I,F,T)-T\mathcal E^+_Y(I,F,T)\big).\]
\end{lem}

\begin{proof}
For $z\in V(I,F)$, using (ii), (i) and (iii) in that order,
\begin{align*}
\Q\big(\mathcal A_\xi(I,F,T) \,\big|\,\xi^T(I^-)=z\big) &= Q^{I,F,T}_z\big(\mathcal A(I,F,T) \big)\\
&\le Q^{I,F,T}_z\big(X_-(I^+)\le x^+(I^+,F),\; Y_-(I^+)\le y^+(I^+,F) \big)\\
&= Q^{I,F,T}_z\big(X_-(I^+)\le x^+(I^+,F)\big)Q^{I,F,T}_z\big(Y_-(I^+)\le y^+(I^+,F) \big).
\end{align*}
We will apply this bound when we are in the $X-$ and $Y-$ cases. Of course, we were not forced to concentrate on the two upper boundaries $x^+(I^+,F)$ and $y^+(I^+,F)$, and by considering the other permutations of boundaries we obtain upper bounds on the same quantity of the form
\[Q^{I,F,T}_z\big(X_+(I^+)\ge x^-(I^+,F)\big)Q^{I,F,T}_z\big(Y_+(I^+)\ge y^-(I^+,F) \big),\]
\[Q^{I,F,T}_z\big(X_+(I^+)\ge x^-(I^+,F)\big)Q^{I,F,T}_z\big(Y_-(I^+)\le y^+(I^+,F) \big)\]
and
\[Q^{I,F,T}_z\big(X_-(I^+)\le x^+(I^+,F)\big)Q^{I,F,T}_z\big(Y_+(I^+)\ge y^-(I^+,F) \big)\]
which we can apply in other cases as appropriate.
Now, for any $\lambda>0$, by Markov's inequality,
\begin{align*}
Q^{I,F,T}_z\big(X_-(I^+)\le x^+(I^+,F)\big) & = Q^{I,F,T}_z\big(e^{-\lambda X_-(I^+)} \ge e^{-\lambda x^+(I^+,F)}\big)\\
&\le Q^{I,F,T}_z\big[e^{-\lambda(X_-(I^+)-X_-(I^-))}\big]e^{\lambda(x^+(I^+,F) - x^-(I^-,F))}\\
&= \exp\Big(-2R^-_X(I,F,T)T|I|\frac{\lambda}{T+\lambda}+\lambda (x^+(I^+,F) - x^-(I^-,F))\Big).
\end{align*}
In the $X-$ case we have $2R^-_X(I,F,T)|I| > x^+(I^+,F) - x^-(I^-,F)$, so we can choose the optimal value
\[\lambda = T\sqrt{\frac{2R^-_X(I,F,T)|I|}{x^+(I^+,F) - x^-(I^-,F)}} - T >0.\]
Simplifying gives
\[Q^{I,F,T}_z\big(X_-(I^+)\le x^+(I^+,F)\big)\le \exp\Big(-T\Big(\sqrt{2R^-_X(I,F,T)|I|}-\sqrt{x^+(I^+,F) - x^-(I^-,F)}\Big)^2\Big)\]
which equals $\exp\big(-T\mathcal E^+_X(I,F,T)\big)$ in the $X-$ case. Similarly, in the $X+$ case, by using
\[Q^{I,F,T}_z\big(X_+(I^+)\ge x^-(I^+,F)\big) \le Q^{I,F,T}_z[\exp\big(\mu(X_+(I^+)-X_+(I^-))\big) - \mu(x^-(I^+,F) - x^+(I^-,F))]\]
for $\mu>0$, we obtain
\begin{align*}
Q^{I,F,T}_z\big(X_+(I^+)\ge x^-(I^+,F)\big) &\le \exp\Big(-T\Big(\sqrt{2R^+_X(I,F,T)|I|}-\sqrt{x^-(I^+,F) - x^+(I^-,F)}\Big)^2\Big)\\
&= \exp\big(-T\mathcal E^+_X(I,F,T)\big)
\end{align*}
and when we are in neither the $X-$ nor $X+$ case we can use a trivial upper bound of $1$. By symmetry we obtain the same bounds in terms of $Y$. Applying these bounds in the appropriate cases completes the proof.
\end{proof}

Our main results in this section are now easy corollaries of Lemma \ref{QtoEub}.

\begin{proof}[Proof of Proposition \ref{Qprobcor}]
Recall that $I_j = [j/n,(j+1)/n]$ and let $V(j) = V(I_j,\Gamma_{M,T}(f,n))$. Note that the restrictions on $z$ ensure that $z\in V(i)$, and therefore by the Markov property,
\begin{align*}
\Q \big(\xi^T|_{[i/n,\theta]}\in\Gamma_{M,T}(f,n)\big|_{[i/n,\theta]}\,\big|\,\xi^T_{i/n} = z\big) &\le \prod_{j=i}^{\lfloor \theta n\rfloor-1} \sup_{z'\in V(j)} \Q\big(\xi^T|_{I_j}\in\Gamma_{M,T}(f,n)|_{I_j} \,\big|\, \xi^T_{j/n} = z'\big)\\
&= \prod_{j=i}^{\lfloor \theta n\rfloor -1} \sup_{z'\in V(j)} \Q\big(\mathcal A_\xi(I_j,\Gamma_{M,T}(f,n),T) \,\big|\, \xi^T_{j/n} = z'\big).
\end{align*}
The result now follows from Lemma \ref{QtoEub}.
\end{proof}




\begin{proof}[Proof of Proposition \ref{Qprobcor2}]
Let $i=\lfloor an\rfloor$ and $\ell=\lceil bn\rceil$. Let $V_i = \{w : \|w-f(a)\|<1/n^2\}$, and for $j \in\{i+1,\ldots,\ell\}$ let $V_j = \{w : \|w-f(j/n)\|<1/n^2\}$. Note that, by the Markov property,
\begin{align*}
\Q \big(\xi^T|_{[a,b]}\in\Lambda_{M,T}(f,n)\big|_{[a,b]}\,\big|\,\xi^T_a = z\big) &\le \prod_{j=i}^{\ell-1} \sup_{w\in V_j} \Q\big(\xi^T|_{I_j\cap[a,b]}\in\Lambda_{M,T}(f,n)\big|_{I_j\cap[a,b]}\,\big|\,\xi^T_{j/n} = w\big)\\
&= \prod_{j=i}^{\ell-1} \sup_{w\in V_j} \Q\big(\mathcal A_\xi(I_j\cap[a,b],\Lambda_{M,T}(f,n),T) \,\big|\, \xi^T_{j/n} = w\big).
\end{align*}
The result now follows from Lemma \ref{QtoEub}, together with \eqref{LambdaGamma} and \eqref{Rtrap}.
\end{proof}

\subsection{Applying the coupling to the lower bound: proofs of Lemmas \ref{BigQtolittleqprod}, \ref{qhatlb} and \ref{qlblem}}\label{CPP_sec}

We begin this section with the proof of Lemma \ref{BigQtolittleqprod}, which links the probability that we want to bound with our coupled compound Poisson processes.

\begin{proof}[Proof of Lemma \ref{BigQtolittleqprod}]
We begin by splitting $[0,1]$ into its subintervals $I_j$, $j=0,\ldots,n-1$. By applying the Markov property at each time $j/n$,
\begin{align*}
&\Q\Big(\xi^T|_{[k/n,1]}\in \Lambda_{M,T}(f,n)|_{[k/n,1]}\,\Big|\, \xi^T(k/n) = w\Big)\\
& \ge \Q\Big( \|\xi^T(s)-f(s)\|<\sfrac{1}{n^2}\;\;\forall s\in I_j, \; \xi^T(\sfrac{j+1}{n})\in \mathcal Z_{j+1}, \\
&\hspace{50mm} \xi^T|_{I_j}\in G_{M,T}^2|_{I_j} \;\; \forall j\in\{k,\ldots,n-1\}\,\Big|\, \xi^T(\sfrac{k}{n}) = w\Big)\\
&\ge \prod_{j=k}^{n-1} \inf_{z \in \mathcal Z_j} \Q\Big( \big\|\xi^T(s)-f(s)\big\| <1/n^2 \;\; \forall s\in I_j, \; \xi^T(\sfrac{j+1}{n})\in \mathcal Z_{j+1}, \; \xi^T|_{I_j}\in G_{M,T}^2|_{I_j} \,\Big|\, \xi^T(\sfrac{j}{n})=z\Big).
\end{align*}
It therefore remains to show that for each $j$ and any $z\in\mathcal Z_j$,
\begin{multline}\label{Qlbremains}
\Q\Big( \big\|\xi^T(s)-f(s)\big\| <1/n^2 \;\; \forall s\in I_j, \; \xi^T(\sfrac{j+1}{n})\in \mathcal Z_{j+1}, \; \xi^T|_{I_j}\in G_{M,T}^2|_{I_j} \,\Big|\, \xi^T(\sfrac{j}{n})=z\Big)\\
\ge q^X_{n,M,T}(z,j,f)\, \hat q^X_{n,M,T}(z,j,f)\, q^Y_{n,M,T}(z,j,f)\, \hat q^Y_{n,M,T}(z,j,f).
\end{multline}

We now use the coupling from Section \ref{coupling_sec}, with $I=I_j$ and $F=\Lambda_{M,T}(f,n)$. We simply write $Q_z$ as shorthand for $Q_z^{I_j,\Lambda_{M,T}(f,n),T}$. By property (ii) of the coupling, we have
\begin{multline*}
\Q\Big( \big\|\xi^T(s)-f(s)\big\| <1/n^2 \;\; \forall s\in I_j, \; \xi^T(\sfrac{j+1}{n})\in \mathcal Z_{j+1}, \; \xi^T|_{I_j}\in G_{M,T}^2|_{I_j} \,\Big|\, \xi^T(\sfrac{j}{n})=z\Big)\\
= Q_z\Big( \big\|Z(s)-f(s)\big\| \le 1/n^2\;\;\forall s\in I_j, \;  Z(\sfrac{j+1}{n})\in \mathcal Z_{j+1}, \; Z|_{I_j}\in G_{M,T}^2|_{I_j}\Big)
\end{multline*}
which, by property (i), is at least
\[Q_z\Big(\big\|Z_-(s)- f(s)\big\| \le \sfrac{1}{n^2} \,\,\forall s\in I_j, \, Z_-(\sfrac{j+1}{n})\in \mathcal Z_{j+1}, \, Z_-|_{I_j}\in G_{M,T}^2|_{I_j},\, Z_+(s)-Z_-(s) =0\,\,\forall s\in I_j\Big).\]
By property (iv), this equals
\begin{multline*}
Q_z\Big(\big\|Z_-(s)- f(s)\big\| \le \sfrac{1}{n^2} \;\;\forall s\in I_j, \; Z_-(\sfrac{j+1}{n})\in \mathcal Z_{j+1}, \; Z_-|_{I_j}\in G_{M,T}^2|_{I_j}\Big)\\
\cdot Q_z\Big( Z_+(s)-Z_-(s) =0\;\;\forall s\in I_j\Big)
\end{multline*}
and finally, by property (iii), the above equals
\begin{multline*}
Q_z\Big(\big|X_-(s)- f_X(s)\big| \le \sfrac{1}{n^2} \;\;\forall s\in I_j,\; \big|X_-(\sfrac{j+1}{n})-f_X(\sfrac{j+1}{n})\big| \le \sfrac{1}{2n^2}, \; X_-|_{I_j}\in G_{M,T}|_{I_j}\Big)\\
\cdot Q_z\Big(\big|Y_-(s)- f_Y(s)\big| \le \sfrac{1}{n^2} \;\;\forall s\in I_j,\; \big|Y_-(\sfrac{j+1}{n})- f_Y(\sfrac{j+1}{n})\big| \le \sfrac{1}{2n^2}, \; Y_-|_{I_j}\in G_{M,T}|_{I_j}\Big)\\
\cdot Q_z\Big( X_+(s)-X_-(s) =0\;\;\forall s\in I_j\Big) \cdot Q_z\Big( Y_+(s)-Y_-(s) =0\;\;\forall s\in I_j\Big).
\end{multline*}
Noting that $X_+-X_-$ and $Y_+-Y_-$ are increasing, this is exactly
\[q^X_{n,M,T}(z,j,f)\, \hat q^X_{n,M,T}(z,j,f)\, q^Y_{n,M,T}(z,j,f)\, \hat q^Y_{n,M,T}(z,j,f).\]
Thus we have shown \eqref{Qlbremains} and the proof is complete.
\end{proof}

The proof of Lemma \ref{qhatlb}, which bounds the $\hat q$ terms, is elementary.

\begin{proof}[Proof of Lemma \ref{qhatlb}]
Recall that
\[\hat q^X_{n,M,T}(z,j,f) = Q_z^{I_j,\Lambda_{M,T}(f,n),T}\Big( X_+(\sfrac{j+1}{n})-X_-(\sfrac{j+1}{n}) =0\Big).\]
Also recall that under $Q_z^{I_j,\Lambda_{M,T}(f,n),T}$, the process $X_+ - X_-$ jumps at rate
\[2\big(R_X^+(I_j,\Lambda_{M,T}(f,n),T) - R_X^-(I_j,\Lambda_{M,T}(f,n),T)\big)T.\]
Therefore, for each $j\in\{0,\ldots,n-1\}$ and $z\in\mathcal Z_j$, using \eqref{Rtrap},
\[\hat q^X_{n,M,T}(z,j,f) \ge \exp\Big(-2\big(R_X^+(I_j,\Gamma_{M,T}(f,n),T) - R_X^-(I_j,\Gamma_{M,T}(f,n),T)\big)T/n\Big).\]
Since $f\in G_M^2$, for $j\ge\sqrt n$, by \eqref{Rtrap2} we have
\[\hat q^X_{n,M,T}(z,j,f) \ge \exp\big(-4\delta_{M,T}(j,n)T/n\big),\]
and by symmetry
\[\hat q^Y_{n,M,T}(z,j,f) \ge \exp\big(-4\delta_{M,T}(j,n)T/n\big).\]
The result then follows from \eqref{sumdeltaeq}.
\end{proof}

The proof of Lemma \ref{qlblem} is much more delicate. Our next result provides a bound on compound Poisson processes, which we prove using standard arguments in Appendix \ref{CPP_append}. This will then be applied to prove Lemma \ref{qlblem}.

\begin{lem}\label{CPPfollowtubecor}
Suppose that $\delta,t,A>0$ and $a\in\R$ satisfy $a<tA/2$ and $|a|\le\delta/2$. Suppose also that $R\ge 1/2$. Let $(X(s),s\ge0)$ be a compound Poisson process of rate $RT$ whose jumps are exponentially distributed with parameter $T$. Then for $T>\frac{2(A-a/t)^{3/2}(4t+\delta)}{R^{1/2}\delta^{2}((A-a/t)\wedge 1)^2}$,
\begin{multline*}
\P(|a+X(s)-As|<\delta \,\, \forall s\le t, \,\, |a+X(t)-At|<\delta/2)\\
\ge \frac{1}{2}\exp\Big(-tT(\sqrt R - \sqrt A)^2 - \delta\Big(1+\sqrt{R}\big(\sqrt{2t/\delta}+1/2\big)\Big)T\Big).
\end{multline*}
\end{lem}

We now apply Lemma \ref{CPPfollowtubecor} to prove Lemma \ref{qlblem}. However we still need to consider two cases: if $f_X$ does not change much over the interval $I_j$ then we may simply ask our process not to jump over that interval, and a bound similar to that in the proof of Lemma \ref{qhatlb} is better than the estimate provided by Lemma \ref{CPPfollowtubecor}.

\begin{proof}[Proof of Lemma \ref{qlblem}]
Recall that
\begin{multline*}
q^X_{n,3M,T}(z,j,f) = Q_z^{I_j,\Lambda_{3M,T}(f,n),T}\Big(\big|X_-(s)- f_X(s)\big| \le \sfrac{1}{n^2} \,\,\forall s\in I_j,\\
\big|X_-(\sfrac{j+1}{n})-f_X(\sfrac{j+1}{n})\big| \le \sfrac{1}{2n^2},\, X_-|_{I_j}\in G_{3M,T}|_{I_j}\Big)
\end{multline*}
Write $Q_z$ as shorthand for $Q_z^{I_j,\Lambda_{3M,T}(f,n),T}$.

Since $f\in G_M^2$, $j\ge 1$ and $n\ge 2M$, under $Q_z$ we also have, for any $s\in I_j$,
\[X_-(s)\ge x \ge f_X \Big(\frac{j}{n}\Big)-\frac{1}{2n^2} \ge \frac{j}{Mn} - \frac{1}{2n^2} \ge \frac{3j}{3Mn} - \frac{1}{3Mn} \ge \frac{j+2}{3Mn} - \frac{1}{3Mn} = \frac{j+1}{3Mn} \ge \frac{s}{3M},\]
and if $\big|X_-(\sfrac{j+1}{n})-f_X(\sfrac{j+1}{n})\big| \le \sfrac{1}{2n^2}$ then also
\[X_-(s)\le X_-(\sfrac{j+1}{n}) \le f_X(\sfrac{j+1}{n}) + \sfrac{1}{2n^2} \le M\sfrac{j+1}{n} + \sfrac{1}{2n^2} \le 3Ms.\]
Thus in fact, under the conditions of the lemma, $X_-|_{I_j}$ is always in $G_{3M,T}|_{I_j}$, so
\begin{equation}\label{slimq}
q^X_{n,3M,T}(z,j,f) = Q_z^{I_j,\Lambda_{3M,T}(f,n),T}\Big(\big|X_-(s)- f_X(s)\big| \le \sfrac{1}{n^2} \;\;\forall s\in I_j,\; \big|X_-(\sfrac{j+1}{n})-f_X(\sfrac{j+1}{n})\big| \le \sfrac{1}{2n^2}\Big).
\end{equation}

For the remainder of this proof, for $I\subset[0,1]$, we write $\hat R^-_X(I)$ as shorthand for the quantity $R^-_X(I,\Lambda_{M,T}(f,n),T)$, and similarly for $\hat R^+_X(I)$, $\hat R^-_Y(I)$ and $\hat R^+_Y(I)$. (Recall that we wrote $R^-_X(I)$ in Section \ref{coupling_sec} to mean $R^-_X(I,\Gamma_{M,T}(f,n),T)$.)

\vspace{3mm}

\noindent
\textbf{Case 1:} $f_X(\frac{j+1}{n})\le x + \frac{1}{2n^2}$.\\
Note that since $z\in\mathcal Z_j$, we have $x\le f_X(\sfrac{j}{n})+\sfrac{1}{2n^2} \le f_X(\sfrac{j+1}{n})+\sfrac{1}{2n^2}$, and therefore $|x-f_X(s)|\le \sfrac{1}{2n^2}$ for all $s\in I_j$. 
Thus \eqref{slimq} can be bounded in the following trivial way:
\[q^X_{n,3M,T}(z,j,f) \ge Q_z\big(X_-(s) = x \;\;\forall s\in I_j) = Q_z\big(X_-(\sfrac{j+1}{n}) = x).\]
Under $Q_z$, $X_-$ jumps at rate $2\hat R_X^-(j)T$, so we deduce that
\begin{equation}\label{lbqtoR-case1}
q^X_{n,3M,T}(z,j,f) \ge \exp\Big(-\frac{2\hat R_X^-(j)T}{n}\Big).
\end{equation}
On the other hand we have
\begin{align*}
&\int_{j/n}^{(j+1)/n} \Big(\sqrt{2R^*_X(f(s))}-\sqrt{f'_X(s)}\Big)^2 ds\\
&\ge 2\int_{j/n}^{(j+1)/n} R^*_X(f(s)) ds - 2\int_{j/n}^{(j+1)/n} \sqrt{2R^*_X(f(s))n\big(f_X(\sfrac{j+1}{n})-f_X(\sfrac{j}{n})\big)} ds + f_X(\sfrac{j+1}{n}) - f_X(\sfrac{j}{n}).
\end{align*}
By \eqref{Rtrap2},
\[R_X^*(f(s)) \ge \hat R_X^+(j) - \delta_{M,T}(j,n) \ge \hat R_X^-(j) - \delta_{M,T}(j,n),\]
so since $f_X(\sfrac{j+1}{n}) - f_X(\sfrac{j}{n}) \le \sfrac{1}{n^2}$ and $R_X^*(f(s))\le M$ for all $s$,
\[\int_{j/n}^{(j+1)/n} \Big(\sqrt{2R^*_X(f(s))}-\sqrt{f'_X(s)}\Big)^2 ds \ge \frac{2\hat R^-_X(j) - 2\delta_{M,T}(j,n)}{n}  - \frac{2 \sqrt{2}M^{1/2}}{n^{3/2}} .\]
The result now follows from this and \eqref{lbqtoR-case1}.

\vspace{3mm}

\noindent
\textbf{Case 2:} $f_X(\frac{j+1}{n})> x+\frac{1}{2n^2}$.\\
Note that $X_-$ jumps at rate $2\hat R_X^-(j)T$ and has exponential jumps of parameter $T$ under $Q_z$. We therefore aim to apply Lemma \ref{CPPfollowtubecor}, with $A=n(f_X(\frac{j+1}{n})-f_X(\frac{j}{n}))$, $\delta = 1/n^2$, $t=1/n$ and $a=x-f_X(j/n)$. We need to check that $a<tA/2$; to see this, note that since $z\in\mathcal Z(j)$ and we are in Case 2,
\[2a=2\big(x-f_X(\sfrac{j}{n})\big)\le \sfrac{1}{2n^2} + x-f_X(\sfrac{j}{n}) < f_X(\sfrac{j+1}{n}) - f_X(\sfrac{j}{n})=tA.\]
It is also easy to check that for $T>8n^{9/2}M^{3/2}$, $T$ is large enough that the conclusion of Lemma \ref{CPPfollowtubecor} holds. Thus applying Lemma \ref{CPPfollowtubecor} to \eqref{slimq} gives
\begin{multline*}
q^X_{n,3M,T}(z,j,f)\\
\ge \frac{1}{2} \exp\bigg(-\frac{T}{n}\Big(\sqrt{2\hat R_X^-(j)}- \sqrt{n\big(f_X(\sfrac{j+1}{n})-f_X(\sfrac{j}{n})\big)}\Big)^2 - \frac{1}{n^2}\Big(1+\sqrt{2\hat R_X^-(j)}\big(\sqrt{2n}+1/2\big)\Big)T\bigg).
\end{multline*}
Since $f\in G_M^2$, we have $\hat R_X^-(j)\le M$ and therefore
\[1+\sqrt{2\hat R_X^-(j)}\big(\sqrt{2n}+1/2\big) \le 1 + \sqrt{2M}\big(\sqrt{2n}+1/2\big) \le 2(M+1)n^{1/2}.\]
Thus
\begin{equation}\label{case2lbpt1}
q^X_{n,3M,T}(z,j,f) \ge \frac{1}{2} \exp\bigg(-\frac{T}{n}\Big(\sqrt{2\hat R_X^-(j)} - \sqrt{n\big(f_X(\sfrac{j+1}{n})-f_X(\sfrac{j}{n})\big)}\Big)^2 - \frac{2(M+1)T}{n^{3/2}}\bigg).
\end{equation}

Noting that since $f\in \PL_n^2$ we have $n\big(f_X(\sfrac{j+1}{n})-f_X(\sfrac{j}{n})\big) = f'(s)$ for all $s\in I_j$, and by \eqref{Rtrap2}
\[\big(R^*_X(f(s))-\delta_{M,T}(j,n)\big)\vee 0 \le \hat R_X^-(j) \le \hat R_X^+(j) \le R_X^*(f(s))+\delta_{M,T}(j,n),\]
we deduce that
\[\frac{2}{n}\hat R_X^-(j) \le \int_{j/n}^{(j+1)/n} 2R^*_X(f(s))ds + \frac{2\delta_{M,T}(j,n)}{n}\]
and using also that $\sqrt{(a-b)\wedge 0} \ge \sqrt a - \sqrt b$ for $a,b\ge0$,
\begin{multline*}
\frac{1}{n}\sqrt{2\hat R_X^-(j)n\big(f_X(\sfrac{j+1}{n})-f_X(\sfrac{j}{n})\big)}\\
\ge \int_{j/n}^{(j+1)/n} \sqrt{2R^*_X(f(s))f'_X(s)}ds - \frac{1}{\sqrt n}\sqrt{2\delta_{M,T}(j,n)\big(f_X(\sfrac{j+1}{n})-f_X(\sfrac{j}{n})\big)}.
\end{multline*}
Thus
\begin{multline*}
\frac{1}{n}\Big(\sqrt{2\hat R_X^-(j)} - \sqrt{n\big(f_X(\sfrac{j+1}{n})-f_X(\sfrac{j}{n})\big)}\Big)^2 \le \int_{j/n}^{(j+1)/n} \Big(\sqrt{2R^*_X(f(s))} - \sqrt{f'(s)}\Big)^2 ds + \frac{2\delta_{M,T}(j,n)}{n}\\
+ \frac{1}{\sqrt n}\sqrt{2\delta_{M,T}(j,n)\big(f_X(\sfrac{j+1}{n})-f_X(\sfrac{j}{n})\big)}.
\end{multline*}
This combines with \eqref{case2lbpt1} to give the result.
\end{proof}

\section{The final details for the upper bound}\label{finaldetailssec}

\subsection{Compactness and semicontinuity}\label{cpctsemisec}

There are a few more technical issues that must be resolved in order to complete the proof of the upper bound in Theorem \ref{mainthm}. One of the remaining ingredients is to prove that the set of functions that we are interested in can be covered by a finite collection of small balls around suitably chosen functions. Recall that $\PL_n$ is the subset of functions in $E$ that are linear on each interval $[i/n,(i+1)/n]$ for all $i=0,\ldots,n-1$ and continuous on $[0,1]$. For $F\subset E$ and $r>0$, write $B_d(F,r) = \bigcup_{f\in F} B_d(f,r)$, where $B_d(f,r)$ is the ball of radius $r$ about $f$ in the metric $d$.

\begin{lem}\label{coverGMT}
Suppose that $F\subset E^2$ and $M>1$. For any $n\ge 4M$, there exist $N\in\N\cup\{0\}$ and $g_1,\ldots,g_N\in G_{4M}^2\cap\PL_n^2$ such that
\[F\cap G_{M,T}^2\subset \bigcup_{i=1}^N \big( B_{\Delta_n}(g_i,1/n^2) \cap B_d(g_i,1/n)\big)\subset B_d(F,2/n)\]
for all $T\ge (4Mn)^{3/2}$.
\end{lem}

\noindent
We will prove this in Appendix \ref{rel_cpct_sec}.

In order to check that the supremum of our rate function $\tilde K$ over $f\in B_d(F,\eps)$ is close to the supremum over $f\in F$ when $\eps$ is small, we will need to show that $\tilde K$ has some form of upper semi-continuity.

\begin{prop}\label{uppersemicts}
Suppose that $0<\theta\le 1$ and there exists $M\in(1,\infty)$ such that $f,f_n\in G_M^2$ for all $n$. Suppose also that either $f$ is continuous at $\theta$, or $\theta=1$. If $d(f_n,f)\to 0$ then
\[\limsup_{n\to\infty} \tilde K(f_n,0,\theta) \le \tilde K(f,0,\theta).\]
\end{prop}

The following simple corollary of Proposition \ref{uppersemicts} is written in a more convenient form.

\begin{cor}\label{uppersemictscor}
Suppose that $M\in(1,\infty)$ and $F\subset E^2$ is closed. Then
\[\lim_{\eps\to 0} \sup_{f\in B_d(F,\eps) \cap G_M^2} \tilde K(f,0,1) \le \sup_{f\in F \cap G_M^2} \tilde K(f,0,1).\]
\end{cor}

We will prove Proposition \ref{uppersemicts} and Corollary \ref{uppersemictscor} in Appendix \ref{usc_sec}.

\subsection{The result for fixed $T$: proof of Propositions \ref{fixedgprop} and \ref{fixedTprop}}\label{fixedTsec}

\begin{proof}[Proof of Proposition \ref{fixedgprop}]
By Markov's inequality, for any $\kappa>0$,
\[\P\Big(N_T\big(\Gamma_{M,T}(g,n),\theta\big) \ge \kappa\Big)\le \E\bigg[\sum_{v\in\Nc_T} \ind_{\{Z_v^T|_{[0,\theta]}\in\Gamma_{M,T}(g,n)|_{[0,\theta]}\}}\bigg]\frac{1}{\kappa},\]
and by Lemma \ref{MtO},
\[\E\bigg[\sum_{v\in\Nc_T} \ind_{\{Z_v^T|_{[0,\theta]}\in\Gamma_{M,T}(g,n)|_{[0,\theta]}\}}\bigg] = \Q\big[\ind_{\{\xi^T|_{[0,\theta]}\in\Gamma_{M,T}(g,n)|_{[0,\theta]}\}}e^{\int_0^{\theta T} R(\xi_s)ds}\big].\]
Now, if $\xi^T|_{[0,\theta]}\in\Gamma_{M,T}(g,n)|_{[0,\theta]}$, then by Lemma \ref{finaldetRbd},
\[\int_0^{\theta T} R(\xi_s)ds = T\int_0^\theta R(T\xi^T(s)) ds \le T\int_0^{\lfloor\theta n\rfloor/n} R^*(g(s))ds + T\eta(M,n,T),\]
and therefore
\begin{multline*}
\Q\big[\ind_{\{\xi^T|_{[0,\theta]}\in\Gamma_{M,T}(g,n)|_{[0,\theta]}\}}e^{\int_0^{\theta T} R(\xi_s)ds}\big]\\
\le \Q\big(\xi^T|_{[0,\theta]}\in\Gamma_{M,T}(g,n)|_{[0,\theta]}\big) e^{T\int_0^{\lfloor\theta n\rfloor/n} R^*(g(s))ds + T\eta(M,n,T)}.
\end{multline*}
We also know from Proposition \ref{Qprobcor} that
\[\Q(\xi^T|_{[0,\theta]}\in\Gamma_{M,T}(g,n)|_{[0,\theta]}) \le \exp\bigg(-T\sum_{j=0}^{\lfloor\theta n\rfloor-1} \big(\mathcal E^+_X(I_j,\Gamma_{M,T}(g,n),T) + \mathcal E^+_Y(I_j,\Gamma_{M,T}(g,n),T)\big)\bigg),\]
and by Proposition \ref{detboundratefn} that, if $g = (g_X,g_Y)$,
\[\sum_{j=\lceil\sqrt n\rceil}^{\lfloor \theta n\rfloor-1} \mathcal E^+_X(I_j,\Gamma_{M,T}(g,n),T) \ge \int_{\lceil\sqrt n\rceil/n}^{\lfloor \theta n\rfloor/n} \Big(\sqrt{2R_X^*(g(s))} - \sqrt{g_X'(s)}\Big)^2 ds - O\Big(\frac{M^4}{n^{1/4}} + \frac{M^3n}{T^{1/2}}\Big).\]
Since $g\in G_{M}^2$, we also have
\begin{align}
\int_0^{\lceil\sqrt n\rceil/n} \Big(\sqrt{2R_X^*(g(s))} - \sqrt{g'_{X}(s)}\Big)^2 ds &\le \int_0^{\lceil\sqrt n\rceil/n} 2R_X^*(g(s)) ds + \int_0^{\lceil\sqrt n\rceil/n} g'_{X}(s) ds\nonumber\\
&\le \frac{2M^2\lceil\sqrt n\rceil}{n} + \frac{M\lceil\sqrt n\rceil}{n} \le \frac{4M^2}{\sqrt n}\label{smallsIbd}
\end{align}
so
\[\sum_{j=\lceil\sqrt n\rceil}^{\lfloor \theta n\rfloor-1} \mathcal E^+_X(I_j,\Gamma_{M,T}(g,n),T) \ge \int_0^{\lfloor \theta n\rfloor/n} \Big(\sqrt{2R_X^*(g(s))} - \sqrt{g_X'(s)}\Big)^2 ds - O\Big(\frac{M^4}{n^{1/4}} + \frac{M^3n}{T^{1/2}}\Big)\]
and by symmetry the same bound holds for $Y$. Recalling from Lemma \ref{finaldetRbd} that $\eta(M,n,T) = O\Big(\frac{M^4}{n^{1/2}} + \frac{M^3n}{T^{1/3}}\Big)$, we deduce that
\begin{align*}
&\P\Big(N_T\big(\Gamma_{M,T}(g,n),\theta\big) \ge \kappa\Big)\\
&\hspace{10mm}\le e^{-T\int_0^{\lfloor \theta n\rfloor/n} \big(\sqrt{2R_X^*(g(s))} - \sqrt{g_X'(s)}\big)^2 ds - T\int_0^{\lfloor \theta n\rfloor/n} \big(\sqrt{2R_Y^*(g(s))} - \sqrt{g_Y'(s)}\big)^2 ds}\\
&\hspace{50mm}\cdot e^{O\left(\frac{M^4T}{n^{1/4}} + M^3nT^{2/3}\right) + T\int_0^{\lfloor \theta n\rfloor/n} R^*(g(s))ds}\cdot \frac{1}{\kappa}\\
&\hspace{10mm} = \frac{1}{\kappa} e^{T \tilde K(g,0,\lfloor \theta n\rfloor/n) + O\left(\frac{M^4T}{n^{1/4}} + M^3n T^{2/3}\right)}
\end{align*}
as required, where for the last equality we used the fact that $g\in G_M^2\cap \PL_n^2$, and therefore $\tilde K(g,0,s)=\int_0^s R^*(g(u))du - I(g,0,s)$ for all $s$.
\end{proof}

Proposition \ref{fixedTprop} essentially establishes the upper bound in Theorem \ref{mainthm} with high probability for a fixed (large) $T$. The proof mostly involves using Lemma \ref{AlwaysGMT} and the technical results stated in Section \ref{cpctsemisec} to ensure that we can cover our set in a suitable way with finitely many balls around piecewise linear functions, and then applying Proposition \ref{fixedgprop}.

\begin{proof}[Proof of Proposition \ref{fixedTprop}]
Take $M\ge M_0$ and the other parameters as in the statement of the Proposition. By Lemma \ref{AlwaysGMT},
\[\P(\exists v\in\Nc_T : Z^T_v\not\in G_{M,T}^2) \le e^{-\delta_0 T^{1/3}}.\]
By Corollary \ref{uppersemictscor}, since $F$ is closed we may choose $n$ large enough such that $n\ge 4M$ and
\[\sup_{f\in B_d(F,2/n)\cap G_{4M}^2} \tilde K(f,0,1) \le \sup_{f\in F\cap G_{4M}^2} \tilde K(f,0,1) + \eps/3.\]
By Lemma \ref{coverGMT} we may choose $N\in\N$ and $g_1,\ldots,g_N\in G_{4M}^2\cap\PL_n^2$ such that
\[F\cap G_{M,T}^2 \subset \bigcup_{i=1}^N \big(B_{\Delta_n}(g_i,1/n^2)\cap B_d(g_i,1/n)\big) \subset B_d(F,2/n)\]
for all $T\ge (4Mn)^{3/2}$. Recall that $\Gamma_{M,T}(g_i,n) = B_{\Delta_n}(g_i,1/n^2)\cap B_d(g_i,1/n)\cap G_{M,T}^2$. Then for any $A\ge 0$,
\begin{align}
\P\big(N_T(F) \ge e^{A T}\big) &\le \P\big(\exists v\in\Nc_T : Z_v^T \not\in G_{M,T}^2\big) + \sum_{i=1}^N \P\Big(N_T\big(\Gamma_{M,T}(g_i,n)\big) \ge \frac{e^{A T}}{N}\Big)\nonumber\\
&\le e^{-\delta_0 T^{1/3}} + \sum_{i=1}^N \P\Big(N_T\big(\Gamma_{M,T}(g_i,n)\big) \ge \frac{e^{A T}}{N}\Big).\label{decompFeq}
\end{align}
By Proposition \ref{fixedgprop}, for each $i$ we have
\[\P\Big(N_T\big(\Gamma_{M,T}(g_i,n)\big) \ge \frac{e^{A T}}{N}\Big) \le \frac{N}{e^{AT}}\exp\bigg(T \tilde K(g_i,0,1) + O\Big(\frac{M^4T}{n^{1/4}} + M^3n T^{2/3}\Big)\bigg),\]
and combining this with \eqref{decompFeq} we see that
\[\P\big(N_T(F) \ge e^{A T}\big) \le e^{-\delta_0 T^{1/3}} + \frac{N^2}{e^{A T}}\max_{i\in\{1,\ldots,N\}} \exp\bigg(T \tilde K(g_i,0,1) + O\Big(\frac{M^4T}{n^{1/4}} + M^3n T^{2/3}\Big)\bigg).\]
By our choice of $g_1,\ldots,g_N$ and $n$, we have
\[\max_{i\in\{1,\ldots,N\}} \tilde K(g_i,0,1) \le \sup_{f\in B_d(F,2/n)\cap G^2_{4M}}\tilde K(f,0,1) \le \sup_{f\in F\cap G^2_{4M}} \tilde K(f,0,1) + \eps/3\]
and therefore
\[\frac{1}{T^{1/3}}\log\P\big(N_T(F) \ge e^{A T}\big) \le (-\delta_0)\vee \bigg(\sup_{f\in F\cap G^2_{4M}} \tilde K(f,0,1) T^{2/3}- A T^{2/3} + \frac{\eps T^{2/3}}{3} + O\Big(\frac{M^4 T^{2/3}}{n^{1/4}}\Big)\bigg).\]
Increasing $n$ if necessary so that the $O(\frac{M^4 T^{2/3}}{n^{1/4}})$ term is smaller than $\sfrac{\eps T^{2/3}}{3}$, and choosing
\[A = \sup_{f\in F\cap G_{4M}^2} \tilde K(f,0,1) + \eps,\]
we have
\[\lim_{T\to\infty} \frac{1}{T^{1/3}}\log \P\big(N_T(F,\theta) \ge e^{A T}\big) \le -\delta_0.\]
This is precisely the statement of the proposition, but with $4M$ in place of $M$. Since we only assumed that $M\ge M_0$ in the proof, the proposition holds when $M\ge 4M_0$.
\end{proof}

\subsection[Paths with $K(f)$ negative are unlikely: proof of Lemma \ref{supKnull}]{Paths with $K(f)=-\infty$ are unlikely: proof of Lemma \ref{supKnull}}\label{sKn_sec}

Before proving Lemma \ref{supKnull}, we need to relate $K$ to $\tilde K$.

\begin{lem}\label{supKneg}
Suppose that $M>1$. If $F\subset E^2$ is closed and $\sup_{f\in F} K(f) = -\infty$, then there exists $\eps>0$ such that
\[\sup_{f\in B(F,\eps)\cap G_{M,1}^2} \inf_{\theta\in[0,1]} \tilde K(f,0,\theta) < 0.\]
\end{lem}

\begin{proof}
If the result is not true, then for each $n\in\N$ we may choose $f_n\in B(F,1/n)\cap G_{M,1}^2$ such that
\[\inf_{\theta\in[0,1]} \tilde K(f_n,0,\theta) \ge -1/n.\]
It is easy to check that $G_{M,1}^2$ is closed and totally bounded. Since $(E^2,d)$ is complete, $G_{M,1}^2$ is compact. We may therefore find a subsequence $(f_{n_j})_{j\ge 1}$ such that $d(f_{n_j},f_\infty)\to 0$ as $j\to\infty$ for some $f_\infty\in G_{M,1}^2$. Since $F\cap G_{M,1}^2$ is closed, and $d(f_\infty,F\cap G_{M,1}^2)=0$, we must in fact have $f_\infty\in F\cap G_{M,1}^2$. On the other hand, by Proposition \ref{uppersemicts}, for any $\theta\in[0,1]$ such that $f_\infty$ is continuous at $\theta$,
\[\tilde K(f_\infty,0,\theta) \ge \limsup_{j\to\infty}\tilde K(f_{n_j},0,\theta) \ge 0.\]
But $f_\infty$ is non-decreasing and therefore continuous almost everywhere, and $t\mapsto\tilde K(f,0,t)$ has only downward jumps, so we must have $\tilde K(f_\infty,0,\theta)\ge 0$ for all $\theta\in[0,1]$. Thus $K(f_\infty)\ge 0$, which contradicts the hypothesis of the lemma.
\end{proof}

We can now prove Lemma \ref{supKnull}, which says that if $F$ is closed and $\sup_{f\in F} K(f) = -\infty$, then with high probability $N_T(F)$ is zero.

\begin{proof}[Proof of Lemma \ref{supKnull}]
Choose $M\ge M_0$. Since $G_{4M}\subset G_{4M,1}$, by Lemma \ref{supKneg} we may choose $n_0\ge 4M$ such that
\[\sup_{f\in B(F,2/n_0)\cap G_{4M}^2} \inf_{\theta\in[0,1]} \tilde K(f,0,\theta)<0.\]
Let
\begin{equation}\label{supKnulleta}
\eta = -\sup_{f\in B(F,2/n_0)\cap G_{4M}^2} \inf_{\theta\in[0,1]} \tilde K(f,0,\theta) > 0.
\end{equation}
Then take $n\ge n_0$ such that the error term in Proposition \ref{fixedgprop} is smaller than $\eta T/3$ for $T$ sufficiently large, and such that $(4M)^2/n \le \eta/3$.

By Lemma \ref{coverGMT} we may choose $N\in\N\cup\{0\}$ and $g_1,\ldots,g_N\in G_{4M}^2\cap \PL_n^2$ such that
\[F\cap G_{M,T}^2 \subset \bigcup_{i=1}^N \big( B_{\Delta_n}(g_i,1/n^2) \cap B_d(g_i,1/n)\big)\subset B_d(F,2/n)\]
for all $T\ge (4Mn)^{3/2}$.


For each $i=1,\ldots,N$, note that since $g_i\in G_{4M}^2$, by the definition of $\tilde K$, for any $0\le s\le t\le 1$ we have
\begin{equation}\label{Krcts}
\tilde K(g_i,0,t) \le \tilde K(g_i,0,s) + (4M)^2(t-s).
\end{equation}
In particular, the function $t\mapsto \tilde K(g_i,0,t)$ has only downward jumps, and therefore its infimum is achieved. Thus, by \eqref{supKnulleta}, we may choose $\theta_i$ such that
\[\tilde K(g_i,0,\theta_i) = \inf_{\theta\in[0,1]} \tilde K(g_i,0,\theta)\le -\eta.\]
Let $\hat\theta_i = \lceil\theta_i n\rceil/n$. Using \eqref{Krcts} again, we then have
\begin{equation}\label{hatthetai}
\tilde K(g_i,0,\hat\theta_i) \le -\eta + (4M)^2/n \le -2\eta/3
\end{equation}
where the last inequality holds because we chose $n$ such that $(4M)^2/n\le \eta/3$.

Now, by our choice of $g_1,\ldots,g_N$, we have
\[N_T(F) \le N_T((G_{M,T}^2)^c) + \sum_{i=1}^N N_T(\Gamma_{M,T}(g_i,n))\]
and therefore
\begin{equation}\label{NTFge1}
\P(N_T(F)\ge 1) \le \P\big(N_T((G_{M,T}^2)^c)\ge 1\big) + \sum_{i=1}^N \P\big(N_T(\Gamma_{M,T}(g_i,n))\ge 1\big).
\end{equation}
By Lemma \ref{AlwaysGMT}, the first term on the right-hand side above is at most $e^{-\delta_0 T^{1/3}}$. Also, since a population that is extinct at time $\theta$ must also be extinct at time $1$, for each $i$ we have
\[\P\big(N_T(\Gamma_{M,T}(g_i,n))\ge 1\big) \le \P\big(N_T(\Gamma_{M,T}(g_i,n),\hat\theta_i)\ge 1\big).\]
Since $\hat\theta_i$ is an integer multiple of $1/n$, by Proposition \ref{fixedgprop} we have
\[\P\big(N_T(\Gamma_{M,T}(g_i,n),\hat\theta_i)\ge 1\big) \le \exp\bigg(T\tilde K\Big(g,0,\hat\theta_i\Big) + O\Big(\frac{M^4T}{n^{1/4}} + M^3 n T^{2/3}\Big)\bigg) \le \exp\Big(-\frac{\eta T}{3}\Big),\]
where the last inequality follows from \eqref{hatthetai} and our choice of $n$. Returning to \eqref{NTFge1}, we have shown that
\[\P(N_T(F)\ge 1) \le e^{-\delta_0 T^{1/3}} + N e^{-\eta T/3},\]
which completes the proof.
\end{proof}

\subsection{Lattice times to continuous time: proof of Proposition \ref{thetaUBprop}}\label{tUBsec}

Before moving on to the proof of Proposition \ref{thetaUBprop}, we state and prove two lemmas that will check that paths of particles are not drastically changed by rescaling by a slightly different value of $T$.

\begin{lem}\label{rescalingcts}
Suppose that $M>1$, $t\ge 3M$ and $t-1\le s\le t$. For any $F\subset E^2$, we have
\[N_s(F\cap G_{M,s}^2)\le N_t\big(B(F, 3M/t)\big).\]
\end{lem}

\begin{proof}
Suppose that $u\in\Nc_{s}$ satisfies $Z_u^s\in F\cap G_{M,s}^2$. Take any $v\in\Nc_{t}$ such that $v$ is a descendant of $u$. We claim that $d(X_u^s, X_v^t)\le 3M/t$, which means that for all $\tau\in[-3M/t,1+3M/t]$,
\[X_v^{t}(\tau-3M/t)-3M/t \le X_u^s(\tau) \le X_v^{t}(\tau+3M/t) + 3M/t\]
where $f(\tau)$ is interpreted to equal $f(0)$ for $\tau<0$ and $f(1)$ for $\tau>1$. Since $Z_u^s\in F$, the claim plus its equivalent $Y$ statement ensure that $Z_v^t\in B(F, 3M/t)$, which is enough to complete the proof.

To prove the claim, first note that it holds when $\tau\le 0$, since in this case $X_u^s(\tau)=X_u^s(0)=X_v^t(0) = X_v^t(\tau)$. If $\tau >0$, since $s\le t$ and
\[\tau s \ge \tau(t-1) = t(\tau - \tau/t) \ge t\Big(\tau-\frac{1+3M/t}{t}\Big) \ge t\Big(\tau - \frac{3M}{t}\Big),\]
we have
\[X^s_u(\tau) = X^s_v(\tau) \ge X_v^t(\tau - \sfrac{3M}{t}).\]
Also, since $X_u^s\in G_{M,s}^2$, for any $\tau\in[0,1]$ we have
\begin{align*}
X^s_u(\tau) = \frac{1}{s}X_u(\tau s) &= \frac{1}{t}X_u(\tau s) + \Big(1-\frac{s}{t}\Big)X_u^s(\tau)\\
&\le \frac{1}{t}X_v(\tau t) + \Big(\frac{t-s}{t}\Big) M(1+2s^{-2/3})\\
&\le X_v^{t}(\tau) + \frac{M}{t}(1+2s^{-2/3}) \le X_v^{t}(\tau) + \frac{3M}{t}
\end{align*}
as required. If $\tau>1$ then $X^s_u(\tau) = X^s_u(1)$ and then the argument above gives that that $X^s_u(1)\le X^t_v(1) + 3M/t = X^t_v(\tau) + 3M/t$.
\end{proof}

\begin{lem}\label{rescalingbadpaths}
Suppose that $M>2$, $T\ge 2$ and $t\in[T-1,T]$. If $N_t((G_{M,t}^2)^c)\ge 1$ then either $N_T((G_{M/2,T}^2)^c)\ge 1$ or $N_{T-1}((G_{M/2,T-1}^2)^c)\ge 1$.
\end{lem}

\begin{proof}
Suppose there exists $v\in\Nc_{t}$ such that $Z^t_v\in (G_{M,t}^2)^c$. It is possible that either $X^t_v$ or $Y^t_v$ (or both) is the reason for $Z^t_v$ falling outside $G_{M,t}^2$; without loss of generality assume that it is $X^t_v$. Then there exists $s\in[0,1]$ such that either $X^t_v(s) > M(s+2t^{-2/3})$, or $X^t_v(s)<s/M - 2t^{-2/3}$. In the first case, take $w\in\Nc_{T}$ such that $w$ is a descendant of $v$. Then
\[X^T_w(s) = \frac{1}{T}X_w(sT) \ge \frac{t}{T}\frac{1}{t}X_v(st) > \frac{1}{2}M(s+2t^{-2/3}) \ge \frac{M}{2}(s+2T^{-2/3})\]
so $Z^T_w\in (G_{M/2,T}^2)^c$. In the second case, let $u$ be the ancestor of $v$ in $\Nc_{T-1}$. Then
\[X^{T-1}_u(s) = \frac{1}{T-1}X_u(s(T-1)) \le \frac{t}{T-1}\frac{1}{t}X_v(st) < \frac{t}{T-1}\Big(\frac{s}{M}-2t^{-2/3}\Big) \le \frac{2s}{M} - 2(T-1)^{-2/3}\]
so $Z^{T-1}_u \in (G_{M/2,T-1}^2)^c$. This completes the proof.
\end{proof}

\begin{proof}[Proof of Proposition \ref{thetaUBprop}]
We begin with the first part of the result. Take $\eps>0$. We start by noting that
\begin{multline}\label{splitFoverG}
\P\Big(\exists t\in[T-1,T] : \frac{1}{t}\log N_t(F)\ge \sup_{f\in F\cap G_M^2} \tilde K(f,0,1) + \eps \Big)\\
\le \P\Big(\exists t\in[T-1,T] : \frac{1}{t}\log N_t(F\cap G_{M,t}^2)\ge \sup_{f\in F\cap G_M^2} \tilde K(f,0,1) + \eps \Big)\\
+ \P\big(\exists t\in [T-1,T] : N_t((G_{M,t}^2)^c)\ge 1\big).
\end{multline}
We show that the right-hand side is exponentially small in $T$. By Corollary \ref{uppersemictscor}, we can choose $\eps'\in(0,1)$ such that
\[\sup_{f\in \overline{B(F,\eps')}\cap G_M^2} \tilde K(f,0,1) \le \sup_{f\in F\cap G_M^2} \tilde K(f,0,1) + \eps/3.\]
By Lemma \ref{rescalingcts}, provided that $3M/T\le \eps'$, we have
\[N_t(F\cap G_{M,t}^2) \le N_T(B(F,\eps'))\]
for all $t\in[T-1,T]$. Therefore for large $T$
\begin{align*}
&\P\Big(\exists t\in[T-1,T] : \frac{1}{t}\log N_t(F\cap G_{M,t}^2)\ge \sup_{f\in F\cap G_M^2} \tilde K(f,0,1) + \eps \Big)\\
&\hspace{30mm}\le \P\Big(\frac{1}{T-1}\log N_T(B(F,\eps'))\ge \sup_{f\in F\cap G_M^2} \tilde K(f,0,1) + \eps \Big)\\
&\hspace{30mm}\le \P\Big(\frac{1}{T}\log N_T(B(F,\eps'))\ge \sup_{f\in \overline{B(F,\eps')}\cap G_M^2} \tilde K(f,0,1) + \eps/3 \Big).
\end{align*}
Then Proposition \ref{fixedTprop} tells us that this is at most $\exp(-\delta_0 T^{1/3}/2)$ for large $T$. Substituting this into \eqref{splitFoverG}, we have
\begin{multline}\label{splitFoverG2}
\P\Big(\exists t\in[T-1,T] : \frac{1}{t}\log N_t(F)\ge \sup_{f\in F\cap G_M^2} \tilde K(f,0,1) + \eps \Big)\\
\le \exp(-\delta_0 T^{1/3}/2) + \P\big(\exists t\in [T-1,T] : N_t((G_{M,t}^2)^c)\ge 1\big).
\end{multline}

For the remaining term, Lemma \ref{rescalingbadpaths} tells us that for $T\ge2$,
\[\P(\exists t\in [T-1,T] : N_t((G_{M,t}^2)^c)\ge 1) \le \P(N_T((G_{M/2,T}^2)^c)\ge 1) + \P(N_{T-1}((G_{M/2,T-1}^2)^c)\ge 1).\]
By Lemma \ref{AlwaysGMT}, this is at most $2 \exp \big( \! -\delta_0 (T-1)^{1/3} \big)$. Returning to \eqref{splitFoverG2}, we have
\[\P\Big(\exists t\in[T-1,T] : \frac{1}{t}\log N_t(F)\ge \sup_{f\in F\cap G_M^2} \tilde K(f,0,1) + \eps\Big)\le \exp(-\delta_0 T^{1/3}/2) + 2 \exp \big( \! -\delta_0 (T-1)^{1/3} \big).\]
By the Borel-Cantelli lemma,
\[\P\Big(\limsup_{t\to\infty}\frac{1}{t}\log N_t(F)\ge \sup_{f\in F\cap G_M^2} \tilde K(f,0,1) + \eps\Big) = 0,\]
and since $\eps>0$ was arbitrary, we deduce the first part of the result.

The proof when $\sup_{f\in F} K(f) = -\infty$ is very similar. By Lemma \ref{supKneg}, we may choose $\eps'' >0$ such that
\begin{equation}\label{supexpanded}
\sup_{f\in \overline{B(F,\eps'')}\cap G_{M,1}^2} K(f) = -\infty.
\end{equation}
Then
\begin{multline}\label{thetaUBpropfinal}
\P(\exists t\in[T-1,T] : N_t(F) \ge 1 ) \le \P(\exists t\in[T-1,T] : N_t(F\cap G_{M,t}^2) \ge 1 )\\
+ \P(\exists t\in[T-1,T] : N_t((G_{M,t}^2)^c) \ge 1 ).
\end{multline}
As argued above, by Lemmas \ref{rescalingbadpaths} and \ref{AlwaysGMT} the last term on the right-hand side is at most $2 \exp \big( \! -\delta_0 (T-1)^{1/3} \big)$ provided that $T\ge 2$. For the first term on the right-hand side, by Lemma \ref{rescalingcts}, provided that $3M/T\le \eps''$ we have
\begin{align*}
\P(\exists t\in[T-1,T] : N_t(F\cap G_{M,t}^2) \ge 1 ) &\le \P(N_T( \overline{B(F,\eps'')}) \ge 1)\\
&\le \P(N_T(\overline{B(F,\eps'')}\cap G_{M,1}^2)\ge 1) + \P(N_T((G_{M,1}^2)^c)\ge 1).
\end{align*}
Due to \eqref{supexpanded}, we can apply Lemma \ref{supKnull} to tell us that the first term on the right-hand side above is at most $e^{-\delta_0 T^{1/3}/2}$, and Lemma \ref{AlwaysGMT} to tell us that the second term on the right-hand side is at most $e^{-\delta_0 T^{1/3}}$. Returning to \eqref{thetaUBpropfinal}, and applying the Borel-Cantelli lemma, we have
\[\P(\limsup_{t\to\infty} N_t(F) \ge 1 ) = 0.\]
This completes the proof.
\end{proof}

\appendix

\section{Deterministic bounds on the rate function}\label{det_bds_rate}

We use the same notation as in Section \ref{coupling_sec}. Our main aim in this section is to prove Proposition \ref{detboundratefn} and Lemma \ref{finaldetRbd}, showing that the bounds obtained in Section \ref{coupling_sec}, in terms of $\mathcal E^+_X(I_j,\Gamma_{M,T}(f,n),T)$, look something like the growth rate seen in our main theorem. This work involves tedious approximations of sums and integrals. Most of the work is in bounding $R^*$ in terms of $R^+$ and $R^-$, which is done using the following lemma. Throughout this section we write $R^-_X(j) = R^-_X(I_j,\Gamma_{M,T}(f,n),T)$ and similarly for $R^+_X$, $R^-_Y$ and $R^+_Y$.


\begin{lem}\label{detRbd}
Suppose that $M>1$, $n\ge 2M$, $f|_{I_j}\in G_M^2|_{I_j}$, $j\ge n^{1/2}$ and $s\in I_j$. Then
\begin{equation}\label{detRbdeq}
R_X^+(j) - \delta_{M,T}(j,n) \le R_X^*(f(s)) \le R_X^-(j) + \delta_{M,T}(j,n)
\end{equation}
where
\[\delta_{M,T}(j,n) = \big(6M^3n^{1/2}+\sfrac{2M^2n}{T}\big)\big(f_X(\sfrac{j+1}{n})-f_X(\sfrac{j}{n})\big) + Mn^{1/2}\big(f_Y(\sfrac{j+1}{n})-f_Y(\sfrac{j}{n})\big) + \sfrac{7M^3}{n^{3/2}} + \sfrac{3M^3n}{T}.\]
Moreover,
\begin{equation}\label{sumdeltaeq}
\sum_{j=\lceil\sqrt n\rceil}^{n-1}\frac{\delta_{M,T}(j,n)}{n} \le \frac{14M^4}{n^{1/2}}+\frac{5M^3 n}{T}.
\end{equation}
\end{lem}

\begin{proof}
We begin with the upper bound in \eqref{detRbdeq}, and claim first that for any $j\in\{0,1,\ldots,n-1\}$ we have
\begin{equation}\label{fytofxclaim}
f_Y\Big(\frac{j}{n}\Big) \le \Big(R_X^-(j) + \frac{1}{2}\Big)\Big(f_X\Big(\frac{j+1}{n}\Big)+\frac{1}{n^2}+\frac{1}{T}\Big) + \frac{1}{n^2}.
\end{equation}
To see why this is true, by the definition of $R^-_X(j)$, for any $\eps>0$ we may take $g\in\Gamma_{M,T}(f,n)$ and $s\in I_j$ such that
\[R_X(Tg(s))\le R^-_X(j)+\eps,\]
and then
\[\frac{g_Y(s)+1/T}{g_X(s)+1/T} - \frac{1}{2} \le R_X(Tg(s)) \le R_X^-(j) + \eps.\]
Noting that $g_Y(s)\ge g_Y(\frac{j}{n})\ge f_Y(\frac{j}{n})-1/n^2$ and $g_X(s)\le g_X(\frac{j+1}{n})\le f_X(\frac{j+1}{n})+1/n^2$, we see that
\[\frac{f_Y(\frac{j}{n})-1/n^2 + 1/T}{f_X(\frac{j+1}{n})+1/n^2+1/T} - \frac{1}{2} \le R_X^-(j) + \eps.\]
Since $\eps>0$ was arbitrary, the left-hand side must in fact be at most $R_X^-(j)$, and then rearranging gives \eqref{fytofxclaim}.

We now aim to bound $R_X^*(f(s))$ above for $s\in I_j$. We concentrate first on the case that $f_Y(s)>f_X(s)$. Whenever this holds, using \eqref{fytofxclaim},
\begin{align*}
R_X^*(f(s)) = \frac{f_Y(s)}{f_X(s)}-\frac12 &= \frac{f_Y(\frac{j}{n})}{f_X(s)}-\frac12 + \frac{f_Y(s)-f_Y(\frac{j}{n})}{f_X(s)}\\
&\le \frac{\big(R_X^-(j)+1/2\big)\big(f_X(\frac{j+1}{n})+1/n^2+1/T\big) + \frac{1}{n^2}}{f_X(s)}-\frac12 + \frac{f_Y(s)-f_Y(\frac{j}{n})}{f_X(s)}.
\end{align*}
Writing
\begin{multline*}
\Big(R_X^-(j)+\frac{1}{2}\Big)\Big(f_X\big(\sfrac{j+1}{n}\big)+\frac{1}{n^2}+\frac{1}{T}\Big)\\
= \Big(R_X^-(j)+\frac{1}{2}\Big)f_X(s) + \Big(R_X^-(j)+\frac{1}{2}\Big)\Big(f_X\big(\sfrac{j+1}{n}\big)-f_X(s)+\frac{1}{n^2}+\frac{1}{T}\Big)
\end{multline*}
and substituting this into the bound above, we have (for $f_Y(s)> f_X(s)$)
\[R_X^*(f(s))\le R_X^-(j) + \frac{\big(R_X^-(j)+\frac{1}{2}\big)\big(f_X(\frac{j+1}{n})-f_X(s)+\frac{1}{n^2}+\frac{1}{T}\big) + \frac{1}{n^2} + f_Y(s)-f_Y(\frac{j}{n})}{f_X(s)}.\]
The first term on the right-hand side is the important one, and we now aim to bound the other terms. Since $f|_{I_j}\in G_M^2|_{I_j}$, we have $f_X(s)\ge s/M$, and since also $f\in\Gamma_{M,T}(f,n)$,
\begin{equation}\label{Rminustrivbd}
R_X^-(j)\le R_X(Tf(s)) \le \frac{Ms+1/T}{s/M}-\frac{1}{2} = M^2 + \frac{M}{sT} - \frac{1}{2},
\end{equation}
so
\[R_X^*(f(s))\le R_X^-(j) + \frac{\big(M^2 + \frac{M}{sT}\big)\big(f_X(\frac{j+1}{n})-f_X(\frac{j}{n})+\frac{1}{n^2}+\frac{1}{T}\big) + \frac{1}{n^2} + f_Y(\frac{j+1}{n})-f_Y(\frac{j}{n})}{s/M}.\]
This is true in the case $f_Y(s)>f_X(s)$, but when $f_Y(s)\le f_X(s)$ we have $R_X^*(f(s))=1/2\le R_X^-(j)$, so the inequality above trivially holds in that case too. Taking $s\ge j/n\ge n^{-1/2}$, and combining some of the terms, we obtain the upper bound in \eqref{detRbdeq}.

The lower bound in \eqref{detRbdeq} is similar. We choose $g\in\Gamma_{M,T}(f,n)$ and $s\in[\frac{j}{n},\frac{j+1}{n}]$ such that $R_X(Tg(s))\ge R^+_X(j)-1/n^2$. If $R_X(Tg(s))=1/2$ then $R_X^+(j)\le 1/2+1/n^2$, so the lower bound on that interval is trivial; we may therefore assume that $R_X(Tg(s))>1/2$ and then we have a similar bound to \eqref{fytofxclaim}:
\begin{equation}\label{fytofxclaim2}
f_Y\Big(\frac{j+1}{n}\Big) \ge \Big(R_X^+(j) + \frac{1}{2}\Big)\Big(f_X\Big(\frac{j}{n}\Big)-\frac{1}{n^2}\Big) - \frac{1}{n^2} - \frac{1}{T}.
\end{equation}
We then apply this essentially as in the proof of the upper bound to obtain
\[R_X^*(f(s))\ge R_X^+(j) - \frac{\big(R_X^+(j)+\frac{1}{2}\big)\big(f_X(s)-f_X(\frac{j}{n})+\frac{1}{n^2}\big) + \frac{1}{n^2} + \frac{1}{T} + f_Y(\frac{j+1}{n}) - f_Y(s)}{f_X(s)}.\]
In place of \eqref{Rminustrivbd} we must use the slightly more involved bound, for $j\ge \sqrt n$ and $n\ge 2M$,
\[R^+_X(j) \le \frac{M\frac{j+1}{n} + \frac{1}{n^2} + \frac{1}{T}}{\frac{1}{M}\frac{j}{n} - \frac{1}{n^2} + \frac{1}{T}} - \frac{1}{2} \le \frac{3M\frac{j}{n} + \frac{1}{T}}{\frac{j}{2Mn}} - \frac{1}{2} = 6M^2 + \frac{2Mn}{jT} - \frac{1}{2} \le 6M^2 + \frac{2M\sqrt n}{T} - \frac{1}{2}.\]
Applying this and taking $s\ge j/n\ge n^{-1/2}$, and combining terms, gives the lower bound in \eqref{detRbdeq}.

To prove \eqref{sumdeltaeq}, summing over $j\ge n^{1/2}$ and telescoping gives
\[\sum_{j=\lceil\sqrt n\rceil}^{n-1}\frac{\delta_{M,T}(j,n)}{n} \le \Big(\frac{6M^3}{n^{1/2}}+\frac{2M^2}{T}\big)f_X(1) + \frac{M}{n^{1/2}}f_Y(1) + \frac{7M^3}{n^{3/2}} + \frac{3M^3n}{T}\]
Using that $f_X(1)\le M$ and $f_Y(1)\le M$, and combining terms, gives the result.
\end{proof}

\subsection{Proof of Proposition \ref{detboundratefn}}\label{dbrf_sec}

We first give a lemma which handles the cross-term that appears when multiplying out the quadratics involved in Proposition \ref{detboundratefn}.

\begin{lem}\label{detcrosstermbound}
Suppose that $f\in \PL_n^2 \cap G_M^2$, $\theta\in(0,1]$, $M>1$ and $n\ge 2M$. Then for any $k\in\{\lceil\sqrt n\rceil,\ldots,\lfloor \theta n\rfloor-1\}$,
\[\int_{k/n}^{\lfloor\theta n\rfloor/n} \sqrt{R_X^*(f(s))f'_X(s)} ds \ge \sum_{j=k}^{\lfloor \theta n\rfloor-1} \sqrt{\frac{R_X^+(j)}{n}(x^+_{j+1}-x^-_j)} - \frac{8M^{5/2}}{n^{1/4}} - \frac{4M^2n^{1/2}}{T^{1/2}}.\]
\end{lem}

\begin{proof}
Take $s\in[\frac{j}{n},\frac{j+1}{n}]$ for some $j\in\{0,1,\ldots,n-1\}$. For $a\in\R$, write $a_+$ to mean $\max\{a,0\}$. Note that 
since $f\in\PL_n$, we have
\[f'_X(s) = n(f(\sfrac{j+1}{n})-f(\sfrac{j}{n})) \ge n\big(x_{j+1}^+-1/n^2 - x_j^--1/n^2\big)_+.\]
Using the elementary inequality $\sqrt{(a-b)_+} \ge a^{1/2}-b^{1/2}$ valid for all $a,b\ge0$, we obtain
\[\sqrt{f'_X(s)} \ge n^{1/2}(x_{j+1}^+ - x_j^-)^{1/2} - \sqrt2 n^{-1/2}.\]
Thus
\[\int_{j/n}^{(j+1)/n} \sqrt{R_X^*(f(s))f'_X(s)} ds \ge \big((x_{j+1}^+ - x_j^-)^{1/2} - \sqrt2/n\big)\int_{j/n}^{(j+1)/n} \sqrt{nR_X^*(f(s))} ds\]
and since $f\in G_M^2$, $R_X^*(f(s))\le M^2$ for all $s>0$, so
\[\int_{j/n}^{(j+1)/n} \sqrt{R_X^*(f(s))f'_X(s)} ds \ge (x_{j+1}^+ - x_j^-)^{1/2}\int_{j/n}^{(j+1)/n} \sqrt{nR_X^*(f(s))} ds - \sqrt2 Mn^{-3/2}.\]
We now use the lower bound in \eqref{detRbdeq} to see that for $j\ge \sqrt n$,
\[\int_{j/n}^{(j+1)/n} \hspace{-1mm} \sqrt{nR_X^*(f(s))} ds \ge \int_{j/n}^{(j+1)/n}\hspace{-1mm} \sqrt{(nR_X^+(j)-n\delta_{M,T}(j,n))_+} \,ds = \sqrt{\Big(\frac{R_X^+(j)}{n}-\frac{\delta_{M,T}(j,n)}{n}\Big)_+}.\]
Again using $\sqrt{(a-b)_+} \ge a^{1/2}-b^{1/2}$, we therefore have, for $j\ge \sqrt n$,
\begin{align}
\int_{j/n}^{(j+1)/n} \sqrt{R_X^*(f(s))f'_X(s)} ds 
&\ge \sqrt{\frac{R_X^+(j)}{n}\big(x_{j+1}^+ - x_j^-\big)} - \frac{\sqrt2 M}{n^{3/2}} - \sqrt{\frac{\delta_{M,T}(j,n)}{n}}(x_{j+1}^+ - x_j^-)^{1/2}.\label{jintsqrtlb}
\end{align}

By the Cauchy-Schwartz inequality,
\begin{align*}
\sum_{j=k}^{n-1}\sqrt{\frac{\delta_{M,T}(j,n)}{n}}(x_{j+1}^+ - x_j^-)^{1/2} &\le \bigg(\sum_{j=k}^{n-1}\frac{\delta_{M,T}(j,n)}{n} \sum_{i=k}^{n-1}(x_{i+1}^+-x_i^-)\bigg)^{1/2}\\
&\le \bigg(\sum_{j=k}^{n-1}\frac{\delta_{M,T}(j,n)}{n} (f_X(1)+1/n)\bigg)^{1/2}
\end{align*}
and applying \eqref{sumdeltaeq}, together with the fact that $f_X(1)\le M$, gives
\[\bigg(\sum_{j=k}^{n-1}\frac{\delta_{M,T}(j,n)}{n}(f(1)+1/n)\bigg)^{1/2} \le \Big(\frac{28M^5}{n^{1/2}}+\frac{10M^4 n}{T}\Big)^{1/2} \le \frac{6M^{5/2}}{n^{1/4}} + \frac{4M^2n^{1/2}}{T^{1/2}}.\]
Summing \eqref{jintsqrtlb} over $k\le j \le \lfloor \theta n\rfloor -1$ and substituting the above bound gives the result.
\end{proof}

We can now prove our main proposition for this section.

\begin{proof}[Proof of Proposition \ref{detboundratefn}]
We first claim that for each $j=0,\ldots,n-1$ we have
\begin{equation}\label{simplifyE+}
\mathcal E^+_X(I_j,\Gamma_{M,T}(f,n),T) \ge \frac{2R_X^-(j)}{n} - 2\sqrt{\frac{2R_X^+(j)}{n}(x^+_{j+1}-x^-_j)} + x^-_{j+1} - x^+_j - \frac{4}{n^2}.
\end{equation}
Indeed, in either the $X+$ case or the $X-$ case, this follows directly from the definition of $\mathcal E^+_X$, even without the $4/n^2$ error term on the right-hand side. If we are in neither the $X+$ nor the $X-$ case, then $2R_X^-(j)/n \le x^+_{j+1}-x^-_j$ and $2R_X^+(j)/n \ge x^-_{j+1}-x^+_j$, so
\begin{align*}
&\frac{2R_X^-(j)}{n} - 2\sqrt{\frac{2R_X^+(j)}{n}(x^+_{j+1}-x^-_j)} + x^-_{j+1} - x^+_j\\
&\hspace{30mm}\le x^+_{j+1}-x^-_j -2\sqrt{\big((x^-_{j+1} - x^+_j)\vee0\big)(x^+_{j+1}-x^-_j)} + \big(x^-_{j+1} - x^+_j\big)\vee 0\\
&\hspace{30mm}= \Big(\sqrt{x^+_{j+1}-x^-_j} - \sqrt{\big(x^-_{j+1} - x^+_j\big)\vee 0}\Big)^2
\end{align*}
and using $\sqrt{(a-b)\vee0} \ge a^{1/2}-b^{1/2}$ we have
\[\sqrt{\big(x^-_{j+1} - x^+_j\big)\vee 0} \ge \sqrt{\big(x_{j+1}^+-x_j^- -4/n^2\big)\vee0}\ge \sqrt{x_{j+1}^+-x_j^-} - 2/n,\]
so in this case
\[\frac{2R_X^-(j)}{n} - 2\sqrt{\frac{2R_X^+(j)}{n}(x^+_{j+1}-x^-_j)} + x^-_{j+1} - x^+_j \le 4/n^2 \le \mathcal E^+_X(I_j,\Gamma_{M,T}(f,n),T) + 4/n^2\]
and the claim is proved.

Now write $K=\lfloor \theta n\rfloor$. By the upper bound in \eqref{detRbdeq}, for any $j\ge n^{1/2}$,
\[\int_{\frac{j}{n}}^{\frac{j+1}{n}} R_X^*(f(s)) ds \le \frac{R_X^-(j)}{n} + \frac{\delta_{M,T}(j,n)}{n}.\]
Summing over $k\le j\le K -1$ and applying \eqref{sumdeltaeq} gives
\[\int_{k/n}^{K/n} R_X^*(f(s)) ds \le \sum_{j=k}^{K-1} \frac{R_X^-(j)}{n} + \frac{14M^4}{n^{1/2}} + \frac{5M^3n}{T}.\]
Lemma \ref{detcrosstermbound} gives that
\[\int_{k/n}^{K/n} \sqrt{R_X^*(f(s))f'_X(s)} ds \ge \sum_{j=k}^{K-1} \sqrt{\frac{R_X^+(j)}{n}(x^+_{j+1}-x^-_j)} - \frac{8 M^{5/2}}{n^{1/4}} - \frac{4M^2n^{1/2}}{T^{1/2}}.\]
Also
\[\int_{k/n}^{K/n} f'_X(s) ds \le \sum_{j=k}^{K-1} \big(f(\sfrac{j+1}{n})-f(\sfrac{j}{n})\big) \le \sum_{j=k}^{K-1} \big(x^-_{j+1}+1/n^2 - x^+_j+1/n^2\big) \le \sum_{j=k}^{K-1} \big(x^-_{j+1} - x^+_j\big) + 2/n.\]
Putting these bounds together with \eqref{simplifyE+}, and combining error terms, we obtain
\[\int_{k/n}^{K/n} \Big(\sqrt{2R_X^*(f(s))} - \sqrt{f'_X(s)}\Big)^2 ds \le \sum_{j=k}^{K-1} \mathcal E^+_X(I_j,\Gamma_{M,T}(f,n),T) + O\Big(\frac{M^4}{n^{1/4}} + \frac{M^3n}{T^{1/2}}\Big),\]
completing the proof.
\end{proof}

\subsection{Proof of Lemma \ref{finaldetRbd}}\label{fdrsec}

The proof of Lemma \ref{finaldetRbd} is relatively straightforward. The upper and lower bounds are very similar, but quite lengthy, so we separate them out into two proofs.

\begin{proof}[Proof of Lemma \ref{finaldetRbd}: upper bound]
Write $K=\lfloor \theta n\rfloor$. We split the integral from $0$ to $\theta$ into three parts:
\begin{equation}\label{splitintegralRbound}
\int_0^\theta R(Tg(s)) ds \le \int_0^{3MT^{-2/3}} R(Tg(s))ds + \int_{3MT^{-2/3}}^{\lceil\sqrt n\rceil/n} R(Tg(s))ds + \int_{\lceil\sqrt n\rceil/n}^\theta R(Tg(s))ds.
\end{equation}
For the first term on the right-hand side, note that for any $g\in G_{M,T}^2$ and $s\le 3MT^{-2/3}$,
\[R(Tg(s)) \le \frac{MT(s + 2T^{-2/3})+1}{1} \le MT(3M+2)T^{-2/3}+1 \le 6M^2T^{1/3}.\]
For the second term on the right-hand side of \eqref{splitintegralRbound}, we note that for $s>3MT^{-2/3}$ we have $2T^{-2/3}\le \frac{2}{3}s/M$ and therefore, since $g\in G_{M,T}^2$,
\begin{equation*}
R(Tg(s)) \le \frac{MT(s + 2T^{-2/3})+1}{T(s/M - 2T^{-2/3})} \le \frac{MT(s+\frac{2s}{3M})+1}{T\frac{s}{3M}} \le 3M^2 \Big(1+\frac{2}{3M}\Big) + \frac{3M}{Ts} \le 6M^2 + T^{-1/3}.
\end{equation*}
We now consider the last term in \eqref{splitintegralRbound}, but work with any $k\ge\lceil\sqrt n\rceil$; since $g\in\Gamma_{M,T}(f,n)$, by definition of $R_X^+$ and $R_Y^+$ we have
\begin{align}
\int_{k/n}^\theta R(Tg(s))ds &= \int_{k/n}^\theta R_X(Tg(s))ds + \int_{k/n}^\theta R_Y(Tg(s))ds\nonumber\\
&\le \sum_{j=k}^{K} \int_{\frac{j}{n}}^{\frac{j+1}{n}} R_X^+(I_j,\Gamma_{M,T}(f,n),T) ds + \sum_{j=k}^{K} \int_{\frac{j}{n}}^{\frac{j+1}{n}} R_Y^+(I_j,\Gamma_{M,T}(f,n),T) ds\nonumber\\
&= \sum_{j=k}^{K} \frac{R_X^+(j)}{n} + \sum_{j=k}^{K} \frac{R_Y^+(j)}{n}.\label{generalkfinaldetRbd}
\end{align}
By the lower bound in \eqref{detRbdeq}, for any $s\in[\frac{j}{n},\frac{j+1}{n}]$,
\[R_X^+(j) \le R_X^*(f(s)) + \delta_{M,T}(j,n),\]
so using \eqref{sumdeltaeq} and the fact that $f$ is $M$-good,
\begin{multline*}
\sum_{j=k}^{K} \frac{R_X^+(j)}{n} \le \int_{k/n}^{(K+1)/n} R_X^*(f(s)) ds + \sum_{j=k}^{K}\frac{\delta_{M,T}(j,n)}{n}\\
\le \int_{k/n}^{(K+1)/n} R_X^*(f(s)) ds + O\Big(\frac{M^4}{n^{1/2}}+\frac{M^3 n}{T}\Big)\\
 \le \int_{k/n}^{K/n} R_X^*(f(s)) ds + O\Big(\frac{M^4}{n^{1/2}}+\frac{M^3 n}{T}\Big).
\end{multline*}
By symmetry we also have the same statement with $Y$ in place of $X$.

Substituting these bounds into \eqref{generalkfinaldetRbd} gives the upper bound in the second part of the lemma. For the first part of the lemma, returning to \eqref{splitintegralRbound} and substituting in our estimates above for the three terms on the right-hand side, we have
\[\int_0^\theta R(Tg(s)) ds \le \int_{k/n}^{K/n} R^*_X(f(s))ds + O\Big(\frac{M^3}{T^{1/3}} + \frac{M^4}{n^{1/2}} + \frac{1}{T^{1/3}n^{1/2}} + \frac{M^3n}{T}\Big).\]
Since $R^*_X(f(s))\ge0$ for all $s$, the result follows.
\end{proof}

\begin{proof}[Proof of Lemma \ref{finaldetRbd}: lower bound]
Note that, again writing $K=\lfloor \theta n\rfloor$ but now with any $k$ satisfying $\lceil \sqrt n\rceil\le k \le K$,
\[\int_{k/n}^\theta R(Tg(s)) ds \ge \int_{k/n}^{K/n} R(Tg(s))ds.\]
Since $g\in\Gamma_{M,T}(f,n)$, by definition of $R_X^-$ and $R_Y^-$ we have
\begin{align*}
\int_{k/n}^\theta R(Tg(s))ds &= \int_{k/n}^\theta R_X(Tg(s))ds + \int_{k/n}^\theta R_Y(Tg(s))ds\\
&\ge \sum_{j=k}^{K-1} \int_{\frac{j}{n}}^{\frac{j+1}{n}} R_X^-(I_j,\Gamma_{M,T}(f,n),T) ds + \sum_{j=k}^{K-1} \int_{\frac{j}{n}}^{\frac{j+1}{n}} R_Y^-(I_j,\Gamma_{M,T}(f,n),T) ds\\
&= \sum_{j=k}^{K-1} \frac{R_X^-(j)}{n} + \sum_{j=k}^{K-1} \frac{R_Y^-(j)}{n}.
\end{align*}
By the upper bound in \eqref{detRbdeq}, for any $s\in[\frac{j}{n},\frac{j+1}{n}]$, we have $R_X^-(j) \ge R_X^*(f(s)) - \delta_{M,T}(j,n)$, so using \eqref{sumdeltaeq} and the fact that $f$ is $M$-good,
\begin{align*}
\sum_{j=k}^{K-1} \frac{R_X^-(j)}{n} \ge \int_{k/n}^{K/n} R_X^*(f(s)) ds - \sum_{j=k}^{K-1}\frac{\delta_{M,T}(j,n)}{n} &\ge \int_{k/n}^{K/n} R_X^*(f(s)) ds - O\Big(\frac{M^4}{n^{1/2}} + \frac{M^3n}{T}\Big).
\end{align*}
By symmetry we also have
\[\sum_{j=k}^{K-1} \frac{R_Y^-(j)}{n} \ge \int_{k/n}^{K/n} R_Y^*(f(s)) ds - O\Big(\frac{M^4}{n^{1/2}} + \frac{M^3n}{T}\Big).\]
Combining these bounds gives the result.
\end{proof}

\subsection{Proof of Lemma \ref{Esum2mlb}}\label{det_bds_rate_2}

The main difference between Lemma \ref{Esum2mlb} and our previous deterministic bounds on the rate function is that it requires us to consider more general time intervals than those of the form $[j/n, (j+1)/n]$. Lemma \ref{E+2mbd} will do most of the work required, and uses the uniform structure of $\Lambda_{M,T}(f,n)$ to get better bounds than are possible for $\Gamma_{M,T}(f,n)$.

\begin{lem}\label{E+2mbd}
Suppose that $M,T>1$, $n\ge 2M$ and $f\in PL_n^2 \cap G_M^2$. Then for any $j\in\{\lceil\sqrt n\rceil,\ldots,n-1\}$ and $u,v$ such that $\frac{j}{n}\le u<v\le\frac{j+1}{n}$,
\begin{multline*}
\int_u^v \Big(\sqrt{2R^*_X(f(s))} - \sqrt{f_X'(s)}\Big)^2 ds\\
\le \mathcal E_X^+\big([u,v],\Lambda_{M,T}(f,n),T\big) + \frac{6\delta_{M,T}(j,n)}{n} + 2\sqrt{\frac{2\delta_{M,T}(j,n)}{n}\big(f_X(\sfrac{j+1}{n})-f_X(\sfrac{j}{n})\big)} + \frac{14M}{n^{3/2}}.
\end{multline*}
\end{lem}

\begin{proof}
As in the proof of Lemma \ref{qlblem}, for $I\subset[0,1]$ we write $\hat R^-_X(I)$ as shorthand for the quantity $R^-_X(I,\Lambda_{M,T}(f,n),T)$, and similarly for $\hat R^+_X(I)$, $\hat R^-_Y(I)$ and $\hat R^+_Y(I)$. We also write, for $s\in[0,1]$,
\[x^-(s) = x^-(s,\Lambda_{M,T}(f,n)) = \inf\{g_X(s) : g\in\Lambda_{M,T}(f,n)\}\]
and similarly for $x^+(s)$, $y^-(s)$ and $y^+(s)$.

By \eqref{Rtrap2} and the fact that $f$ is linear on $I_j$ (and therefore on $[u,v]$), we have
\begin{align}
&\int_u^v \Big(\sqrt{2R^*_X(f(s))}-\sqrt{f'_X(s)}\Big)^2 ds\nonumber\\
&= 2\int_u^v R^*_X(f(s)) ds + \int_u^v f_X'(s) ds - 2\int_u^v \sqrt{2R^*_X(f(s))f_X'(s)} ds\nonumber\\
&\le 2\hat R_X^-(I_j)(v-u) + 2\delta_{M,T}(j,n)(v-u)  + f_X(v)-f_X(u) - 2\int_u^v \sqrt{2R^*_X(f(s))\frac{f_X(v)-f_X(u)}{v-u}}ds.\label{detintbdeq1}
\end{align}
Applying \eqref{Rtrap2} and using the elementary inequality $\sqrt{(a-b)\vee 0} \ge \sqrt a -\sqrt b$, valid for all $a,b\ge 0$, for any $s\in I_j$ we have
\begin{multline*}
\sqrt{R^*_X(f(s))}\ge \sqrt{\big(\hat R_X^+(I_j)-\delta_{M,T}(j,n)\big)\vee 0}\\
\ge \sqrt{\hat R_X^+(I_j)} - \sqrt{\delta_{M,T}(j,n)} \ge \sqrt{\hat R_X^+([u,v])} - \sqrt{\delta_{M,T}(j,n)}
\end{multline*}
so we have
\begin{align}\label{crosslb}
&\int_u^v \sqrt{2R^*_X(f(s))\frac{f_X(v)-f_X(u)}{v-u}}ds \nonumber \\
&\ge \int_u^v \Big(\sqrt{2\hat R_X^+([u,v])} - \sqrt{2\delta_{M,T}(j,n)}\Big)\sqrt{\frac{f_X(v)-f_X(u)}{v-u}}ds \nonumber \\
&\ge \sqrt{2\hat R_X^+([u,v]) (f_X(v)-f_X(u))(v-u)}- \sqrt{ \frac{2\delta_{M,T}(j,n)}{n} \big(f_X(\sfrac{j+1}{n})-f_X(\sfrac{j}{n})\big)}.
\end{align}
Using again that $\sqrt{(a-b)\vee 0} \ge \sqrt a -\sqrt b$ we have
\[\sqrt{f_X(v)-f_X(u)} \ge \sqrt{(x^+(v) - 1/n^2 - (x^-(u)+1/n^2))\vee0} \ge \sqrt{x^+(v)-x^-(u)} - \sqrt{2/n^2},\]
and since $f$ is $M$-good, and therefore by \eqref{Rtrap2} $R^+_X([u,v]) \le M^2+\delta_{M,T}(j,n)$, we deduce that
\[ \sqrt{2\hat R_X^+([u,v]) (f_X(v)-f_X(u))(v-u)} \ge \sqrt{2\hat R_X^+([u,v]) (x^+(v)-x^-(u))(v-u)} - \sfrac{2 \sqrt{M^2+\delta_{M,T}(j,n)}}{n^{3/2}}.\]
Substituting this into \eqref{crosslb}, and using that
\[\sqrt{M^2+\delta_{M,T}(j,n)} \le \sqrt{M^2} + \sqrt{\delta_{M,T}(j,n)} \le M + \delta_{M,T}(j,n) + 1\]
gives that
\begin{align*} 
&\int_u^v \sqrt{2R^*_X(f(s))\frac{f_X(v)-f_X(u)}{v-u}}ds \\
&\ge \sqrt{2\hat R_X^+([u,v])(v-u)(x^+(v)-x^-(u))} - \sqrt{ \sfrac{2\delta_{M,T}(j,n)}{n} \big(f_X(\sfrac{j+1}{n})-f_X(\sfrac{j}{n})\big)} - \sfrac{2(M+\delta_{M,T}(j,n)+1)}{n^{3/2}}.
\end{align*}
Substituting this bound into \eqref{detintbdeq1} and using that $\hat R_X^-(I_j)\le \hat R_X^-([u,v])$ and $v-u\le 1/n$, we obtain
\begin{align*}
&\int_u^v \Big(\sqrt{2R^*_X(f(s))}-\sqrt{f'_X(s)}\Big)^2 ds\\
&\le 2\hat R_X^-([u,v])(v-u) + x^-(v) - x^+(u) - 2\sqrt{2\hat R_X^+([u,v])(v-u)(x^+(v)-x^-(u))}\\
&\hspace{15mm} + \frac{6\delta_{M,T}(j,n)}{n} + \frac{2}{n^2} + 2\sqrt{\frac{2\delta_{M,T}(j,n)}{n}\big(f_X(\sfrac{j+1}{n})-f_X(\sfrac{j}{n})\big)} + \frac{8M}{n^{3/2}}.
\end{align*}
It then remains to note that, following exactly the same argument as \eqref{simplifyE+},
\begin{multline*}
\mathcal E_X^+\big([u,v],\Lambda_{M,T}(f,n),T\big)\\
\ge 2\hat R_X^-([u,v])(v-u) + x^-(v) - x^+(u) - 2\sqrt{2\hat R_X^+([u,v])(v-u)(x^+(v)-x^-(u))} - 4/n^2.
\end{multline*}
Combining error terms gives the result.
\end{proof}

It is now a relatively simple task to apply Lemma \ref{E+2mbd} to complete the proof of Lemma \ref{Esum2mlb}.

\begin{proof}[Proof of Lemma \ref{Esum2mlb}]
By symmetry it suffices to show that
\[\sum_{j=\lfloor an\rfloor}^{\lceil bn\rceil -1} \mathcal E_X^+(I_j\cap[a,b],\Lambda_{M,T}(f,n),T) \ge \int_a^b \Big(\sqrt{2R^*_X(f(s))} - \sqrt{f'_X(s)}\Big)^2 ds - O\Big(\frac{M^4}{n^{1/4}}+\frac{M^3 n}{T^{1/2}}\Big).\]
By Lemma \ref{E+2mbd},
\begin{align*}
&\sum_{j=\lfloor an\rfloor}^{\lceil bn\rceil -1} \mathcal E_X^+(I_j\cap[a,b],\Lambda_{M,T}(f,n),T)\\
&\hspace{20mm}\ge \sum_{j=\lfloor an\rfloor}^{\lceil bn\rceil -1} \int_{I_j\cap[a,b]} \Big(\sqrt{2R^*_X(f(s))} - \sqrt{f'_X(s)}\Big)^2 ds\\
&\hspace{40mm} - \sum_{j=\lfloor an\rfloor}^{\lceil bn\rceil -1}\Bigg(\frac{6\delta_{M,T}(j,n)}{n} + 2\sqrt{\frac{2\delta_{M,T}(j,n)}{n}\big(f_X(\sfrac{j+1}{n})-f_X(\sfrac{j}{n})\big)} + \frac{14M}{n^{3/2}}\Bigg).
\end{align*}
Note that 
\[\sum_{j=\lfloor an\rfloor}^{\lceil bn\rceil -1} \int_{I_j\cap[a,b]} \Big(\sqrt{2R^*_X(f(s))} - \sqrt{f'_X(s)}\Big)^2 ds = \int_a^b \Big(\sqrt{2R^*_X(f(s))} - \sqrt{f'_X(s)}\Big)^2 ds,\]
and by \eqref{sumdeltaeq}
\[\sum_{j=\lfloor an\rfloor}^{\lceil bn\rceil -1} \Big(\frac{6\delta_{M,T}(j,n)}{n} + \frac{14M}{n^{3/2}}\Big) = O\Big(\frac{M^4}{n^{1/2}} + \frac{M^3n}{T}\Big).\]
Finally, by Cauchy-Schwarz,
\[\sum_{j=\lfloor an\rfloor}^{\lceil bn\rceil -1} \sqrt{\frac{2\delta_{M,T}(j,n)}{n}\big(f_X(\sfrac{j+1}{n})-f_X(\sfrac{j}{n})\big)} \le \bigg(\sum_{i=\lfloor an\rfloor}^{\lceil bn\rceil -1} \frac{2\delta_{M,T}(i,n)}{n} \sum_{j=\lfloor an\rfloor}^{\lceil bn\rceil -1} \big(f_X(\sfrac{j+1}{n})-f_X(\sfrac{j}{n})\big)\bigg)^{1/2},\]
and using \eqref{sumdeltaeq} and the fact that $f\in G_M^2$, we see that
\[\sum_{j=\lfloor an\rfloor}^{\lceil bn\rceil -1} \sqrt{\frac{2\delta_{M,T}(j,n)}{n}\big(f_X(\sfrac{j+1}{n})-f_X(\sfrac{j}{n})\big)} = O\Big(\frac{M^{5/2}}{n^{1/4}}+\frac{M^2 n^{1/2}}{T^{1/2}}\Big).\]
Combining these estimates completes the proof.
\end{proof}

\section{Elementary bounds on compound Poisson processes: proof of Lemma \ref{CPPfollowtubecor}}\label{CPP_append}

Before proving Lemma \ref{CPPfollowtubecor}, the bulk of the work is done by the following intermediate result.

\begin{lem}\label{CPPfollowtubelem}
Suppose that $\delta,t,A>0$ and $R\ge 1/2$. Let $(X(s),s\ge0)$ be a compound Poisson process of rate $RT$ whose jumps are exponentially distributed with parameter $T$. Then for any $T$,
\[\P(|X(s)-As|<\delta \,\, \forall s\le t) \le \exp\Big(-tT(\sqrt R - \sqrt A)^2 + \delta \big|1-\sqrt{R/A}\big| T\Big),\]
and for $T>\frac{2A^{3/2}(4t+\delta)}{R^{1/2}\delta^{2}(A\wedge 1)^2}$,
\[\P(|X(s)-As|<\delta \,\, \forall s\le t)\ge \frac{1}{2}\exp\Big(-tT(\sqrt R - \sqrt A)^2 - \delta(A\wedge1) \big|1-\sqrt{R/A}\big| T\Big).\]
\end{lem}

\begin{proof}
For any $q<T$ and $s\ge0$, we have
\[\E[e^{qX(s)}] = \exp\Big(\frac{Rqs}{1-q/T}\Big).\]
Fix $a= T(1-\sqrt{R/A})$; then elementary calculations show that
\[\frac{\E[X(s)e^{aX(s)}]}{\E[e^{aX(s)}]} = As.\]
Let $(\sigma_s)_{s\ge0}$ be the natural filtration of $X$, and define a new probability measure $\mu$ by setting
\[\frac{d\mu}{d\P}\Big|_{\sigma_s} = \frac{e^{aX(s)}}{\E[e^{aX(s)}]} = \exp\Big(aX(s)-\frac{Ras}{1-a/T}\Big).\]
Then, by the definition of $\mu$, for any $\delta'>0$ we have
\[\P(|X(s)-As|<\delta'\,\,\,\,\forall s\le t) = \mu\Big[\exp\Big(-aX(t)+\frac{Rat}{1-a/T}\Big)\ind_{\{|X(s)-As|<\delta'\,\,\,\,\forall s\le t\}}\Big]\]
and using the bound $|X(t)-At|<\delta'$ and simplifying we obtain
\begin{multline}\label{muinitial}
\exp\Big(-tT(\sqrt R - \sqrt A)^2 - |a|\delta')\Big)\mu(|X(s)-As|<\delta'\,\,\,\,\forall s\le t)\\
\hspace{-35mm}\le \P(|X(s)-As|<\delta'\,\,\,\,\forall s\le t)\\
\le \exp\Big(-tT(\sqrt R - \sqrt A)^2 + |a|\delta')\Big)\mu(|X(s)-As|<\delta'\,\,\,\,\forall s\le t).
\end{multline}
The required upper bound follows immediately by taking $\delta'=\delta$. For the lower bound we take $\delta' = \delta(A\wedge1)$, and then it remains to bound $\mu(|X(s)-As|<\delta(A\wedge1)\,\,\,\,\forall s\le t)$ from below.

One may easily check that $(X(s)-As,s\ge0)$ is a martingale under $\mu$, and therefore by Jensen's inequality, $(e^{\nu(X(s)-As)},s\ge0)$ is a submartingale under $\mu$ for any $\nu<T-a$ (the upper bound on $\nu$ is required to ensure that the expectation is finite). By Doob's submartingale inequality, for any $\nu\in(0,T-a)$, we have
\begin{equation}\label{nuupper}
\mu(\exists s\le t : X(s)-As\ge\delta(A\wedge1)) = \mu\Big(\sup_{s\le t}e^{\nu(X(s)-As)} \ge e^{\nu\delta(A\wedge1)}\Big) \le \mu[e^{\nu(X(t)-At)}]e^{-\nu\delta(A\wedge1)}
\end{equation}
and for any $\nu<0$ we have
\begin{equation}\label{nulower}
\mu(\exists s\le t : X(s)-As\le-\delta(A\wedge1)) = \mu\Big(\sup_{s\le t}e^{\nu(X(s)-As)} \ge e^{-\nu\delta(A\wedge1)}\Big) \le \mu[e^{\nu(X(t)-At)}]e^{\nu\delta(A\wedge1)}.
\end{equation}
Now, for any $\nu<T-a$,
\[\mu[e^{\nu(X(t)-At)}] = \E[e^{(a+\nu)X(t)}]e^{-Rat/(1-a/T)-A\nu t} = \exp\Big(\frac{R(a+\nu)t}{1-(a+\nu)/T} - \frac{Rat}{1-a/T} - A\nu t\Big),\]
and simplifying we obtain
\[\mu[e^{\nu(X(t)-At)}] = \exp\Big(\frac{A\nu t}{1-\frac{\nu}{T}(A/R)^{1/2}}-A\nu t\Big) = \exp\Big(\frac{A\nu^2 t}{(R/A)^{1/2}T-\nu}\Big).\]
It is then easy to check that for $T>\frac{2A^{3/2}(4t+\delta)}{R^{1/2}\delta^{2}(A\wedge 1)^2}$, each of the probabilities in \eqref{nuupper} and \eqref{nulower} can be made smaller than $e^{-3/2}<1/4$ by choosing $\nu=\pm\frac{2}{\delta(A\wedge 1)}$. Thus we have
\[\mu(|X(s)-As|<\delta(A\wedge1)\,\,\,\,\forall s\le t) \ge 1/2\]
for such $T$. Substituting this into the lower bound in \eqref{muinitial}, using $\delta'=\delta(A\wedge1)$, gives the result.
\end{proof}

\noindent
Lemma \ref{CPPfollowtubecor} now requires us to deal with the position of our process at the endpoints of the intervals $I_j$.

\begin{proof}[Proof of Lemma \ref{CPPfollowtubecor}]
Let $A_t(a) = A-a/t$. Note that if $|X(s)-A_t(a)s|<\delta/2$ for all $s\le t$, then $|a+X(s)-As|<\delta$ for all $s\le t$ and $|a+X(t)-At|<\delta/2$. Thus
\[\P(|a+X(s)-As|<\delta \,\, \forall s\le t, \,\, |a+X(t)-At|<\delta/2) \ge \P(|X(s)-A_t(a)s|<\delta/2 \,\, \forall s\le t).\]
Lemma \ref{CPPfollowtubelem} tells us that the latter probability is at least
\begin{equation}\label{CPPfollowtubeeq}
\frac{1}{2}\exp\bigg(-tT(\sqrt{R} - \sqrt{A_t(a)})^2 - \frac{\delta(A_t(a)\wedge1)}{2} \Big|1-\Big(\frac{R}{A_t(a)}\Big)^{1/2}\Big| T\bigg).
\end{equation}
Using the fact that $(1-x)^{1/2}\ge 1-x^{1/2}$ for $x\in[0,1]$, we have
\begin{align*}
t(\sqrt{R} - \sqrt{A_t(a)})^2 &\le \Big(R + A+\frac{\delta}{2t}-2\sqrt{AR}\Big(1-\frac{\delta}{2tA}\Big)^{1/2}\Big)t\\
&\le (\sqrt R - \sqrt A)^2 t + \frac{\delta}{2} + \sqrt{2\delta Rt}.
\end{align*}
and
\[\frac{\delta (A_t(a)\wedge 1)}{2}\Big|1-\Big(\frac{R}{A_t(a)}\Big)^{1/2}\Big| \le \frac{\delta}{2} \Big(1+\Big(\frac{R}{A_t(a)}\Big)^{1/2}\Big) \le \frac{\delta}{2}(1+\sqrt{R}).\]
Substituting these estimates into \eqref{CPPfollowtubeeq} gives
\[\frac{1}{2}\exp\bigg(-tT(\sqrt R - \sqrt A)^2 - \delta \Big(1+\sqrt{R}\big(\sqrt{2t/\delta}+1/2\big)\Big) T\bigg),\]
as required.
\end{proof}

\section{Proofs of compactness and semicontinuity}\label{techsec}

\subsection{Compactness of $G_{M,T}^2$: proof of Lemma \ref{coverGMT}}\label{rel_cpct_sec}

The proof of Lemma \ref{coverGMT}, which says that for any $F\subset E^2$ we can cover $F\cap G_{M,T}^2$ in a nice way with small balls around piecewise linear functions, is straightforward. We directly construct piecewise linear approximations to an arbitrary function within $F\cap G_{M,T}^2$.

\begin{proof}[Proof of Lemma \ref{coverGMT}]
Suppose that $T\ge (4Mn)^{3/2}$ and take $h\in F\cap G_{M,T}^2$. Then define a function $g\in \PL_{n}^2$ by interpolating linearly between the values
\[g(j/n) = \lfloor n^2 h(j/n)\rfloor /n^2, \hspace{5mm} j=0,1,\ldots,n.\]
Then clearly
\[\Delta_n(g,h) < 1/n^2.\]
We claim that $d(g,h)\le 1/n$. To see this, take $s\in[0,1]$, and then fix $j\in\{0,1,\ldots,n-1\}$ such that $s\in[j/n,(j+1)/n]$. Then
\[g(s)\le g(\sfrac{j+1}{n}) \le h(\sfrac{j+1}{n})\]
and
\[g(s) \ge g(\sfrac{j}{n}) \ge h(\sfrac{j}{n})- 1/n^2\]
which, by the definition \eqref{metricdefn} of $d$, establishes the claim.

Next we claim that $g\in G_{4M}^2$. Since $h\in G_{M,T}^2$ we know that for any $j=1,2,\ldots,n$,
\[\frac{j}{Mn} - 2T^{-2/3} \le h(j/n) \le M\Big(\frac{j}{n}+2T^{-2/3}\Big)\]
and since $T\ge (4Mn)^{3/2}$ we obtain
\[\frac{j}{2Mn} \le \frac{j-1/2}{Mn} \le h(j/n) \le \frac{Mj+1/2}{n} \le \frac{2Mj}{n}.\]
But $h(j/n)-1/n^2\le g(j/n)\le h(j/n)$ so
\[\frac{j}{2Mn} - \frac{1}{n^2} \le g(j/n) \le \frac{2Mj}{n}\]
and since $n\ge 4M$, $j/(2Mn)-1/n^2 \ge j/(4Mn)$, which, given that $g$ interpolates linearly between these values, proves the claim.

Since the functions $g$ created in this way can take only finitely many values (namely integer multiples of $1/n^2$ with a maximum of at most $2M$) at the times $0,1/n,2/n,\ldots,1$, and interpolate linearly between these values, there are only finitely many possible such functions, and therefore the proof is complete.
\end{proof}

\subsection[Partial lower semi-continuity of $K$: proof of Proposition \ref{lowersemicts}]{Partial lower semi-continuity of $\tilde K$: proof of Proposition \ref{lowersemicts}}\label{lsc_sec}

To complete the proof of the lower bound in Section \ref{lower_bd_sec}, we need to prove a partial semi-continuity result about $\tilde K$, which was stated in Proposition \ref{lowersemicts}. We begin with a useful lemma which states that given continuity, convergence under $d$ implies convergence pointwise.

\begin{lem}\label{pointwiselem}
If $f\in E$ is continuous at $s$ and $d(f_n,f)\to 0$, then $f_n(s)\to f(s)$. Moreover, if $d(f_n,f)\to 0$, then $f_n(1)\to f(1)$ (regardless of whether $f$ is continuous at $1$).
\end{lem}

\begin{proof}
Fix $\eps>0$ and $s\in[0,1]$ such that $f$ is continuous at $s$. Then we can find $\delta>0$ such that $|f(u)-f(s)|<\eps/2$ for any $u\in[s-\delta,s+\delta]\cap[0,1]$. Choose $N$ such that $d(f_n,f)<(\eps/2)\wedge\delta$ for all $n\ge N$. By the definition \eqref{metricdefn} of $d$, this means that
\[f((s-\delta)\vee0) - \eps/2 \le f_n(s) \le f((s+\delta)\wedge1) + \eps/2.\]
Then we have
\[f(s)-\eps \le f((s-\delta)\vee0) - \eps/2 \le f_n(s) \le f((s+\delta)\wedge1) + \eps/2 \le f(s)+\eps\]
and since $\eps>0$ was arbitrary, we have shown that $f_n(s)\to f(s)$.

For the second part of the lemma, simply note that by the definition of $d$, if $d(f_n,f)<\eps$ then $|f_n(1)-f(1)|<\eps$.
\end{proof}

We now show that when $f_n$ is the piecewise linear interpolation to $f$, the cross-terms that appear when multiplying out the quadratic terms in $\tilde K$ satisfy a semicontinuity property.

\begin{lem}\label{liminfsqrt}
Suppose that $0\le a<b\le 1$ and that $f\in G_M^2$ for some $M$. Let $f_n$ be the function in $\PL_n$ constructed by setting $f_n(j/n)=f(j/n)$ for each $j=0,\ldots,n$ and interpolating linearly. Then
\[\liminf_{n\to\infty} \int_a^b \sqrt{R^*_X(f_n(s)) f_{n,X}'(s)} \d s \ge \int_a^b \sqrt{R^*_X(f(s))f_X'(s)}\d s\]
where we write $f_n(s) = (f_{n,X}(s),f_{n,Y}(s))$.
\end{lem}

\begin{proof}
We carry out the proof when $a=0$ and $b=1$; the general case follows by including $\ind_{\{s\in[a,b]\}}$ throughout.

Note that
\begin{align*}
\int_0^1 \sqrt{R^*_X(f_n(s)) f_{n,X}'(s)} \d s &= \int_0^1 \sum_{i=1}^n \ind_{\{s\in[\frac{i-1}{n},\frac{i}{n})\}} \sqrt{R^*_X(f_n(s))} \sqrt{n(f_{n,X}(\sfrac{i}{n})-f_{n,X}(\sfrac{i-1}{n}))} \d s\\
&= \int_0^1 \sum_{i=1}^n \ind_{\{s\in[\frac{i-1}{n},\frac{i}{n})\}} \sqrt{R^*_X(f_n(s))} \sqrt{n(f_X(\sfrac{i}{n})-f_X(\sfrac{i-1}{n}))} \d s\\
&\ge \int_0^1 \sum_{i=1}^n \ind_{\{s\in[\frac{i-1}{n},\frac{i}{n})\}}\inf_{u\in[\frac{i-1}{n},\frac{i}{n}]} \sqrt{R^*_X(f_n(u))}\sqrt{n(f_X(\sfrac{i}{n})-f_X(\sfrac{i-1}{n}))}\,ds.
\end{align*}
Since $f$ is continuous almost everywhere, by Lemma \ref{pointwiselem}, $f_n(u)\to f(u)$ almost everywhere. Since $f$ is $M$-good, and $R^*_X$ is continuous away from $0$, $R^*_X(f_n(u))\to R^*_X(f(u))$ for almost every $u\in[0,1]$. Since $f$ is differentiable almost everywhere, we deduce that the integrand above converges to $\sqrt{R^*_X(f(s))f_X'(s)}$ for almost every $s\in[0,1]$. It is also bounded above by
\[F_n(s) = \sum_{i=1}^n \ind_{\{s\in[\frac{i-1}{n},\frac{i}{n})\}} M \bigg(n\Big(f_X\Big(\frac{i}{n}\Big) - f_X\Big(\frac{i-1}{n}\Big)\Big) + 1\bigg)\]
which is integrable and whose integral equals $M(f_X(1)+1)$ for each $n$, which is also the integral of $\lim_{n\to\infty} F_n(s)$. Therefore, by the generalised dominated convergence theorem, the integral converges to
\[\int_0^1 \sqrt{R^*_X(f(s))f_X'(s)}\, ds\]
and the proof is complete.
\end{proof}

It is then a simple task to prove Proposition \ref{lowersemicts}, which shows that $\tilde K(f,0,t)$ can be bounded above by taking piecewise linear approximations to $f$.

\begin{proof}[Proof of Proposition \ref{lowersemicts}]
By \eqref{Kalt}, for any $f\in E^2$,
\begin{multline*}
\tilde K(f,0,t) = -\int_0^t R^*(f(s))ds + 2\sqrt2\int_0^t \sqrt{R^*_X(f(s))f'_X(s)}ds - f_X(t)\\
+ 2\sqrt2\int_0^t \sqrt{R^*_Y(f(s))f'_Y(s)}ds - f_Y(t).
\end{multline*}
It therefore suffices, by symmetry, to show that
\[\limsup_{n\to\infty}\int_0^t R^*(f_n(s))ds \le \int_0^t R^*(f(s))ds,\]
\[\liminf_{n\to\infty}\int_0^t \sqrt{R^*_X(f_n(s))f'_{n,X}(s)}ds \ge \int_0^t \sqrt{R^*_X(f(s))f'_X(s)}ds\]
and
\[\limsup_{n\to\infty} f_{n,X}(t) \le f_X(t).\]
The first of these statements follows from Lemma \ref{pointwiselem} and the continuity and boundedness of $R^*$ away from $0$, using the fact that $f$, and therefore $f_n$, is good. The second follows from Lemma \ref{liminfsqrt}. For the third, we observe that since $f$ is increasing and right-continuous,
\[f_{n,X}(t) \le f_X\big(\sfrac{\lceil nt\rceil}{n}\big) \to f_X(t),\]
which completes the proof.
\end{proof}

\subsection[Upper semi-continuity of $K$: proofs of Proposition \ref{uppersemicts} and Corollary \ref{uppersemictscor}]{Upper semi-continuity of $\tilde K$: proofs of Proposition \ref{uppersemicts} and Corollary \ref{uppersemictscor}}\label{usc_sec}





The following consequence of the Cauchy-Schwarz inequality is the key to proving Proposition \ref{uppersemicts}.

\begin{lem}\label{limsupsqrt}
Suppose that $0\le a<b\le 1$ and $f,f_n\in G_M^2$ for all $n$. If $f$ is differentiable on $[a,b]$, and $d(f_n,f)\to 0$, then
\[\limsup_{n\to\infty} \int_a^b \sqrt{R^*_X(f_n(s)) f_{n,X}'(s)} \d s \le \int_a^b \sqrt{R^*_X(f(s))f_X'(s)}\d s\]
where we write $f_{n,X}$ for the $x$-component of $f_n$.
\end{lem}

\begin{proof}
We carry out the proof when $a=0$ and $b=1$; the general case follows by including $\ind_{\{s\in[a,b]\}}$ throughout. By the Cauchy-Schwarz inequality, for any $m\in\N$,
\begin{align*}
\int_0^1 \sqrt{R^*_X(f_n(s)) f_{n,X}'(s)} \d s &= \sum_{i=1}^m \int_{(i-1)/m}^{i/m} \sqrt{R^*_X(f_n(s)) f_{n,X}'(s)} \d s\\
&\le \sum_{i=1}^m \bigg(\int_{(i-1)/m}^{i/m} R^*_X(f_n(s))\d s\bigg)^{1/2}\bigg(\int_{(i-1)/m}^{i/m} f'_{n,X}(s)\d s\bigg)^{1/2}\\
&\le \sum_{i=1}^m \bigg(\int_{(i-1)/m}^{i/m} R^*_X(f_n(s))\d s\bigg)^{1/2}\big(f_{n,X}\big(\sfrac{i}{m}\big) - f_{n,X}\big(\sfrac{i-1}{m}\big)\big)^{1/2}
\end{align*}
where the last inequality is not an equality since we do not know whether $f_{n,X}$ is absolutely continuous.
Since $f$ is continuous, by Lemma \ref{pointwiselem} we know that $f_n(s)\to f(s)$ for every $s$. Thus, using that $f,f_n\in G_M^2$ and $R^*_X$ is continuous away from $0$, by bounded convergence the right-hand side above converges to
\[\sum_{i=1}^m \bigg(\int_{(i-1)/m}^{i/m} R^*_X(f(s))\d s\bigg)^{1/2}\Big(f_{X}\Big(\frac{i}{m}\Big) - f_{X}\Big(\frac{i-1}{m}\Big)\Big)^{1/2}\]
%
which is at most
\begin{equation}\label{approxinvolvingm}
\sum_{i=1}^m \frac{1}{m} \bigg(\sup_{u\in[\frac{i-1}{m},\frac{i}{m}]} R^*_X(f(u))\bigg)^{1/2} \bigg( m\Big(f_X\Big(\frac{i}{m}\Big) - f_X\Big(\frac{i-1}{m}\Big)\Big)\bigg)^{1/2}.
\end{equation}
We claim that \eqref{approxinvolvingm} converges, as $m\to\infty$, to $\int_0^1 \sqrt{R^*_X(f(s))f'_X(s)}\d s$. To prove this we can follow almost exactly the same argument as in the proof of Lemma \ref{liminfsqrt}, writing \eqref{approxinvolvingm} in the form
\[\int_0^1 \sum_{i=1}^m \ind_{\{s\in[\frac{i-1}{m},\frac{i}{m})\}} \bigg(\sup_{u\in[\frac{i-1}{m},\frac{i}{m}]} R^*_X(f(u))\bigg)^{1/2} \bigg( m\Big(f_X\Big(\frac{i}{m}\Big) - f_X\Big(\frac{i-1}{m}\Big)\Big)\bigg)^{1/2} \d s\]
and applying the generalised dominated convergence theorem since the integrand evaluated at $s$ converges as $m\to\infty$ to $R^*_X(f(s))f'_X(s)$ for almost every $s\in[0,1]$, and can be bounded above by 
\[F_m(s) = \sum_{i=1}^m \ind_{\{s\in[\frac{i-1}{m},\frac{i}{m})\}} M \bigg(m\Big(f_X\Big(\frac{i}{m}\Big) - f_X\Big(\frac{i-1}{m}\Big)\Big) + 1\bigg).\]
This completes the proof.
%
\end{proof}

The next step is to extend the previous lemma to functions that are not necessarily continuous.

\begin{lem}\label{limsupsqrt2}
Suppose that $0\le a<b\le 1$ and $f,f_n\in G_M^2$ for all $n$. If $d(f_n,f)\to 0$, then
\[\limsup_{n\to\infty} \int_a^b \sqrt{R^*_X(f_n(s)) f_{n,X}'(s)} \d s \le \int_a^b \sqrt{R^*_X(f(s))f_X'(s)}\d s\]
where we write $f_{n,X}$ for the $x$-component of $f_n$.
\end{lem}

\begin{proof}
Fix $\eps\in(0,6M)$. Let $S\subset(0,1)$ be the set of points (in $(0,1)$) at which $f$ is not differentiable. Since $f$ is increasing, $S$ has zero Lebesgue measure, and can therefore be covered by a finite collection $(s_1^-,s_1^+), \ldots, (s_N^-, s_N^+)$ of open intervals whose total length is at most $\eps^2/M^3$. Let $S' = \bigcup_{i=1}^N (s_i^-,s_i^+)$. Then by Lemma \ref{limsupsqrt}, since $[a,b]\setminus S'$ is a finite union of closed intervals on which $f$ is absolutely continuous, we have
\[\limsup_{n\to\infty} \int_{[a,b]\setminus S'} \sqrt{R^*_X(f_n(s)) f_{n,X}'(s)} \d s \le \int_{[a,b]\setminus S'} \sqrt{R^*_X(f(s))f_X'(s)}\d s \le \int_a^b \sqrt{R^*_X(f(s))f_X'(s)}\d s.\]
It therefore suffices to show that
\begin{equation}\label{limsupsqrt2eq}
\limsup_{n\to\infty} \int_{[a,b]\cap S'} \sqrt{R^*_X(f_n(s)) f_{n,X}'(s)} \d s \le \eps.
\end{equation}
However, since $f_n\in G_M^2$, we have 
\[\int_{[a,b]\cap S'} \sqrt{R^*_X(f_n(s)) f_{n,X}'(s)} \d s \le \int_{[a,b]\cap S'} M\sqrt{f_{n,X}'(s)} \d s,\]
and by Jensen's inequality, this is at most
\[M\sqrt{\big|[a,b]\cap S'\big|} \bigg(\int_{[a,b]\cap S'} f_{n,X}'(s) \d s\bigg)^{1/2} \le M\sqrt{|S'|} (f_{n,X}(b)-f_{n,X}(a))^{1/2}\]
where $|S'|$ denotes the Lebesgue measure of $S'$. Since $f_n\in G_M^2$, this is at most $M^{3/2}\sqrt{|S'|}$, which is smaller than $\eps$ by construction. Thus \eqref{limsupsqrt2eq} holds and the proof is complete.
\end{proof}

The proof of Proposition \ref{uppersemicts} is now a simple consequence of the results above.

\begin{proof}[Proof of Proposition \ref{uppersemicts}]
We use the alternative form of $\tilde K$ mentioned in \eqref{Kalt}, i.e.
\begin{multline}
\tilde K(f,0,\theta) = -\int_0^\theta R^*(f(s)) ds + 2\sqrt2\int_0^\theta \sqrt{R^*_X(f(s))f'_X(s)} ds + 2\sqrt2\int_0^\theta \sqrt{R^*_Y(f(s))f'_Y(s)} ds\\
- f_X(\theta) - f_Y(\theta).
\end{multline}
Since either $f$ is continuous at $\theta$, or $\theta=1$, by Lemma \ref{pointwiselem} we have
\[f_{n,X}(\theta) + f_{n,Y}(\theta) \to  f_X(\theta) + f_Y(\theta).\]
Since $f$ is continuous almost everywhere, by Lemma \ref{pointwiselem} and the continuity of $R^*$ away from $0$ (using the fact that $f_n,f\in G_M^2$), we have
\[\int_0^\theta R^*(f_n(s)) ds \to \int_0^\theta R^*(f(s)) ds .\]
The result then follows from Lemma \ref{limsupsqrt2} and the symmetry between the $X$ and $Y$ components.
\end{proof}

Corollary \ref{uppersemictscor} follows easily from Proposition \ref{uppersemicts}.

\begin{proof}[Proof of Corollary \ref{uppersemictscor}]
For each $n\in\N$, take $f_n\in B_d(F,1/n)\cap G_M^2$ such that
\[\tilde K(f_n,0,1)\ge \sup_{f\in B_d(F,1/n)\cap G_M^2} \tilde K(f,0,1) - 1/n.\]
By Lemma \ref{coverGMT} we know that $G_{M,T}^2$ is totally bounded, and since $G_M^2\subset G_{M,T}^2$ and is closed, and $(E^2,d)$ is complete, we deduce that $G_M^2$ is compact under $d$. Therefore there exists a subsequence $(f_{n_j})_{j\ge1}$ such that $d(f_{n_j},f_\infty)\to 0$ as $j\to\infty$ for some $f_\infty\in G_M^2$. Since $d(f_{n_j},f_\infty)\to 0$, and $F$ is closed, we also have $f_\infty\in F$. By Proposition \ref{uppersemicts}
\[\limsup_{j\to\infty} \tilde K(f_{n_j},0,1)\le \tilde K(f_\infty,0,1).\]
Then by our choice of $f_n$,
\[\limsup_{j\to\infty} \hspace{-1mm} \sup_{f\in B_d(F,1/n_j)\cap G_M^2} \hspace{-2mm} \tilde K(f,0,1) \le \limsup_{j\to\infty} (\tilde K(f_{n_j},0,1) + 1/n_j) \le \tilde K(f_\infty,0,1) \le \sup_{f\in F \cap G_M^2} \hspace{-2mm} \tilde K(f,0,1)\]
which completes the proof.
\end{proof}

\bibliographystyle{plain}\def\cprime{$'$}

\end{document}